\documentclass[11 pt]{amsart}
\usepackage{fancyhdr}
\usepackage[margin=1.15in]{geometry}
\usepackage{amsthm}
\usepackage{amsmath}
\usepackage{amssymb}
\usepackage{enumerate}
\usepackage{graphicx}
\usepackage{cancel}
\usepackage{hyperref}

\usepackage{comment}
\usepackage{enumitem}
\usepackage{mathtools}
\usepackage{tikz}
\usetikzlibrary{matrix,arrows,decorations.pathmorphing}
\usepackage{mathabx}
\usetikzlibrary{matrix}
\usepackage{scrextend}
\usepackage{tkz-graph}
\usetikzlibrary{patterns}
\usepackage{tikz-cd}
\usepackage[color=blue!20!white]{todonotes}
\makeatletter
\providecommand\@dotsep{5}
\def\listtodoname{List of Todos}
\def\listoftodos{\@starttoc{tdo}\listtodoname}
\makeatother

\newtheorem {theorem}{Theorem}[section]
\newtheorem {lemma}[theorem]{Lemma}

\newtheorem {corollary}[theorem]{Corollary}

\newtheorem {definition}[theorem]{Definition}
\theoremstyle{remark}
\newtheorem {remark}[theorem]{Remark}
\newtheorem {example}[theorem]{Example}

\DeclareMathOperator{\gr}{grad}

\newcommand\Z{\mathbb{Z}}

\newcommand\Q{\mathbb{Q}}

\def\Hp{\mathcal{H}}

\newcommand \bJ {\bar{J}}
\newcommand \bh {\mkern3mu \overline{\mkern-3mu H \mkern-1mu} \mkern1mu}
\newcommand \bHp {\overline{\Hp}}

\newcommand\Ta{\mathbb{T}_\alpha}
\newcommand\Tb{\mathbb{T}_\beta}

\def\Sym{\mathrm{Sym}}

\def\spinc {{\operatorname{spin^c}}}

\def\fin\qedhere
\def\pr {{\text{pr}}}

\DeclareMathOperator{\Hom}{Hom}

\def\from {{\leftarrow}}

\def\s{\mathfrak s}

\newcommand\alphas{\boldsymbol\alpha}
\newcommand\betas{\boldsymbol\beta}

\def\bs{\bar{\mathfrak{\s}}}

\def\gr{\mathrm{gr}}

\def\ff {{\mathbb{F}}}

\def\fin\qedhere
\def\from {{\leftarrow}}

\def\ccdot {\! \cdot \!}

\newcommand{\bunderline}[1]{\underline{#1\mkern-2mu}\mkern2mu }

\def\du {\bar{d}}
\def\dl {\bunderline{d}}

\def\Hm{\mathbb{H}^-}

\def\CF {\mathit{CF}}
\def\HF {\mathit{HF}}

\newcommand\CFhat{\widehat{\CF}}

\newcommand \CFp {\CF^+}
\newcommand \CFm {\CF^-}
\newcommand \HFm {\HF^-}
\newcommand \CFinf {\CF^{\infty}}

\newcommand \CFo {\CF^{\circ}}
\newcommand \HFo {\HF^{\circ}}

\def\CFI {\mathit{CFI}}
\def\HFI {\mathit{HFI}}

\newcommand \CFIm {\CFI^-}
\newcommand \HFIm {\HFI^-}

\newcommand \CFIo {\CFI^{\circ}}
\newcommand \HFIo {\HFI^{\circ}}

\def\inv{\iota}

\def\Inv{\mathfrak{I}}

\newcommand{\co}{\nobreak\mskip2mu\mathpunct{}\nonscript
\mkern-\thinmuskip{:}\penalty300\mskip6muplus1mu\relax}

\def\rp{\mathbb{RP}}

\def\Vert{\operatorname{Vert}}

\def\co{\colon\thinspace}

\input xy
\xyoption{all}

\begin{document}

\title{On homology cobordism and local equivalence between plumbed manifolds}

\author[Irving Dai]{Irving Dai}
\author[Matthew Stoffregen]{Matthew Stoffregen}
\thanks{ID was partially supported by NSF grant DGE-1148900.}
\thanks{MS was partially supported by NSF grant DMS-1702532.}

\address{Department of Mathematics, Princeton University, Princeton, NJ 08540}
\email{idai@math.princeton.edu}

\address {Department of Mathematics, Massachusetts Institute of Technology, Cambridge, MA 02142}
\email{mstoff@mit.edu}

\begin{abstract}
We establish a structural understanding of the involutive Heegaard Floer homology for all linear combinations of almost-rational (AR) plumbed three-manifolds. We use this to show that the Neumann-Siebenmann invariant is a homology cobordism invariant for all linear combinations of AR plumbed homology spheres. As a corollary, we prove that if $Y$ is a linear combination of AR plumbed homology spheres with $\mu(Y) = 1$, then $Y$ is not torsion in the homology cobordism group. A general computation of the involutive Heegaard Floer correction terms for these spaces is also included.
\end{abstract}



\maketitle

\section{Introduction and Results}
\label{sec:1}
The aim of the present work is to investigate the involutive Heegaard Floer homology of a certain class of plumbed three-manifolds, as a means for understanding their span in the homology cobordism group. The application of Floer theory to the problem of homology cobordism is well-known and has an established history in the literature; see e.g.\ \cite{FSinstanton}, \cite{FurutaHom}, \cite{Froy}, \cite{AbsGraded}, \cite{Triangulations}.  In this paper, we build on the set-up of \cite{DM}, which uses the involutive Floer package developed by Hendricks-Manolescu \cite{HMinvolutive} and Hendricks-Manolescu-Zemke \cite{HMZ} to study the family of almost-rational plumbed three-manifolds. (We refer to these as AR plumbed manifolds, for short.) This family, which includes all Seifert fibered rational homology spheres with base orbifold $S^2$, is defined via placing certain combinatorial constraints on the usual three-dimensional plumbing construction; see Section~\ref{sec:2.2}. In \cite{DM}, the involutive Floer homology of the class of AR plumbed manifolds was computed and a start was made on understanding connected sums of such manifolds in a restricted range of cases. 

Our main result (see Theorem~\ref{thm:1.1} below) provides a structural understanding of the involutive Floer homology for all linear combinations of AR plumbed manifolds. We use this to derive a nontorsion result for all linear combinations $Y$ of AR plumbed homology spheres with Rokhlin invariant $\mu(Y) = 1$. We are also able to relate the Neumann-Siebenmann invariant to the involutive Floer homology (in the case of linear combinations of AR manifolds), and carry out a general computation of the involutive Floer correction terms $\du$ and $\dl$ for these spaces.

\subsection{Overview and Motivation} \label{sec:1.1} \hspace*{\fill} \\
\indent
Historically, the homology cobordism group $\Theta_{\mathbb{Z}}^3$ has occupied a central place in the development of three- and four-manifold topology. Several gauge-theoretic invariants have been used to derive results involving $\Theta_{\mathbb{Z}}^3$, including a proof that $\Theta_{\mathbb{Z}}^3$ has a $\mathbb{Z}^\infty$ subgroup (originally due to Furuta \cite{FurutaHom} and Fintushel-Stern \cite{FSinstanton}). In turn, the structure of $\Theta_{\mathbb{Z}}^3$ has turned out to be connected to several questions in classical topology, especially concerning the triangulability of high-dimensional topological manifolds. Following this program, in \cite{Triangulations} Manolescu disproved the triangulation conjecture by ruling out 2-torsion in $\Theta_{\mathbb{Z}}^3$ with Rohklin invariant one. This was established via the introduction of a new gauge-theoretic invariant, called Pin(2)-equivariant Seiberg-Witten Floer homology, along with a related suite of homology cobordism invariants $\alpha$, $\beta$, and $\gamma$. (See also work of Lin \cite{Lin}.) Despite these advances, many problems involving the structure of $\Theta_{\mathbb{Z}}^3$ remain open. One particularly obvious question to ask is whether $\Theta_{\mathbb{Z}}^3$ contains any torsion. In light of Manolescu's result, it is also natural to ask this in the more restricted setting of having Rokhlin invariant one, in which case there is no 2-torsion. For further results in this direction, see work of Saveliev \cite{Sav} and Lin-Ruberman-Saveliev \cite{LRS}.

Recently, work by the second author has involved understanding the Pin(2)-equivariant Seiberg-Witten Floer homology of Seifert fibered spaces (see \cite{Stoffregen}, \cite{Stoffregen2}).  These Floer homologies have explicit algebraic models which make them amenable to computation; and, in addition, one can attempt to identify classical invariants such as the Neumann-Siebenmann invariant (defined in \cite{Neu}, \cite{Sieb}) for such spaces in terms of their Floer homology (see \cite[Conjecture 4.1]{Triangulations}). Using the formulation of Pin(2)-equivariant monopole Floer homology by Lin \cite{Lin}, many of these results can be extended to the class of AR plumbed manifolds by utilizing combinatorial techniques inspired by the work of N\'emethi on lattice homology \cite{Nem}. (See work by the first author in \cite{Dai}.) Although this gives a good understanding of the Pin(2)-homology of individual AR manifolds, a general description of the Pin(2)-homology of connected sums is more difficult. See \cite{Stoffregen2} for results in this direction. In the present paper, we will study the subgroups of $\Theta^3_\Z$ generated by Seifert spaces and AR plumbed manifolds; these will be denoted by $\Theta_{SF}$ and $\Theta_{AR}$, respectively. 

On the symplectic geometry side, in \cite{HMinvolutive} Hendricks and Manolescu defined a new three-manifold invariant called involutive Heegaard Floer homology. This is a modification of the usual Heegaard Floer homology of Ozsv\'ath and Szab\'o, taking into account the conjugation action $\iota$ on the Heegaard Floer complex coming from interchanging the $\alpha$- and $\beta$-curves. More precisely, given a rational homology sphere $Y$ equipped with a self-conjugate $\spinc$-structure $\s$, one can associate to the pair $(Y, \s)$ an algebraic object called an $\inv$-complex, from which one constructs a well-defined three-manifold invariant $\HFIm(Y, \s)$.\footnote{Involutive Floer homology is defined for all three-manifolds $Y$, but in this paper we will only need the case where $Y$ is a rational homology sphere.} (There are analogous constructions for the other three flavors of Heegaard Floer homology.) This is a module over $\ff[U, Q]/(Q^2)$, where $\ff = \Z/2\Z$ and the degrees of $U$ and $Q$ are $-2$ and $-1$, respectively. We also have the involutive Floer correction terms $\du$ and $\dl$, which are the analogues of the $d$-invariant in Heegaard Floer homology. These are homology cobordism invariants, but are not additive under connected sum. See Section~\ref{sec:2.1} for a review of involutive Floer homology. 

In \cite{HMZ}, Hendricks, Manolescu, and Zemke constructed an abelian group $\Inv_{\Q}$, consisting of all possible $\inv$-complexes up to an algebraic equivalence relation called local equivalence.\footnote{In \cite{HMZ}, only the case of integer homology spheres was considered, but the construction is essentially the same. See Section~\ref{sec:2.1}.} This notion is modeled on the relation of homology cobordism, in the sense that if two (integer) homology spheres are homology cobordant, then their $\inv$-complexes are locally equivalent. The group operation on $\Inv_{\Q}$ is given by tensor product. If $Y$ is a rational homology sphere equipped with a self-conjugate $\spinc$-structure $\s$, then taking the local equivalence class of the (grading-shifted) $\inv$-complex of $(Y, \s)$ gives an element of $\Inv_{\Q}$, which we denote by $h(Y, \s)$:
\[
(Y, \s) \mapsto h(Y, \s) \in \Inv_{\Q}.
\]
In \cite{HMZ} it was shown that $h$ takes connected sums to tensor products, and hence that restricting to the case of integer homology spheres yields a homomorphism
\[
h\co \Theta_\mathbb{Z}^3 \rightarrow \Inv_{\Q}.
\]
This is the analogue in the involutive Floer setting of the chain local equivalence group $\mathfrak{CLE}$, defined by the second author in \cite{Stoffregen} for Pin(2)-homology. We can thus attempt to study $\Theta_\mathbb{Z}^3$ by understanding the structure of the group $\Inv_{\Q}$ and the image of the map $h$. One should think of $\Inv_{\Q}$ as capturing all of the information contained in the involutive Floer homology, from the point of view of homology cobordism. Note that we can further restrict $h$ to the subgroups $\Theta_{SF}$ and $\Theta_{AR}$ of $\Theta_{\mathbb{Z}}^3$. See Section~\ref{sec:2.1} for a more precise discussion of local equivalence and the construction of $\Inv_{\Q}$. 

As in the Pin(2)-case, one can use the lattice homology construction of Ozsv\'ath-Szab\'o \cite{Plumbed} and N\'emethi \cite{Nem} to determine the involutive Heegaard Floer homology of the class of AR manifolds. This was carried out by the first author and Manolescu in \cite{DM}. In this paper, we use the algebraic model developed in \cite{DM} to complete the analysis of $h(\Theta_{SF})$ and $h(\Theta_{AR})$. In addition to providing a structural understanding of the involutive Floer homology for all linear combinations of AR manifolds, this will allow us to derive some applications about the subgroups $\Theta_{SF}$ and $\Theta_{AR}$.

\subsection{Statement of Results} \label{sec:1.2} \hspace*{\fill} \\
\indent
Consider the Brieskorn homology spheres $\Sigma(p, 2p-1, 2p+1)$ for $p \geq 3$ odd. In \cite{Stoffregen2} it was shown that these are linearly independent in $\Theta_{\mathbb{Z}}^3$ by using the Manolescu correction terms $\alpha, \beta,$ and $\gamma$. An analogous argument was given in \cite{DM} using the involutive Floer homology. For convenience, we re-parameterize slightly and denote by $Y_i$ the local equivalence class
\[
Y_i = h(\Sigma(2i+1, 4i+1, 4i+3))
\]
for $i \geq 1$.\footnote{This is a slight abuse of notation from e.g.\ \cite{Stoffregen2}, \cite{DM}, where $Y_p$ is used to denote $\Sigma(p, 2p-1, 2p+1)$.} (Here, we suppress writing the $\spinc$-structure, since in the case of an integer homology sphere the $\spinc$-structure is unique.) A schematic picture of the $Y_i$ is given in Figure~\ref{fig:1}. Our main result is that these classes in fact form a basis for $h(\Theta_{AR})$:
\begin{theorem}
\label{thm:1.1}
We have
\[
h(\Theta_{SF}) = h(\Theta_{AR}) \cong \mathbb{Z}^\infty.
\]
An explicit basis may be given by (the image of) the Poincar\'e homology sphere, together with (the images of) the Brieskorn homology spheres $\Sigma(p, 2p-1, 2p+1)$ for $p \geq 3$ odd.  
\end{theorem}
\noindent
There is nothing particularly special about the manifolds $\Sigma(p, 2p-1, 2p+1)$, except that they happen to realize a convenient set of $\inv$-complexes. The inclusion of $\Sigma(2, 3, 5)$ should be thought of as accounting for the possibility of an overall grading shift. Indeed, we have the following generalization of Theorem~\ref{thm:1.1}:

\begin{theorem}
\label{thm:1.2}
Let $Y$ be a linear combination of almost-rational plumbed three-manifolds, and let $\s$ be a self-conjugate $\spinc$-structure on $Y$. Then $h(Y, \s)$ is equal to a linear combination of the $Y_i$, up to a grading shift by some $\Delta \in \mathbb{Q}:$
\[
h(Y, \s) = \left( \sum_{i} c_iY_i \right)[\Delta].
\]
This expression is unique, in the sense that any such expansion for $h(Y, \s)$ must have the same $c_i$ and the same $\Delta$. Moreover, 
\[
\Delta = 2\bar{\mu}(Y, \s),
\]
where $\bar{\mu}(Y, \s)$ is the Neumann-Siebenmann invariant of $(Y, \s)$. In addition, the class $h(Y, \s)$ can be represented (up to orientation reversal) by an individual almost-rational plumbed manifold only if:
\begin{enumerate}
\item We have $c_i \in \{-1, 0, 1\}$ for all $i$, and 
\item The nonzero coefficients $c_i$ alternate in sign. More precisely, for each $c_i \neq 0$, we require that the next nonzero coefficient $c_j$ (with $j > i$) be given by $c_j = -c_i$.
\end{enumerate}
If $Y$ is to have the usual orientation, then the last nonzero coefficient must be $+1$; for $Y$ to have reversed orientation, the last nonzero coefficient must be $-1$.
\end{theorem}
\noindent
Note that in order for $h(Y, \s)$ to be zero in $\Inv_{\Q}$, both the coefficients $c_i$ and the grading shift $\Delta$ must be zero. In further sections, we will describe a convenient way of thinking about $\inv$-complexes and explain how to explicitly decompose the $\inv$-complex of a connected sum of AR manifolds as a linear combination of the $Y_i$. 

\begin{figure}[h!]
\center
\includegraphics[scale=1]{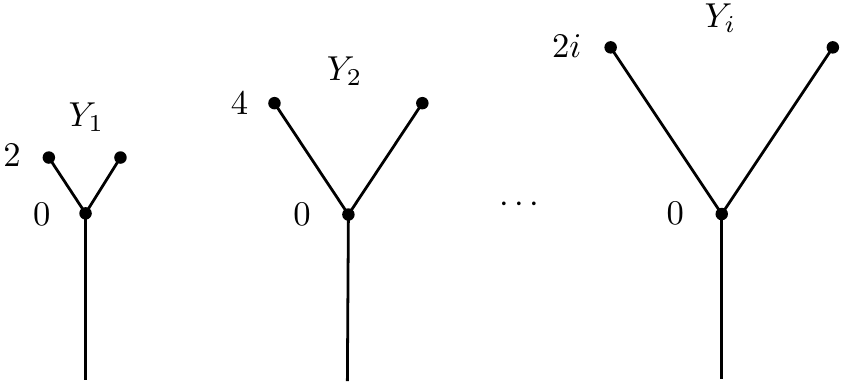}
\caption{Schematic picture of the $Y_i$. Here, $Y_i$ may be thought of as a particular submodule of the (grading-shifted) Heegaard Floer homology of $\Sigma(2i+1, 4i+1, 4i+3)$. See Section~\ref{sec:4} for details.}\label{fig:1}
\end{figure}

Theorem~\ref{thm:1.1} shows that the Neumann-Siebenmann invariant of $(Y, \s)$ can be read off from the local equivalence class of its involutive Floer homology, at least for linear combinations of AR manifolds. In particular, this implies that the Neumann-Siebenmann invariant restricted to $\Theta_{AR}$ is a homology cobordism invariant:

\begin{corollary}\label{cor:1.3}
The Neumann-Siebenmann invariant descends to a well-defined homomorphism $\bar{\mu}\co \Theta_{AR} \rightarrow \mathbb{Z}$.
\end{corollary}
\begin{proof}
It is easy to check from the definitions that $\bar{\mu}$ changes sign under orientation reversal and is additive under connected sum. The local equivalence class $h(Y)$ is an invariant of homology cobordism, so the claim follows from the uniqueness part of Theorem~\ref{thm:1.2}.
\end{proof}
\noindent
This provides a partial proof of a conjecture by Neumann \cite{Neu}, which states that the Neumann-Siebenmann invariant is a homology cobordism invariant in general. Corollary~\ref{cor:1.3} may be viewed as a strengthening of e.g.\ \cite[Theorem 1.3]{Stoffregen}, \cite[Theorem 1.3]{Dai}, \cite[Theorem 1.2]{DM}.

Corollary~\ref{cor:1.3} also implies that if $Y$ is a linear combination of AR homology spheres with $\bar{\mu}(Y) \neq 0$, then $Y$ is not torsion in the homology cobordism group. We can phrase this in terms of the Rokhlin invariant $\mu$ instead, which has the advantage of being defined for all three-manifolds:

\begin{corollary}\label{cor:1.4}
Let $Y$ be homology cobordant to a linear combination of AR homology spheres, so that $[Y] \in \Theta_{AR}$. If $\mu(Y) = 1$, then $[Y]$ is not torsion in $\Theta_{\mathbb{Z}}^3$.
\end{corollary}
\begin{proof}
The Neumann-Siebenmann invariant reduces to the Rokhlin invariant mod $2$.
\end{proof}
\noindent
More generally, if $Y$ is a manifold as above for which $h(Y)$ is nonzero, then Theorem~\ref{thm:1.1} evidently implies that the cobordism class $[Y]$ cannot be torsion in $\Theta_{\mathbb{Z}}^3$. Corollaries~\ref{cor:1.3} and \ref{cor:1.4} can be thought of as stating that if the Rokhlin invariant (or the Neumann-Siebenmann invariant) of a linear combination of AR manifolds is nontrivial, then $h(Y) \neq 0$.

In another direction, we can also use the decomposition of Theorem~\ref{thm:1.2} to carry out a computation of the involutive Floer correction terms $\du$ and $\dl$. Consider a linear combination in $\Inv_{\Q}$ of the form
\[
(Y_{s_1} + Y_{s_2} + \ldots + Y_{s_m}) - (Y_{t_1} + Y_{t_2} + \ldots + Y_{t_n}),
\]
where the $Y_i$ are the basis elements of Theorem~\ref{thm:1.2}. Without loss of generality, we assume that the $s_i$ are distinct from the $t_i$, and that $s_1 \geq s_2 \geq \cdots \geq s_m$ and $t_1 \geq t_2 \geq \cdots \geq t_n$. Then we have:

\begin{theorem}
\label{thm:1.5}
Let $Y$ be a linear combination of almost-rational plumbed three-manifolds, and let $\s$ be a self-conjugate $\spinc$-structure on $Y$. Suppose that $h(Y, \s)$ decomposes as a linear combination
\[
(Y_{s_1} + Y_{s_2} + \ldots + Y_{s_m}) - (Y_{t_1} + Y_{t_2} + \ldots + Y_{t_n})
\]
as above, shifted by some grading $\Delta$. Define the quantities
\begin{align*}
P_i &= 2\left(\sum_{j=1}^i t_j - \sum_{j=1}^i s_j\right) \text{ for } 0 \leq i \leq \min(m, n), \text{ and}\\
Q_i &= 2\left(\sum_{j=1}^i t_j - \sum_{j=1}^{i + 1} s_j \right) \text{ for } 0 \leq i \leq \min(m-1, n).
\end{align*}
Note that $P_0 = 0$ and $Q_0 = -2s_1$. Then
\begin{align*}
\dl(Y, \s) = d(Y, \s) + \max \{ &\min(P_0, Q_0), \\
&\min(P_0, P_1, Q_1), \\
&\ \ \ \ \ \ \ \ \ \ \ \ \ \vdots \\
&\min(P_0, P_1, P_2, \ldots, P_{\min(m,n)}, Q_{\min(m,n)})\},
\end{align*}
with the understanding that if $\min(m, n) = m$, the $Q_{\min(m, n)}$ in the last line should be deleted. \
\end{theorem}
\noindent
Note that if $m = 0$, then $\min(m, n) = m$, and the above expression reduces to
\[
\dl(Y, \s) = d(Y, \s) + \max \{ \min(P_0) \} = d(Y, \s). 
\]
Similarly, if $n = 0$ (but $m > 0$), then the above expression reduces to 
\[
\dl(Y, \s) = d(Y, \s) + \max \{ \min(P_0, Q_0) \} = d(Y, \s) - 2s_1.
\]
Compare with \cite[Corollary 1.4]{DM}.

Since orientation reversal corresponds to negation in $\Inv_{\Q}$, changing $Y$ to $-Y$ interchanges the index sets $\{s_i\}$ and $\{t_i\}$. Due to the fact that $\du(Y, \s) = -\dl(-Y, \s)$, Theorem~\ref{thm:1.5} can also be used to compute $\du(Y, \s)$. We will give some motivation for the terms appearing in the actual formula when we give the proof in Section~\ref{sec:5}.

Although the statement of Theorem~\ref{thm:1.5} is rather cumbersome, it is still possible to use it to derive some general facts about $\du$ and $\dl$. To this end, we have the following asymptotic characterization of the involutive Floer correction terms:
\begin{corollary}\label{cor:1.6}
Let $Y$ be a linear combination of almost-rational plumbed three-manifolds, and let $\s$ be a self-conjugate $\spinc$-structure on $Y$. Suppose that $h(Y, \s)$ is given by a linear combination
\[
(Y_{s_1} + Y_{s_2} + \ldots + Y_{s_m}) - (Y_{t_1} + Y_{t_2} + \ldots + Y_{t_n})
\]
as in the statement of Theorem~\ref{thm:1.5}. Then the involutive Floer correction terms of $\#_k(Y, \s)$ for $k$ sufficiently large are as follows. If $t_1 > s_1$, then for $k$ sufficiently large, we have
\begin{align*}
\du(\#_k(Y, \s)) &= k \cdot d(Y, \s) + 2t_1 \text{ and}\\
 \dl(\#_k(Y, \s)) &= k \cdot d(Y, \s).
\end{align*}
If $s_1 > t_1$, then for $k$ sufficiently large, we have
\begin{align*}
\du(\#_k(Y, \s)) &= k \cdot d(Y, \s) \text{ and}\\
 \dl(\#_k(Y, \s)) &= k \cdot d(Y, \s) - 2s_1.
\end{align*}
\end{corollary}
\noindent
We currently do not know whether a similar ``stabilization" result holds more generally for all rational homology spheres; if not, then Corollary~\ref{cor:1.6} provides an interesting obstruction to being a connected sum of AR plumbed manifolds. Theorem~\ref{thm:1.5} and Corollary~\ref{cor:1.6} should be compared with other expressions derived for various Floer correction terms in e.g.\ \cite[Theorem 1.4]{Stoffregen2}, \cite[Theorem 1.3]{DM}, \cite[Corollary 1.5]{DM}. It is possible (although actually rather involved) to show that Theorem~\ref{thm:1.5} reduces to \cite[Theorem 1.3]{DM} in the appropriate cases. 

We also have the following realization result:
\begin{corollary}\label{cor:1.7}
Let $d, \du,$ and $\dl$ be any triple of even integers such that $\dl \leq d \leq \du$ and $d, \du,$ and $\dl$ are not all equal to each other. In addition, let $\bar{\mu}$ be any integer. Then there are infinitely many distinct classes in $\Theta_{SF}$ with the invariants $d, \du, \dl,$ and $\bar{\mu}$.
\end{corollary}
\noindent
Given Theorem~\ref{thm:1.1}, Corollary~\ref{cor:1.7} is not particularly surprising, although it does establish that $\bar{\mu}$ is almost entirely independent from the involutive Floer correction terms. Note that if $\dl = d = \du$, then it is not hard to show the corresponding local equivalence class must be trivial (up to grading shift); in our case, this also implies $d =  -2\bar{\mu}$.

\medskip
\noindent {\bf Organization of the paper.} In Section~\ref{sec:2}, we review the construction of involutive Heegaard Floer homology and describe the set-up of \cite{DM} concerning the involutive Floer homology of AR plumbed manifolds. In Section~\ref{sec:3}, we establish the main technical result needed to prove Theorems~\ref{thm:1.1} and \ref{thm:1.2}, which we do in Section~\ref{sec:4}. In Section~\ref{sec:5}, we prove Theorem~\ref{thm:1.5} on the involutive Floer correction terms $\du$ and $\dl$. Finally, in Section~\ref{sec:6}, we give some examples and prove Corollaries~\ref{cor:1.6} and \ref{cor:1.7}. 

\medskip
\noindent {\bf Acknowledgements.} We would like to thank Jen Hom, Tye Lidman, Francesco Lin, and Christopher Scaduto for helpful conversations and suggestions, as well as Duncan McCoy for pointing out several interesting examples to us. We would also like to thank our advisors, Zolt\'an Szab\'o and Ciprian Manolescu, respectively, for their continued support and guidance.  We note that Theorem \ref{thm:1.1} has also been independently established by Hendricks-Hom-Lidman. 


\section{Preliminary Notions}
\label{sec:2}
In this section, we give the necessary background required for the rest of the paper. Much of the exposition here is taken from \cite{DM}, but since familiarity with the relevant constructions (and their notation) will be critical in future sections, we have included it as a matter of convenience to the reader. 

\subsection{Involutive Heegaard Floer Homology}\label{sec:2.1} \hspace*{\fill} \\
\indent
We begin by reviewing the construction of involutive Heegaard Floer homology as given in \cite{HMinvolutive}. We restrict ourselves to the case where $Y$ is a rational homology sphere and $\s$ is a self-conjugate $\spinc$-structure on $Y$. Let $\Hp = (H, J)$ be a Heegaard pair for $Y$, consisting of a pointed Heegaard splitting $H = (\Sigma, \alphas, \betas, z)$ of $Y$, together with a family of almost-complex structures $J$ on $\Sym^g(\Sigma)$. Associated to $\Hp$, we have the Heegaard Floer complex $\CFm(\Hp, \s)$, which is a $\Q$-graded, free $\ff[U]$-module generated by the intersection points $\Ta \cap \Tb$ in $\Sym^g(\Sigma)$. There are also three other variants of this complex, related by
\[
\CFhat = \CFm/(U = 0), \CFinf = U^{-1}\CFm, \text{ and } \CFp = \CFinf/\CFm.
\]
We denote these by $\CFo$ for $\circ = \widehat{\phantom{a}}, \infty, +, -$. If $\Hp$ and $\Hp'$ are two Heegaard pairs for $Y$ with the same basepoint $z$, then by work of Juh\'asz and Thurston \cite{Naturality}, any sequence of Heegaard moves relating $\Hp$ and $\Hp'$ defines a homotopy equivalence
\[
\Phi(\Hp, \Hp')\co \CFo(\Hp, \s) \rightarrow \CFo(\Hp', \s).
\]
This assignation is itself unique up to chain homotopy, in the sense that any two sequences of Heegaard moves define chain-homotopic maps $\Phi$; see also \cite[Proposition 2.3]{HMinvolutive}. This justifies the use of the notation $\CFo(Y, \s)$, rather than $\CFo(\Hp, \s)$. Taking the homology of $\CFo(Y, \s)$ yields the Heegaard Floer homology $\HFo(Y, \s)$ constructed by Ozsv\'ath and Szab\'o in \cite{HolDisk}, \cite{HolDiskTwo}.

Now consider the conjugate Heegaard pair $\bHp = (\bh, \bJ)$. This is defined by reversing the orientation on $\Sigma$ and interchanging the $\alpha$ and $\beta$ curves to give the Heegaard splitting
\[
\bh = (-\Sigma, \betas, \alphas, z),
\]
and taking the conjugate family $\bJ$ of almost-complex structures on $\Sym^g(-\Sigma)$. The points of $\Ta \cap \Tb$ are in obvious correspondence with the points of $\Tb \cap \Ta$, and $J$-holomorphic disks with boundary on $(\Ta, \Tb)$ are in bijection with $\bJ$-holomorphic disks with boundary on $(\Tb, \Ta)$. This yields a canonical isomorphism
\[
\eta\co \CFo(\Hp, \s) \rightarrow \CFo(\bHp, \s),
\]
where here we have used the fact that $\s = \bs$, since $\s$ is self-conjugate. Note that this is \textit{not} the map $\Phi(\Hp, \bHp)$ defined in the previous paragraph. Instead, defining
\[
\inv = \Phi(\bHp, \Hp) \circ \eta,
\]
we obtain a chain map from $\CFo(\Hp, \s)$ to itself. In \cite[Section 2.2]{HMinvolutive} it is shown that $\inv$ is a homotopy involution and is independent (up to the notion of equivalence defined later in this section) of the choice of $\Hp$. The involutive Heegaard Floer complex is then defined to be the mapping cone 
\begin{equation}\label{eqn:1}
\CFIo(Y, \s) = \left( \CFo(Y, \s) \xrightarrow{\phantom{o} Q (1+\inv) \phantom{o}} Q \ccdot \CFo(Y, \s) [-1] \right).
\end{equation}
Here, $Q$ is a formal variable marking the right-hand copy of $\CFo(Y, \s)$. Taking the homology of $\CFIo(Y, \s)$ yields the involutive Heegaard Floer homology $\HFIo(Y, \s)$. In \cite[Proposition 2.8]{HMinvolutive} it is shown that the quasi-isomorphism type of $\CFIo(Y, \s)$ is an invariant of $(Y, \s)$. In this paper, we will mainly deal with the minus version of the involutive Floer complex, $\CFIm$.

We formalize the algebra underlying $\CFIm$ by recalling \cite[Section 8]{HMZ}:

\begin{definition}\cite[Definition 8.1]{HMZ}\label{def:2.1}
An {\em $\inv$-complex} is a pair $(C, \inv)$, consisting of
\begin{itemize}
\item a $\Q$-graded, finitely generated, free chain complex $C$ over the ring $\ff[U]$, where $\operatorname{deg}(U)=-2$. Moreover, we ask that there is some $\tau \in \Q$ such that the complex $C$ is supported in degrees differing from $\tau$ by integers. We also require that there is a relatively graded isomorphism
\[
U^{-1}H_*(C) \cong \ff[U, U^{-1}],
\]
and that $U^{-1}H_*(C)$ is supported in degrees differing from $\tau$ by even integers;
\item a grading-preserving chain homomorphism $\inv \co C \to C$, such that $\inv^2$ is chain homotopic to the identity.
\end{itemize}
\end{definition}
\noindent
The set of $\inv$-complexes comes with a natural equivalence relation, as follows:

\begin{definition}\cite[Definition 8.3]{HMZ}\label{def:2.2}
Two $\inv$-complexes $(C, \inv)$ and $(C', \inv')$ are called {\em equivalent} if there exist chain homotopy equivalences
\[
F \co C \to C', \ \ G \co C' \to C
\]
that are homotopy inverses to each other, and such that 
\[
F \circ \inv \simeq \inv' \circ F,  \ \ \ G \circ \inv' \simeq \inv \circ G,
\]
where $\simeq$ denotes $\ff[U]$-equivariant chain homotopy.
\end{definition}
\noindent
Given $(Y, \s)$ as above, the pair $(\CFm(Y, \s), \inv)$ is evidently a $\inv$-complex. Note that if $Y$ is an integer homology sphere, then $\CFm(Y, \s)$ is $\mathbb{Z}$-graded, rather than $\mathbb{Q}$-graded (and we may take $\tau = 0$). It is shown in \cite{HMinvolutive} that choosing different Heegaard pairs for $Y$ yields equivalent $\inv$-complexes in the sense of Definition~\ref{def:2.2}. Hence we use the notation $(\CFm(Y, \s), \inv)$, rather than $(\CFm(\Hp, \s), \inv)$. Given an arbitrary $\inv$-complex, we can of course take the homology of its mapping cone as in (\ref{eqn:1}), which we refer to as the involutive homology of $(C, \inv)$. In \cite{HMinvolutive} it is shown that an equivalence of $\inv$-complexes induces a quasi-isomorphism between their mapping cones.

The following important structural result about involutive Floer homology was proven by Hendricks, Manolescu, and Zemke in \cite{HMZ}:

\begin{theorem}\cite[Theorem 1.1]{HMZ}
\label{thm:2.3}
Suppose $Y_1$ and $Y_2$ are rational homology spheres equipped with self-conjugate $\spinc$-structures $\s_1$ and $\s_2$. Let $\inv_1$, $\inv_2$ and $\inv$ denote the conjugation involutions on the Floer complexes $\CFm(Y_1, \s_1)$, $\CFm(Y_2,\s_2)$, and $\CFm(Y_1 \# Y_2, \s_1 \# \s_2)$, respectively. Then, the equivalence class of the $\inv$-complex $(\CFm(Y_1 \# Y_2), \inv)$ is the same as that of
\[
\bigl( \CFm(Y_1, \s_1) \otimes_{\ff[U]} \CFm(Y_2, \s_2) [-2], \ \inv_1 \otimes \inv_2 \bigr),
\]
where $[-2]$ denotes a grading shift.
\end{theorem}

One can also consider a coarser equivalence relation on the set of $\inv$-complexes, as discussed in \cite[Section 8]{HMZ}:

\begin{definition}\cite[Definition 8.5]{HMZ}\label{def:2.4}
Two $\inv$-complexes $(C, \inv)$ and $(C', \inv')$ are called {\em locally equivalent} if there exist (grading-preserving) chain maps
\[
F \co C \to C', \ \ G \co C' \to C
\]
such that 
\[
F \circ \inv \simeq \inv' \circ F,  \ \ \ G \circ \inv' \simeq \inv \circ G,
\]
and $F$ and $G$ induce isomorphisms on homology after inverting the action of $U$.
\end{definition}
\noindent
We call a map $F$ as above a \textit{local map} from $(C, \inv)$ to $(C', \inv')$, and similarly we refer to $G$ as a local map in the other direction. The relation of local equivalence is modeled on that of homology cobordism, in that if $Y_1$ and $Y_2$ are homology cobordant, then their respective $\inv$-complexes are locally equivalent. In \cite[Section 8]{HMZ}, it is shown that the set of $\inv$-complexes up to local equivalence forms a group, with the group operation being given by tensor product. We call this group the \textit{involutive Floer group} and denote it by $\Inv_{\Q}$:

\begin{definition}\cite[Proposition 8.8]{HMZ}\label{def:2.5}
Let $\Inv_{\Q}$ be the set of $\inv$-complexes up to local equivalence. This has a multiplication given by tensor product, which sends (the local equivalence classes of) two $\inv$-complexes $(C_1, \inv_1)$ and $(C_2, \inv_2)$ to (the local equivalence class of) their tensor product complex $(C_1 \otimes C_2, \inv_1 \otimes \inv_2)$. 
\end{definition}
\noindent
The identity element of $\Inv_{\Q}$ is given by the trivial complex consisting of a single $\ff[U]$-tower starting in grading zero, together with the identity map on this complex. Inverses in $\Inv_{\Q}$ are given by dualizing. There is an obvious subgroup of $\Inv_{\Q}$ generated by the set of $\inv$-complexes which are $\Z$-graded; we denote this by $\Inv$. Clearly, $\Inv_{\Q}$ consists of an infinite number of copies of $\Inv$, one for each $[\tau] \in \Q/2\Z$. See \cite[Section 8]{HMZ} for further discussion, and also \cite{Stoffregen} for the construction of the analogous group $\mathfrak{CLE}$ in Pin(2)-homology.  We will sometimes use additive notations for elements of $\Inv_{\Q}$.  

Let $h$ be the map sending a pair $(Y, \s)$ to the local equivalence class of its (grading-shifted) $\inv$-complex:
\[
h(Y, \s) = (\CFm(Y, \s), \iota)[-2].
\]
In light of Theorem~\ref{thm:2.3}, restricting to the case of integer homology spheres, we obtain a homomorphism 
\[
h\co \Theta_{\Z}^3 \rightarrow \Inv.
\]
We can further restrict $h$ to the subgroup of $\Theta_{\Z}^3$ generated by Seifert fibered spaces (respectively, AR plumbed homology spheres), which we denote by $\Theta_{SF}$ (respectively, $\Theta_{AR}$). As described in Section~\ref{sec:1}, in this paper we will use the algebraic set-up of \cite{DM} to complete the analysis of $h(\Theta_{SF})$ and $h(\Theta_{AR})$.

Finally, we review the definition of the involutive Floer correction terms $\du$ and $\dl$. These are given by
\[ 
\dl(Y,\s) = \max \{r \mid \exists \ x \in \HFIm_r(Y, \s), \forall \ n \geq 0, \ U^nx\neq 0 \ \text{and} \ U^nx \notin \operatorname{Im}(Q)\} + 1 
\]
and
\[ 
\du(Y,\s) = \max \{r \mid \exists \ x \in \HFIm_r(Y,\s),\forall \ n \geq 0, U^nx\neq 0; \exists \ m\geq 0 \ \operatorname{s.t.} \ U^m x \in \operatorname{Im}(Q)\} +2.
\]
See \cite[Lemma 2.9]{HMinvolutive}. (The shifts by one and two in these definitions are chosen so that $d=\dl=\du=0$ for $Y=S^3$.) Like the Ozsv\'ath-Szab\'o $d$-invariant, it is easily checked that the involutive Floer correction terms are invariant under homology cobordism, although it should be noted that they are not homomorphisms; see \cite[Theorem 1.3]{HMinvolutive}. More generally, we can define 
\[
\du, \dl\co \Inv_{\Q} \rightarrow \Q
\]
by taking the involutive homology of any $\inv$-complex and using a similar definition as before. In this paper, we will use a slightly different convention than in \cite[Section 8]{HMZ}, and subtract two from the above definitions of $\du$ and $\dl$ when defining them on $\Inv_{\Q}$. This is to cancel out the grading shift in the definition of $h$, so that
\begin{equation}\label{eqn:dudl}
\du(Y, \s) = \du(h(Y, \s)) \text{ and } \dl(Y, \s) = \dl(h(Y, \s)).
\end{equation}
Note that the conventions in \cite[Section 8]{HMZ} are such that $\du(Y, \s) = \du(\CFm(Y, \s), \inv)$ and $\dl(Y, \s) = \dl(\CFm(Y, \s), \inv)$.

\subsection{Almost-Rational Plumbed Manifolds}\label{sec:2.2}\hspace*{\fill} \\
\indent
We now define the class of AR plumbed manifolds. The motivation behind these spaces originally comes from the study of complex singularities (see e.g.\ \cite{Nem}). However, from the viewpoint of involutive Floer homology, it will be more convenient to think of AR manifolds simply as a class of three-manifolds whose Floer homology always takes a certain form; see Theorem~\ref{thm:2.7}. Thus the topological details of the construction will not actually be of much consequence in the rest of the paper, but we include them here for completeness.

Let $G$ be a decorated graph, and denote the decoration of a vertex $v$ in $G$ by $m(v) \in \Z$. Let $L$ be the integer lattice spanned by the vertices of $G$. We define a bilinear pairing on $L$ by setting
\[
\langle v, w \rangle = \begin{cases}
m(v) & \text{if $v=w$},\\
1 & \text{if $v$ and $w$ are connected by an edge, and}\\
0 & \text{otherwise,}
\end{cases}
\]
and extending linearly to $L$. We call this pairing the \textit{intersection form} associated to $G$. Denote the dual lattice of $L$ by $L' = \Hom(L, \Z)$, and define the \textit{canonical characteristic element} $K \in L'$ by setting
\[
K(v) = -m(v) - 2,
\]
for all vertices $v$. If $x$ is a vector in $L$, then we write $x > 0$ if all of the coefficients of $x$ are non-negative and if $x \neq 0$.

Now let $Y = Y(G)$ be the manifold defined by the usual three-dimensional plumbing construction on $G$. This means that $Y$ is the boundary of the oriented four-manifold $W = W(G)$ obtained by attaching two-handles to $B^4$ according to $G$. (Note that $Y$ is oriented as the boundary of $W$.) We further suppose that $G$ is a tree and that its intersection form is negative definite; this is equivalent to the claim that $Y$ is a rational homology sphere realized as the link of a normal surface singularity. Such a singularity is called \textit{rational} if its geometric genus is zero. According to a theorem of Artin (see \cite{Artin1}, \cite{Artin2}), this occurs if and only if
\[
-(K(x) + \langle x, x \rangle)/2 \geq 1,
\]
for all $x > 0$ in $L$. A graph satisfying this property is called a \textit{rational plumbing graph}. See \cite[Section 8]{NemethiOS} for further discussion.

\begin{definition}\cite[Definition 8.1]{NemethiOS}\label{def:2.6}
Let $G$ be a tree with negative definite intersection form. We say that $G$ is an {\em almost-rational graph} if there exists a vertex $v$ of $G$ such that by replacing the decoration $m(v)$ with some other $m' \leq m(v)$, we obtain a rational plumbing graph $G'$. In this situation, we refer to $Y=Y(G)$ as an {\em almost-rational plumbed manifold}.
\end{definition}

The family of AR plumbed manifolds includes all Seifert fibered rational homology spheres with base orbifold $S^2$, and, more generally, all one-bad-vertex plumbed manifolds in the sense of Ozsv\'ath and Szab\'o \cite{Plumbed}. In the case that $Y$ is also an (integer) homology sphere, we will refer to $Y$ as an AR plumbed homology sphere. Note that the class of AR manifolds is not closed under connected sum or orientation reversal, although in the latter case the Floer homology is still easy to understand via dualizing. Usually, we will refer to a connected sum of AR manifolds (some with reversed orientation) as a linear combination of AR manifolds.

\begin{remark}
All Seifert fibered rational homology spheres have base orbifold with underlying space $S^2$ or $\rp^2$. (Further, if they are integer homology spheres, this has to be $S^2$.) When the underlying space is $\rp^2$, then the Seifert fibered rational homology sphere is an L-space by \cite[Proposition 18]{BoyerGordonWatson}. In such cases, it is straightforward to show that the involutive Heegaard Floer homology is determined by the ordinary Heegaard Floer homology (and is uninteresting).
\end{remark}

\subsection{Graded Roots and Standard Complexes}\label{sec:2.3}\hspace*{\fill} \\
\indent
We now turn to a discussion of the involutive Floer homology of AR manifolds. For those unfamiliar with the results of \cite{DM} and related papers, a brief outline of the ideas is as follows. Let $Y$ be an AR plumbed three-manifold and let $\s$ be a $\spinc$-structure on $Y$. Then it is a consequence of the isomorphism between lattice homology and Heegaard Floer homology (see e.g.\ \cite{Plumbed}, \cite{NemethiOS}, \cite{Nem}) that the Heegaard Floer homology $\HFm(Y, \s)$ may be compactly expressed in the form of a combinatorial object called a \textit{graded root}. If $\s$ is self-conjugate, then the graded root $R$ associated to $(Y, \s)$ encodes not only the Heegaard Floer homology, but also the homological action of $\inv_*$ on $\HFm(Y, \s)$ (see \cite{DM}, \cite{Dai}). One of the main results of \cite{DM} is that the $\inv$-complex $(\CFm(Y, \s), \inv)$ can be recovered from $R$ together with this homological involution.

More precisely, associated to $R$ we may define a chain complex $C_*(R)$, called the \textit{standard complex of} $R$, together with an involution $J_0$ on $C_*(R)$ such that $(\CFm(Y, \s), \inv)$ and $(C_*(R), J_0)$ are equivalent $\inv$-complexes. In general, of course, the homotopy type of $\CFIm(Y, s)$ cannot be recovered from $\HFm(Y, \s)$, even with the knowledge of $\inv_*$. However, for AR manifolds the Heegaard Floer homology has certain structural constraints (coming from lattice homology) which imply that in such cases the equivalence class of $(\CFm(Y, \s), \inv)$ is determined for purely algebraic reasons. The standard complex $(C_*(R), J_0)$ has the advantage of being a particularly simple model for $(\CFm(Y, \s), \inv)$, and has a pleasing geometric/combinatorial interpretation which makes understanding and manipulating such complexes easier. (Indeed, the remainder of \cite{DM} is devoted to using this geometric picture to perform computations of $\du$ and $\dl$ for connected sums of AR manifolds.) In this paper, we will similarly manipulate these complexes to prove Theorems~\ref{thm:1.1} and \ref{thm:1.2}.

We now give the precise definition of a graded root. Let $R$ be an infinite tree, and let $\gr\co \Vert(R) \rightarrow \Q$ be a grading function from the vertices of $R$ to a coset of $2\mathbb{Z}$ in $\Q$. We say that $R$ (together with the grading function $\gr$) is a \textit{graded root} if the following conditions are satisfied:
\begin{itemize}
\item $\gr(u) - \gr(v) = \pm 2$ for any edge $(u, v)$,
\item $\gr(u) < \max\{\gr(v), \gr(w)\}$ for any edges $(u,v)$ and $(u,w)$ with $v \neq w$,
\item $\gr$ is bounded above, 
\item $\gr^{-1}(k)$ is finite for any $k \in \Q$, and
\item $\# \gr^{-1}(k)=1$ for all $k \ll 0$ in the image of $\gr$.
\end{itemize}
Note that this definition differs from the one given in \cite[Section 2.3]{DM} by a factor of two in the grading function. We think of the grading of a vertex $v$ as corresponding to its height, so that geometrically a graded root may be realized as an upwards-opening tree with an infinite downwards stem, as pictured in Figure~\ref{fig:2}. See \cite{NemethiGRS} for further discussion.

Given a graded root $R$, we define a graded $\ff[U]$-module $\Hm(R)$ by giving $\Hm(R)$ one generator for each vertex $v$ of $R$, located in the same grading as $\gr(v)$. (These are generators of $\Hm(R)$ over $\ff$.) We define the $U$-action by specifying $$U \cdot v = w \text{ if $(v, w)$ is an edge and } \gr(v) - \gr(w) = 2.$$ Clearly, $U$ has degree $-2$. Since $R$ and $\Hm(R)$ evidently contain precisely the same information, we will sometimes abuse notation and refer to them interchangeably. As alluded to earlier, the following theorem is a consequence of the isomorphism between lattice homology and Heegaard Floer homology:

\begin{theorem}\cite[Theorem 5.2.2]{Nem}
\label{thm:2.7}
Let $Y$ be an almost-rational plumbed three-manifold, and let $\s$ be a $\spinc$-structure on $Y$. Then $\HFm(Y, \s)$ is isomorphic to $\Hm(R)$ (as a graded $\ff[U]$-module) for some graded root $R$.
\end{theorem}
\noindent
See also \cite[Theorem 1.2]{Plumbed}, \cite[Theorem 8.3]{NemethiOS}.

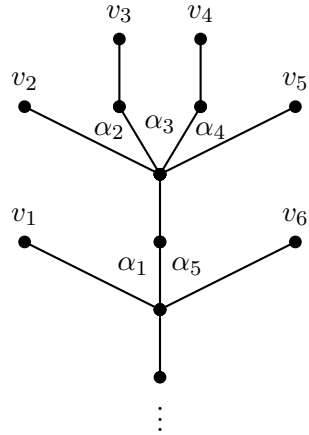
\begin{figure}[h!]
\begin{tikzpicture}[thick,scale=0.9]%
	\node[text width=0.1cm] at (0, -3.5) {\vdots};
	\draw (0, 0) node[circle, draw, fill=black!100, inner sep=0pt, minimum width=4pt] {} -- (0, -1) node[circle, draw, fill=black!100, inner sep=0pt, minimum width=4pt] {};
	\draw (0, -1) node[circle, draw, fill=black!100, inner sep=0pt, minimum width=4pt] {} -- (0, -2) node[circle, draw, fill=black!100, inner sep=0pt, minimum width=4pt] {};
	\draw (0, -2) node[circle, draw, fill=black!100, inner sep=0pt, minimum width=4pt] {} -- (0, -3) node[circle, draw, fill=black!100, inner sep=0pt, minimum width=4pt] {};
	
	\draw (0, 0) node[circle, draw, fill=black!100, inner sep=0pt, minimum width=4pt] {} -- (-0.6, 1) node[circle, draw, fill=black!100, inner sep=0pt, minimum width=4pt] {};
	\draw (0, 0) node[circle, draw, fill=black!100, inner sep=0pt, minimum width=4pt] {} -- (0.6, 1) node[circle, draw, fill=black!100, inner sep=0pt, minimum width=4pt] {};
	
	\draw (-0.6, 1) node[circle, draw, fill=black!100, inner sep=0pt, minimum width=4pt] {} -- (-0.6, 2) node[circle, draw, fill=black!100, inner sep=0pt, minimum width=4pt, label = $v_3$] {};
	\draw (0.6, 1) node[circle, draw, fill=black!100, inner sep=0pt, minimum width=4pt] {} -- (0.6, 2) node[circle, draw, fill=black!100, inner sep=0pt, minimum width=4pt, label = $v_4$] {};	
	
	\draw (0, 0) node[circle, draw, fill=black!100, inner sep=0pt, minimum width=4pt] {} -- (-2, 1) node[circle, draw, fill=black!100, inner sep=0pt, minimum width=4pt, label = $v_2$] {};
	\draw (0, 0) node[circle, draw, fill=black!100, inner sep=0pt, minimum width=4pt] {} -- (2, 1) node[circle, draw, fill=black!100, inner sep=0pt, minimum width=4pt, label = $v_5$] {};
	
	\draw (0, -2) node[circle, draw, fill=black!100, inner sep=0pt, minimum width=4pt] {} -- (-2, -1) node[circle, draw, fill=black!100, inner sep=0pt, minimum width=4pt, label = $v_1$] {};
	\draw (0, -2) node[circle, draw, fill=black!100, inner sep=0pt, minimum width=4pt] {} -- (2, -1) node[circle, draw, fill=black!100, inner sep=0pt, minimum width=4pt, label = $v_6$] {};
	\node[label = $\alpha_3$] at (0,0.3) {};
	\node[label = $\alpha_2$] at (-0.75,0.2) {};
	\node[label = $\alpha_4$] at (0.75,0.2) {};
	\node[label = $\alpha_1$] at (-0.4,-1.8) {};
	\node[label = $\alpha_5$] at (0.4,-1.8) {};
\end{tikzpicture}
\caption{Graded root (with leaf and angle labels) for $\HFm(\Sigma(2, 7, 15))$. \hspace{\textwidth} Uppermost vertices have Maslov grading $-2$; successive rows have grading difference two. (Taken from \cite{DM}.)} \label{fig:2}
\end{figure}

In our case, we will be considering graded roots which have an additional symmetry corresponding to the homological action of $\inv_*$ on $\HFm(Y, s)$. To this end, we define a \textit{symmetric graded root} to be a graded root $R$ with an involution $J$ satisfying the following properties:
\begin{itemize}
\item $ \gr(v) = \gr(Jv)$ for any vertex $v$,
\item $(v, w)$ is an edge in $R$ if and only if $(Jv, Jw)$ is an edge in $R$, and
\item for every $k \in \Q$, there is at most one $J$-invariant vertex $v$ with $\gr(v) = k$.
\end{itemize} 
See also \cite[Definition 2.11]{DM}. Geometrically, a symmetric graded root is simply a graded root which is symmetric under reflection about the obvious central vertical axis. It is clear that the action of $J$ on $R$ may be viewed as an involution on $\Hm(R)$ commuting with the action of $U$. As described in the beginning of the section, if $\s$ is self-conjugate, then $\HFm(Y, \s)$ is in fact isomorphic to a symmetric graded root $R$, and this isomorphism takes the action of $\inv_*$ to the involution on $R$. See \cite[Theorem 3.1]{DM}; also \cite[Lemma 2.5]{Dai} for the corresponding statement in the monopole category.

Now suppose that $R$ is a symmetric graded root. Following \cite[Section 4]{DM}, we construct a free $\ff[U]$-complex whose homology is $\Hm(R)$, as follows. Let $v_1, v_2, \ldots, v_n$ be the leaves of $R$, enumerated in left-to-right lexicographic order. We also enumerate by $\alpha_1, \alpha_2, \ldots, \alpha_{n-1}$ the $n-1$ upward-opening angles in left-to-right lexicographic order. The appropriate labels for the case of $\HFm(\Sigma(2, 7, 15))$ are shown above in Figure~\ref{fig:2}. We denote by $\gr(v_i)$ the grading of vertex $v_i$ and by $\gr(\alpha_i)$ the grading of the vertex supporting angle $\alpha_i$. 

The generators (over $\ff[U]$) of our complex are given as follows. For each leaf $v_i$, we place a single generator in grading $\gr(v_i)$, which by abuse of notation we also denote by $v_i$. Note that $v_i$ is a generator of our complex over $\ff[U]$, so that as an abelian group we are introducing an entire tower of generators $\ff[U]v_i$. For each angle $\alpha_i$, we similarly place a single generator in grading $\gr(\alpha_i) + 1$, denoting this by $\alpha_i$. (The gradings of these generators are motivated by the notion of a \textit{geometric complex}, to appear in Section \ref{sec:2.4}.) We define our differential to be identically zero on the $v_i$, and set 
\[
\partial \alpha_i = U^{(\gr(v_i)-\gr(\alpha_i))/2}v_i + U^{(\gr(v_{i+1})-\gr(\alpha_i))/2}v_{i+1}
\]
on the $\alpha_i$, extending to the entire complex linearly and $U$-equivariantly. We call this complex the \textit{standard complex of $R$} and denote it by $C_*(R)$. The standard complex associated to the graded root given in Figure~\ref{fig:2} is shown in Figure~\ref{fig:3}. 

\begin{figure}[h!]
\begin{tikzpicture}[thick,scale=0.9]%

	
	\draw (-0.7, 0) node[circle, draw, fill=black!100, inner sep=0pt, minimum width=4pt, label = $v_3$] {} -- (-0.7, -1) node[circle, draw, fill=black!100, inner sep=0pt, minimum width=4pt] {};
	\draw (-0.7, -1) node[circle, draw, fill=black!100, inner sep=0pt, minimum width=4pt] {} -- (-0.7, -2) node[circle, draw, fill=black!100, inner sep=0pt, minimum width=4pt] {};
	\draw (-0.7, -2) node[circle, draw, fill=black!100, inner sep=0pt, minimum width=4pt] {} -- (-0.7, -3) node[circle, draw, fill=black!100, inner sep=0pt, minimum width=4pt] {};
	\draw (-0.7, -3) node[circle, draw, fill=black!100, inner sep=0pt, minimum width=4pt] {} -- (-0.7, -4) node[circle, draw, fill=black!100, inner sep=0pt, minimum width=4pt] {};
	\node[text width=0.1cm] at (-0.7, -4.5) {\vdots};
	
	\draw (0.7, 0) node[circle, draw, fill=black!100, inner sep=0pt, minimum width=4pt, label = $v_4$] {} -- (0.7, -1) node[circle, draw, fill=black!100, inner sep=0pt, minimum width=4pt] {};
	\draw (0.7, -1) node[circle, draw, fill=black!100, inner sep=0pt, minimum width=4pt] {} -- (0.7, -2) node[circle, draw, fill=black!100, inner sep=0pt, minimum width=4pt] {};
	\draw (0.7, -2) node[circle, draw, fill=black!100, inner sep=0pt, minimum width=4pt] {} -- (0.7, -3) node[circle, draw, fill=black!100, inner sep=0pt, minimum width=4pt] {};
	\draw (0.7, -3) node[circle, draw, fill=black!100, inner sep=0pt, minimum width=4pt] {} -- (0.7, -4) node[circle, draw, fill=black!100, inner sep=0pt, minimum width=4pt] {};
	\node[text width=0.1cm] at (0.7, -4.5) {\vdots};
	
	\draw (-2.1, -1) node[circle, draw, fill=black!100, inner sep=0pt, minimum width=4pt, label = $v_2$] {} -- (-2.1, -2) node[circle, draw, fill=black!100, inner sep=0pt, minimum width=4pt] {};
	\draw (-2.1, -2) node[circle, draw, fill=black!100, inner sep=0pt, minimum width=4pt] {} -- (-2.1, -3) node[circle, draw, fill=black!100, inner sep=0pt, minimum width=4pt] {};
	\draw (-2.1, -3) node[circle, draw, fill=black!100, inner sep=0pt, minimum width=4pt] {} -- (-2.1, -4) node[circle, draw, fill=black!100, inner sep=0pt, minimum width=4pt] {};
	\node[text width=0.1cm] at (-2.1, -4.5) {\vdots};
	
	\draw (2.1, -1) node[circle, draw, fill=black!100, inner sep=0pt, minimum width=4pt, label = $v_5$] {} -- (2.1, -2) node[circle, draw, fill=black!100, inner sep=0pt, minimum width=4pt] {};
	\draw (2.1, -2) node[circle, draw, fill=black!100, inner sep=0pt, minimum width=4pt] {} -- (2.1, -3) node[circle, draw, fill=black!100, inner sep=0pt, minimum width=4pt] {};
	\draw (2.1, -3) node[circle, draw, fill=black!100, inner sep=0pt, minimum width=4pt] {} -- (2.1, -4) node[circle, draw, fill=black!100, inner sep=0pt, minimum width=4pt] {};
	\node[text width=0.1cm] at (2.1, -4.5) {\vdots};
	
	\draw (-3.5, -3) node[circle, draw, fill=black!100, inner sep=0pt, minimum width=4pt, label = $v_1$] {} -- (-3.5, -4) node[circle, draw, fill=black!100, inner sep=0pt, minimum width=4pt] {};
	\node[text width=0.1cm] at (-3.5, -4.5) {\vdots};
	
	\draw (3.5, -3) node[circle, draw, fill=black!100, inner sep=0pt, minimum width=4pt, label = $v_6$] {} -- (3.5, -4) node[circle, draw, fill=black!100, inner sep=0pt, minimum width=4pt] {};
	\node[text width=0.1cm] at (3.5, -4.5) {\vdots};
	

	\draw (0, -1.5) node[circle, draw, fill=black!100, inner sep=0pt, minimum width=4pt, label = $\alpha_3$] {} -- (0, -2.5) node[circle, draw, fill=black!100, inner sep=0pt, minimum width=4pt] {};
	\draw (0, -2.5) node[circle, draw, fill=black!100, inner sep=0pt, minimum width=4pt] {} -- (0, -3.5) node[circle, draw, fill=black!100, inner sep=0pt, minimum width=4pt] {};
	\node[text width=0.1cm] at (0, -4) {\vdots};	
	
	\draw (0, -1.5)  -- (-0.7, -2) [dashed];
	\draw (0, -1.5)  -- (0.7, -2) [dashed];
	\draw (0, -2.5)  -- (-0.7, -3) [dashed];
	\draw (0, -2.5)  -- (0.7, -3) [dashed];
	\draw (0, -3.5)  -- (-0.7, -4) [dashed];
	\draw (0, -3.5)  -- (0.7, -4) [dashed];

	\draw (-1.4, -1.5) node[circle, draw, fill=black!100, inner sep=0pt, minimum width=4pt, label = $\alpha_2$] {} -- (-1.4, -2.5) node[circle, draw, fill=black!100, inner sep=0pt, minimum width=4pt] {};
	\draw (-1.4, -2.5) node[circle, draw, fill=black!100, inner sep=0pt, minimum width=4pt] {} -- (-1.4, -3.5) node[circle, draw, fill=black!100, inner sep=0pt, minimum width=4pt] {};
	\node[text width=0.1cm] at (-1.4, -4) {\vdots};	
	
	\draw (-1.4, -1.5)  -- (-2.1, -2) [dashed];
	\draw (-1.4, -1.5)  -- (-0.7, -2) [dashed];
	\draw (-1.4, -2.5)  -- (-2.1, -3) [dashed];
	\draw (-1.4, -2.5)  -- (-0.7, -3) [dashed];
	\draw (-1.4, -3.5)  -- (-2.1, -4) [dashed];
	\draw (-1.4, -3.5)  -- (-0.7, -4) [dashed];

		\draw (1.4, -1.5) node[circle, draw, fill=black!100, inner sep=0pt, minimum width=4pt, label = $\alpha_4$] {} -- (1.4, -2.5) node[circle, draw, fill=black!100, inner sep=0pt, minimum width=4pt] {};
	\draw (1.4, -2.5) node[circle, draw, fill=black!100, inner sep=0pt, minimum width=4pt] {} -- (1.4, -3.5) node[circle, draw, fill=black!100, inner sep=0pt, minimum width=4pt] {};
	\node[text width=0.1cm] at (1.4, -4) {\vdots};	
	
	\draw (1.4, -1.5)  -- (2.1, -2) [dashed];
	\draw (1.4, -1.5)  -- (0.7, -2) [dashed];
	\draw (1.4, -2.5)  -- (2.1, -3) [dashed];
	\draw (1.4, -2.5)  -- (0.7, -3) [dashed];
	\draw (1.4, -3.5)  -- (2.1, -4) [dashed];
	\draw (1.4, -3.5)  -- (0.7, -4) [dashed];

	\draw (-2.8, -3.5) node[circle, draw, fill=black!100, inner sep=0pt, minimum width=4pt, label = $\alpha_1$] {};
	\node[text width=0.1cm] at (-2.8, -4) {\vdots};
	\draw (-2.8, -3.5)  -- (-2.1, -4) [dashed];
	\draw (-2.8, -3.5)  -- (-3.5, -4) [dashed];

	\draw (2.8, -3.5) node[circle, draw, fill=black!100, inner sep=0pt, minimum width=4pt, label = $\alpha_5$] {};
	\node[text width=0.1cm] at (2.8, -4) {\vdots};
	\draw (2.8, -3.5)  -- (2.1, -4) [dashed];
	\draw (2.8, -3.5)  -- (3.5, -4) [dashed];

\end{tikzpicture}
\caption{Standard complex associated to the $\ff[U]$-module $\HFm(\Sigma(2, 7, 15))$. \hspace{\textwidth}Solid lines represent the action of $U$; dashed lines represent the action of $\partial$. (Taken from \cite{DM}.)} \label{fig:3}
\label{figex}
\end{figure}
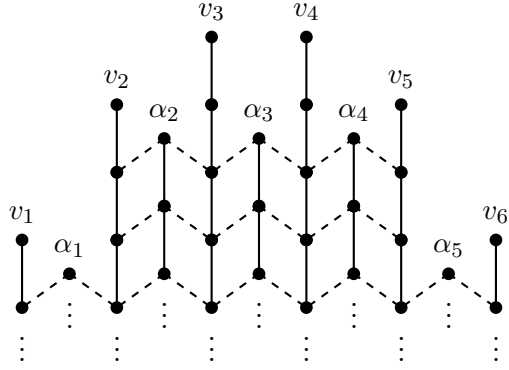

It is easy to check that the homology of $C_*(R)$ is isomorphic to $\Hm(R)$; see \cite[Lemma 4.1]{DM}. There is an obvious involution $J_0$ on $C_*(R)$ given by sending $v_i$ to $v_{n-i+1}$, $\alpha_i$ to $\alpha_{n-i}$, and extending linearly and $U$-equivariantly. This induces the involution on $\Hm(R)$ corresponding to the symmetry of $R$. Given a symmetric chain complex $(C, J_0)$ of this form, we define the \textit{absolute value} of the index of a generator to be $\min(i, n-i+1)$ for leaf generators $v_i$ and $\min(i, n-i)$ for angle generators $\alpha_i$. Where no confusion is possible, we denote this by $|i|$.

As discussed in the beginning of the subsection, we now have:

\begin{theorem}\cite[Theorem 4.5]{DM}
\label{thm:2.8}
Let $Y$ be an almost-rational plumbed three-manifold, and let $\s$ be a self-conjugate $\spinc$-structure on $Y$. Let $R$ be the symmetric graded root corresponding to $\HFm(Y, \s)$ afforded by the isomorphism between lattice homology and Heegaard Floer homology. Then $(\CFm(Y, \s), \inv)$ and $(C_*(R), J_0)$ are equivalent $\inv$-complexes.
\end{theorem}
\noindent
We will often refer to a symmetric graded root $R$, its standard complex $C_*(R)$, and its associated $\inv$-complex $(C_*(R), J_0)$ interchangeably. 

\subsection{Geometric Complexes}\label{sec:2.4}\hspace*{\fill} \\
\indent
Given Theorems~\ref{thm:2.3} and \ref{thm:2.8}, it is clear that to understand the image of $\Theta_{AR}$ in $\Inv$, we must understand tensor products of complexes of the form $C_*(R)$ and their duals. To this end, we describe a simple geometric/combinatorial representation of $C_*(R)$ which will be useful in later sections. For each generator $v_i$ of $C_*(R)$, we draw a single 0-cell as in Figure~\ref{fig:4}, placing these on a horizontal line from left-to-right in the same order as they appear in $C_*(R)$. For each generator $\alpha_i$ of $C_*(R)$, we then draw a 1-cell connecting the two 0-cells corresponding to the terms in $\partial \alpha_i$. We also remember the gradings $\gr(v_i)$ and $\gr(\alpha_i)$ of the generators associated to each of these cells. The generators of $C_*(R)$ are thus identified with the cellular generators of this (rather trivial) cell complex, tensored with $\ff[U]$. Under this correspondence, we see that the differential on $C_*(R)$ is the usual cellular differential, except twisted by powers of $U$ according to the gradings of the relevant generators:
\begin{equation}\label{eqn:2}
\partial (\square_e) = \sum _{\square_{e-1} \in \text{ bdry} (\square_e)} U^{(\gr(\square_{e-1})-\gr(\square_e))/2}\square_{e-1}.
\end{equation}
Here, $\square_e$ is a $e$-cell, and we sum over the $(e-1)$-cells appearing in the usual cellular boundary of $\square_e$. We have written $\gr(\square_e)$ to mean either $\gr(v_i)$ or $\gr(\alpha_i)$, depending on whether $d$ is zero or one. The twisted differential should be thought of exactly as the usual cellular differential, only multiplied by the necessary powers of $U$ so as to be of degree $-1$ with respect to the grading coming from $C_*(R)$. We call this cell complex (over $\ff[U]$) with its twisted differential (extended linearly and $U$-equivariantly) the \textit{geometric realization} of $C_*(R)$. See Figure~\ref{fig:4}.

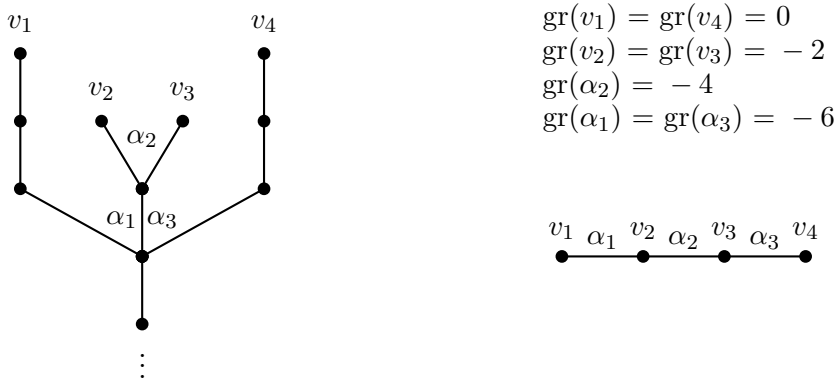
\begin{figure}[h!]
\begin{tikzpicture}[thick,scale=0.9]%

	\node[text width=0.1cm] at (0, -2.5) {\vdots};
	\draw (0, 0) node[circle, draw, fill=black!100, inner sep=0pt, minimum width=4pt] {} -- (0, -1) node[circle, draw, fill=black!100, inner sep=0pt, minimum width=4pt] {};
	\draw (0, -1) node[circle, draw, fill=black!100, inner sep=0pt, minimum width=4pt] {} -- (0, -2) node[circle, draw, fill=black!100, inner sep=0pt, minimum width=4pt] {};
	
	\draw (0, 0) node[circle, draw, fill=black!100, inner sep=0pt, minimum width=4pt] {} -- (-0.6, 1) node[circle, draw, fill=black!100, inner sep=0pt, minimum width=4pt] {};
	\draw (0, 0) node[circle, draw, fill=black!100, inner sep=0pt, minimum width=4pt] {} -- (0.6, 1) node[circle, draw, fill=black!100, inner sep=0pt, minimum width=4pt] {};
	
	\draw (0, -1) node[circle, draw, fill=black!100, inner sep=0pt, minimum width=4pt] {} -- (-1.8, 0) node[circle, draw, fill=black!100, inner sep=0pt, minimum width=4pt] {};
	\draw (0, -1) node[circle, draw, fill=black!100, inner sep=0pt, minimum width=4pt] {} -- (1.8, 0) node[circle, draw, fill=black!100, inner sep=0pt, minimum width=4pt] {};
	
	\draw (-1.8, 1) node[circle, draw, fill=black!100, inner sep=0pt, minimum width=4pt] {};
	\draw (1.8, 1) node[circle, draw, fill=black!100, inner sep=0pt, minimum width=4pt] {};
	\draw (-1.8, 2) node[circle, draw, fill=black!100, inner sep=0pt, minimum width=4pt] {};
	\draw (1.8, 2) node[circle, draw, fill=black!100, inner sep=0pt, minimum width=4pt] {};
	
	\draw (-1.8, 0) -- (-1.8, 1);
	\draw (-1.8, 1) -- (-1.8, 2);
	\draw (1.8, 0) -- (1.8, 1);
	\draw (1.8, 1) -- (1.8, 2);
	
	\node[label = $v_1$] at (-1.8,2) {};
	\node[label = $v_4$] at (1.8,2) {};
	\node[label = $v_2$] at (-0.6,1) {};
	\node[label = $v_3$] at (0.6,1) {};
	
	\node[label = $\alpha_2$] at (0,0.3) {};
	\node[label = $\alpha_1$] at (-0.3,-1+0.1) {};
	\node[label = $\alpha_3$] at (0.3,-1+0.1) {};
	
	\draw (-1.8+8, -1) node[circle, draw, fill=black!100, inner sep=0pt, minimum width=4pt, label = $v_1$] {};
	\draw (-0.6+8, -1) node[circle, draw, fill=black!100, inner sep=0pt, minimum width=4pt, label = $v_2$] {};
	\draw (0.6+8, -1) node[circle, draw, fill=black!100, inner sep=0pt, minimum width=4pt, label = $v_3$] {};
	\draw (1.8+8, -1) node[circle, draw, fill=black!100, inner sep=0pt, minimum width=4pt, label = $v_4$] {};		

	\draw (-1.8+8, -1) -- (1.8+8, -1);
	
	\node[label = $\alpha_2$] at (8, -1-0.2) {};
	\node[label = $\alpha_1$] at (8-1.2, -1-0.2) {};
	\node[label = $\alpha_3$] at (8+1.2, -1-0.2) {};
	
	\node[label = $\gr(v_1) \text{ = } \gr(v_4) \text{ = } 0$] at (7.75, 2) {};
	\node[label = $\gr(v_2) \text{ = } \gr(v_3) \text{ = } -2$] at (8, 1.5) {};		
	\node[label = $\gr(\alpha_2) \text{ = } -4$] at (7.175, 1) {};
	\node[label = $\gr(\alpha_1) \text{ = } \gr(\alpha_3) \text{ = } -6$] at (8.075, 0.5) {};	
			
\end{tikzpicture}
\caption{Geometric realization (right) of a graded root (left). The leaves $v_1$ and $v_4$ in this example have grading zero. It is helpful to think of the picture on the right as being obtained by projecting the graded root on the left onto a horizontal line.} \label{fig:4}
\end{figure}

More generally, we can formalize the above discussion as follows. Let $C$ be a cell complex (in the usual sense), and let $\gr$ be a function from the cells of $C$ to a coset of $2\mathbb{Z}$ in $\mathbb{Q}$ such that $\gr(\square_e) \leq \gr(\square_{e-1})$ whenever $\square_{e-1}$ is in the cellular boundary of $\square_e$. We give $C \otimes \ff[U]$ a grading by defining a $e$-cell $\square_e$ to have grading $\gr(\square_e) + e$, and declaring $U$ to have degree $-2$.\footnote{We use the term ``grading" to refer both to the actual chain complex grading and the grading function $\gr$. The former is given by $\gr + e$, while we reserve the notation $\gr$ for the latter.} Then the twisted differential (\ref{eqn:2}) is of degree $-1$ and turns $C \otimes \ff[U]$ into a chain complex. We call a complex defined in this way a \textit{geometric complex} and refer to the underlying cell complex $C$ (with its usual cellular differential) as the \textit{skeleton} of such a complex. The preceding discussion states that every standard complex has a geometric realization whose skeleton is homeomorphic to a line. Note that the passage from a geometric complex to its skeleton may be obtained by setting $U = 1$.

\begin{remark}
This is almost identical to the construction of lattice homology given in \cite[Definition 3.1.4]{Nem}. See \cite[Theorem 3.1.12]{Nem} for a discussion of the relationship between this construction and the sublevel set construction of e.g.\ \cite[Definition 3.1.11]{Nem}, \cite[Section 7]{DM}.
\end{remark}

If $X$ and $Y$ are two geometric complexes, then there is an obvious geometric complex which realizes $X \otimes Y$, defined as follows. The skeleton of $X \otimes Y$ is given by the usual product of the skeleton of $X$ and the skeleton of $Y$. We then define the grading of a $e$-cell $x \times y$ by setting $\gr(x \times y) = \gr(x) + \gr(y)$. This construction evidently extends to the product of any number of geometric complexes, so that if $R_1, \ldots, R_k$ are a set of $k$ symmetric graded roots, then $C_*(R_1) \otimes \cdots \otimes C_*(R_k)$ has an obvious geometric realization as a product $k$-dimensional rectangular cell complex. The action of $J_0 = J_0 \otimes \cdots \otimes J_0$ on this complex corresponds to reflection through the center of its skeleton. See Figure~\ref{fig:5}. 

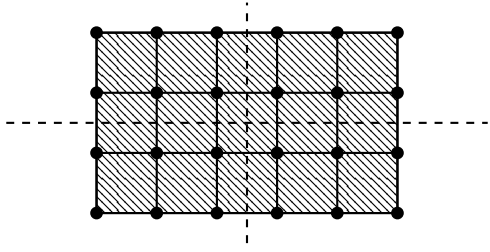
\begin{figure}[h!]
\begin{tikzpicture}[thick,scale=0.8]%

	\draw[dashed] (-4, 0) -- (4, 0);
	\draw[dashed] (0, -2) -- (0, 2);

	\draw (-2.5, 1.5) node[circle, draw, fill=black!100, inner sep=0pt, minimum width=4pt] {};
	\draw (-1.5, 1.5) node[circle, draw, fill=black!100, inner sep=0pt, minimum width=4pt] {};
	\draw (-0.5, 1.5) node[circle, draw, fill=black!100, inner sep=0pt, minimum width=4pt] {};
	\draw (0.5, 1.5) node[circle, draw, fill=black!100, inner sep=0pt, minimum width=4pt] {};
	\draw (1.5, 1.5) node[circle, draw, fill=black!100, inner sep=0pt, minimum width=4pt] {};
	\draw (2.5, 1.5) node[circle, draw, fill=black!100, inner sep=0pt, minimum width=4pt] {};
		
	\draw (-2.5, 0.5) node[circle, draw, fill=black!100, inner sep=0pt, minimum width=4pt] {};
	\draw (-1.5, 0.5) node[circle, draw, fill=black!100, inner sep=0pt, minimum width=4pt] {};
	\draw (-0.5, 0.5) node[circle, draw, fill=black!100, inner sep=0pt, minimum width=4pt] {};
	\draw (0.5, 0.5) node[circle, draw, fill=black!100, inner sep=0pt, minimum width=4pt] {};	
	\draw (1.5, 0.5) node[circle, draw, fill=black!100, inner sep=0pt, minimum width=4pt] {};
	\draw (2.5, 0.5) node[circle, draw, fill=black!100, inner sep=0pt, minimum width=4pt] {};
	
	\draw (-2.5, -0.5) node[circle, draw, fill=black!100, inner sep=0pt, minimum width=4pt] {};
	\draw (-1.5, -0.5) node[circle, draw, fill=black!100, inner sep=0pt, minimum width=4pt] {};
	\draw (-0.5, -0.5) node[circle, draw, fill=black!100, inner sep=0pt, minimum width=4pt] {};
	\draw (0.5, -0.5) node[circle, draw, fill=black!100, inner sep=0pt, minimum width=4pt] {};	
	\draw (1.5, -0.5) node[circle, draw, fill=black!100, inner sep=0pt, minimum width=4pt] {};
	\draw (2.5, -0.5) node[circle, draw, fill=black!100, inner sep=0pt, minimum width=4pt] {};
	
	\draw (-2.5, -1.5) node[circle, draw, fill=black!100, inner sep=0pt, minimum width=4pt] {};
	\draw (-1.5, -1.5) node[circle, draw, fill=black!100, inner sep=0pt, minimum width=4pt] {};
	\draw (-0.5, -1.5) node[circle, draw, fill=black!100, inner sep=0pt, minimum width=4pt] {};
	\draw (0.5, -1.5) node[circle, draw, fill=black!100, inner sep=0pt, minimum width=4pt] {};	
	\draw (1.5, -1.5) node[circle, draw, fill=black!100, inner sep=0pt, minimum width=4pt] {};
	\draw (2.5, -1.5) node[circle, draw, fill=black!100, inner sep=0pt, minimum width=4pt] {};
	
	\draw (-2.5, 1.5) -- (2.5, 1.5);
	\draw (-2.5, 0.5) -- (2.5, 0.5);
	\draw (-2.5, -0.5) -- (2.5, -0.5);
	\draw (-2.5, -1.5) -- (2.5, -1.5);
	
	\draw (-2.5, 1.5) -- (-2.5, -1.5);
	\draw (-1.5, 1.5) -- (-1.5, -1.5);
	\draw (-0.5, 1.5) -- (-0.5, -1.5);
	\draw (0.5, 1.5) -- (0.5, -1.5);
	\draw (1.5, 1.5) -- (1.5, -1.5);
	\draw (2.5, 1.5) -- (2.5, -1.5);
	
	\draw[pattern=north west lines] (-2.5, -1.5) rectangle (2.5,1.5);
\end{tikzpicture}
\caption{Skeleton of the tensor product of two roots ($k = 2$). Points (0-cells) represent generators $x = x_1 \otimes x_2$ where both $x_1$ and $x_2$ are leaf generators; edges (1-cells) represent generators where precisely one of $x_1$ and $x_2$ is an angle generator; squares (2-cells) represent generators where both of the $x_i$ are angle generators. Note the symmetry about each coordinate axis.} \label{fig:5}
\end{figure}


\section{Local Equivalence}
\label{sec:3}
We now turn to a discussion of local equivalence, beginning with the case of a single graded root. Following \cite[Section 6]{DM}, we define a special class of symmetric graded roots which we call \textit{monotone}, as follows. Let $(C, J_0)$ be a symmetric chain complex as in Section~\ref{sec:2.3}, with $2n$ leaf generators and $2n-1$ angle generators. Thus the leaf and angle generators of $C$ occur in symmetric pairs, with the exception of the single $J_0$-invariant generator $\alpha_n$. If the inequalities
\begin{enumerate}
\item $\gr(v_1) > \gr(v_2) > \cdots > \gr(v_n)$, 
\item $\gr(\alpha_1) < \gr(\alpha_2) < \cdots < \gr(\alpha_n)$, and 
\item $\gr(v_n) \geq \gr(\alpha_n)$
\end{enumerate}
hold, then we say that the symmetric graded root $M$ corresponding to the homology of $C$ is a \textit{monotone root of type $n$}. To emphasize the $2n$ parameters $\gr(v_i)$ and $\gr(\alpha_i)$, we write 
\[
M = M(\gr(v_1), \gr(\alpha_1); \ldots; \gr(v_n), \gr(\alpha_n)).
\]
Examples of monotone roots are given in Figure~\ref{fig:6}. Note that if $\gr(v_n) > \gr(\alpha_n),$ then $M$ itself has $2n$ leaves, while if $\gr(v_n) = \gr(\alpha_n)$, we are in the slightly degenerate case where on the level of homology $v_n$ and $v_{n+1}$ are collapsed into a single $J_0$-invariant leaf. In the special case that a monotone root $M = M(h_1, r_1)$ has only two parameters, we will sometimes refer to $M$ as a \textit{projective root}, following \cite[Fact 5.6]{Stoffregen}. In this case it will be convenient to define\footnote{This convention differs from $\tilde{\delta}$ in \cite{Stoffregen2} in that the present $\tilde{\delta}$ is twice that of \cite{Stoffregen2}.}
\[
\tilde{\delta}(M) = h_1 - r_1.
\]
The graded roots $Y_i$ of Figure~\ref{fig:1} are projective, with $\tilde{\delta}(Y_i) = 2i$ for all $i \geq 1$.

\begin{figure}[h!]
\center
\includegraphics[scale=1]{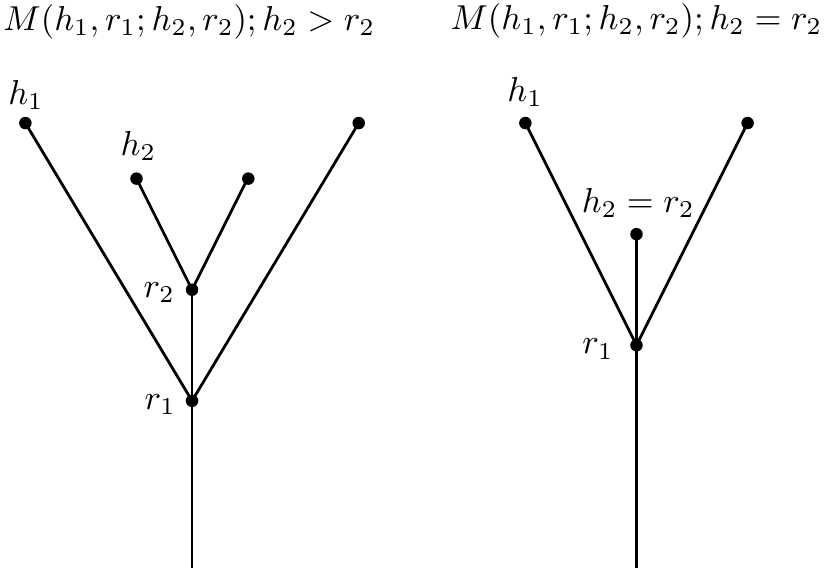}
\caption{Examples of monotone roots $(n = 2)$. Most of the vertices in the diagram have not been drawn explicitly.} \label{fig:6}
\end{figure}

We will also be interested in a slightly more general class of symmetric graded roots, defined as follows. If $C$ is a complex as above for which the slightly modified conditions 
\begin{enumerate}
\item $\gr(v_1) \geq \gr(v_2) \geq \cdots \geq \gr(v_n)$, 
\item $\gr(\alpha_1) \leq \gr(\alpha_2) \leq \cdots \leq \gr(\alpha_n)$, and
\item $\gr(v_n) \geq \gr(\alpha_n)$
\end{enumerate}
hold, then we say that the homology $M$ of $C$ is a \textit{weakly monotone root of type $n$}, and we parameterize it similarly as above. 

It should be noted that occasionally the parameterized complex $C$ does not quite coincide with the standard complex of $M$, at least as defined in Section~\ref{sec:2.3}. For example, if $M$ is a (strictly) monotone root for which $\gr(v_n) = \gr(\alpha_n)$, then the standard complex of $M$ has $2n-1$ leaf generators, while $C$ as defined above has $2n$ leaf generators. Similarly, if $M = M(h_1, r_1; \ldots; h_n, r_n)$ is a weakly monotone root, then appending any number of pairs $(r_n, r_n)$ to the end of the parameter list yields a larger complex whose weakly monotone root
\[
M = M(h_1, r_1; \ldots; h_n, r_n; r_n, r_n; \ldots; r_n, r_n)
\]
is easily checked to be isomorphic to the original, but whose associated complex $C$ has more generators. However, it is easy to see that all of these complexes are homotopy equivalent (via a $J_0$-equivariant chain homotopy), and that they are in fact homotopic to the usual standard complex of $M$. Hence when discussing chain complexes of monotone roots, we will freely use either the parameterized complex $C$ or the usual standard complex, whichever is more convenient.

The importance of the class of monotone roots lies in the following pair of theorems established in \cite{DM}:
\begin{theorem}\cite[Theorem 6.1]{DM}
\label{thm:3.1}
Every symmetric graded root is locally equivalent to some (strictly) monotone root (which is in fact a subroot of the original).
\end{theorem}
\begin{theorem}\cite[Theorem 6.2]{DM}
\label{thm:3.2}
Two (strictly) monotone roots are locally equivalent if and only if they have the same set of ordered parameters.
\end{theorem}
\noindent
Here, we again abuse notation slightly and say that two symmetric graded roots are locally equivalent if there is a local equivalence between their $\inv$-complexes. We will similarly make reference to the tensor product of two graded roots to mean the tensor product of their standard (or parameterized) complexes, and so on.

According to Theorems~\ref{thm:3.1} and \ref{thm:3.2}, monotone roots parameterize the local equivalence classes of (positively oriented) AR manifolds in $\Inv_{\Q}$. Indeed, if one is given two graded roots, to see if they are locally equivalent it suffices to extract their monotone subroots and determine if they are the same. (A straightforward algorithm for doing this is given in \cite[Section 6]{DM}.) However, it turns out that there are non-trivial local equivalences between tensor products of monotone roots which introduce unexpected relations in $\Inv_{\Q}$. The remainder of this section will be devoted to establishing some of these equivalences.

We begin by describing a convenient way of defining maps between geometric complexes. Let $X$ and $Y$ be two geometric complexes graded by the same coset of $2\mathbb{Z}$ in $\mathbb{Q}$. Suppose that we have a cellular map $f$ from the skeleton of $X$ to the skeleton of $Y$. Thus $f$ takes each $e$-cell in the skeleton of $X$ to a sum of $e$-cells in the skeleton of $Y$, and preserves the usual cellular differential.\footnote{Here, we abuse language slightly and refer to any chain map between the cellular chain complexes of skeleta as a ``cellular map", in order to distinguish these from chain maps between $\ff[U]$-complexes. Thus, $e$-cells are allowed to be taken to sums of $e$-cells and/or zero. In all of the examples considered in this paper, these maps will have a clear topological interpretation (which somewhat justifies their terminology).} We wish to know if $f$ can be promoted to a true grading-preserving chain map between $X$ and $Y$. Suppose that 
\[
f(x) = \sum_i y_i,
\]
where $x$ is a $e$-cell in the skeleton of $X$ and the $y_i$ are distinct $e$-cells in the skeleton of $Y$. Then there is an obvious choice of grading-preserving lift given by
\begin{equation}\label{eqn:3}
\widetilde{f}(x) = \sum_i U^{(\gr(y_i) - \gr(x))/2} y_i,
\end{equation}
where now we think of $x$ and the $y_i$ as actual generators in $X$ and $Y$ and the same set of $y_i$ appears in the sum. In order for this to be well-defined, we must check that each $e$-cell $y_i$ appearing in $f(x)$ satisfies the inequality $\gr(y_i) \geq \gr(x)$. If this condition holds for every $x$ in the skeleton of $X$, then $f$ evidently lifts to a grading-preserving map from $X$ to $Y$, and it is easily verified that this preserves the twisted differential (\ref{eqn:2}). We formalize this in the following lemma:

\begin{lemma}\label{lem:3.3}
Let $X$ be a geometric complex such that:
\begin{enumerate}
\item[(a)] The skeleton of $X$ is contractible, and
\item[(b)] There is an involution $J_0$ on the skeleton of $X$ respecting the usual cellular differential.
\end{enumerate}
Then $(X, J_0)$ is an $\inv$-complex, where $J_0$ is the induced map on $X$ coming from the involution on its skeleton. In sufficiently low gradings, the $\ff[U]$-homology of $X$ is generated by $U$-powers of $[x]$, where $[x]$ is the (appropriately homogenized) generator of the cellular homology of the skeleton of $X$.

Now suppose that $X$ and $Y$ are two such geometric complexes graded by the same coset of $2\Z$ in $\Q$. Let $f$ be a cellular map between their skeleta satisfying the following conditions:
\begin{enumerate}
\item[(a)] (Lifting condition.) $\gr(y_i) \geq \gr(x)$ holds whenever $y_i$ appears in $f(x)$,
\item[(b)] The map $f$ is $J_0$-equivariant, and
\item[(c)] The map $f$ takes each 0-cell in the skeleton of $X$ to the sum of an odd number of 0-cells in the skeleton of $Y$.
\end{enumerate}
Then the lift (\ref{eqn:3}) of $f$ is a local map from $X$ to $Y$.
\end{lemma}
\begin{proof}
Given that the skeleton of $X$ is contractible, it is straightforward to check that the homology of $X$ in sufficiently negative gradings (congruent to $\tau$ mod 2) is isomorphic to $\ff$, and is generated in these gradings by the sum of any odd number of 0-cells (multiplied by appropriate powers of $U$). As $J_0$ obviously lifts to an involution on $X$, this easily implies the first claim. For the second claim, the fact that $f$ is an isomorphism after inverting the action of $U$ follows from the requirement that $f$ takes each 0-cell in $X$ to an odd number of 0-cells in $Y$. See also \cite[Lemma 7.1]{DM}. 
\end{proof}

\begin{example}\label{ex:3.4}
We illustrate the use of Lemma~\ref{lem:3.3} with the following example. Let $x$ and $y$ be two (negative) even integers for which $x < y < 0$. We establish the local equivalence 
\[
M(0, x; y, y) + M(0, x - y) = M(0, x).
\]
One can of course instead write $M(0, x; y, y) = M(0, x) - M(0, x-y)$; this expresses the type-2 root on the left as the tensor product of a type-1 root and the dual of a type-1 root. The monotone roots in question are displayed in Figure~\ref{fig:7}, together with their standard complexes. 

\begin{figure}[h!]
\center
\includegraphics[scale=1]{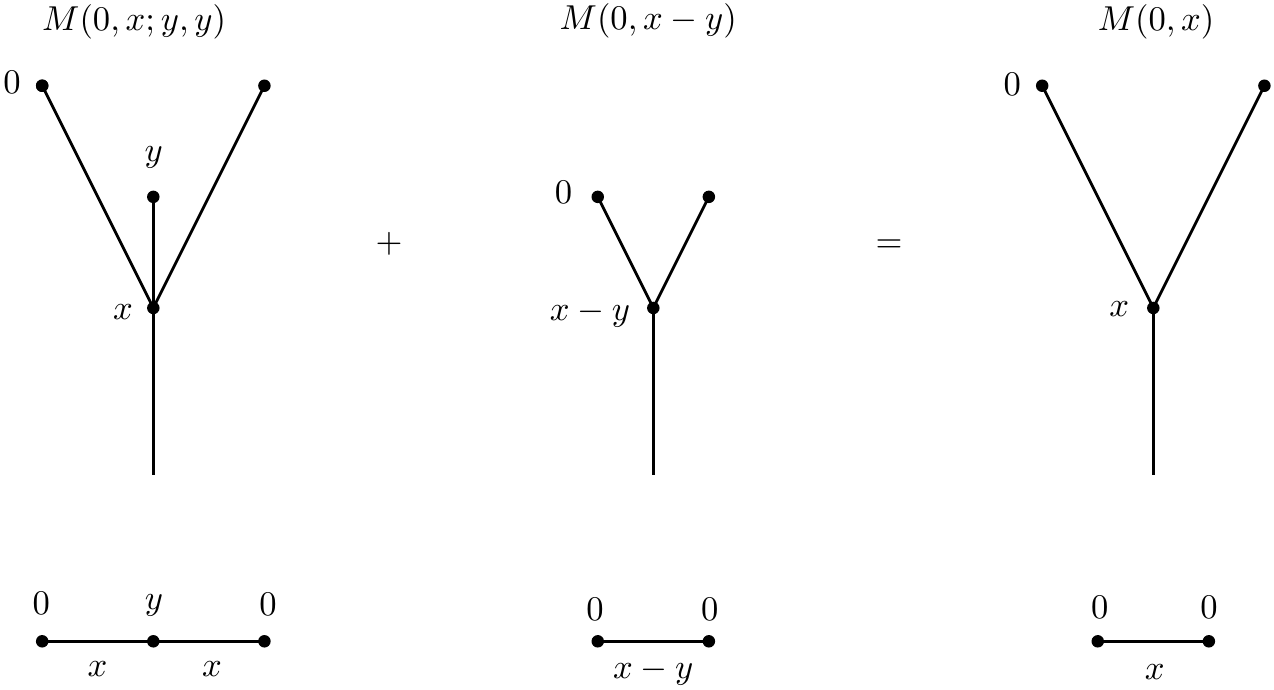}
\caption{Monotone roots of Example~\ref{ex:3.4} (above) and the geometric realizations of their standard complexes (below). Gradings of leaf and angle generators are marked in both parts of the diagram.} \label{fig:7}
\end{figure}

We have also displayed the geometric realization of $M(0, x; y, y) + M(0, x - y)$ on the left-hand side in Figure~\ref{fig:8}, together with the gradings of its cells. (Both 2-cells have grading $x + (x - y) = 2x - y$.) Note that the middle vertical 1-cell in the diagram corresponds to the tensor product of the $J_0$-invariant 0-cell of $M(0, x; y, y)$ with the $J_0$-invariant 1-cell of $M(0, x-y)$ and thus has grading $y + (x-y) = x$. For brevity we will refer to the skeleton of $M(0, x; y, y) + M(0, x - y)$ as $A$ and the skeleton of $M(0, x)$ as $B$. 

\begin{figure}[h!]
\center
\includegraphics[scale=1]{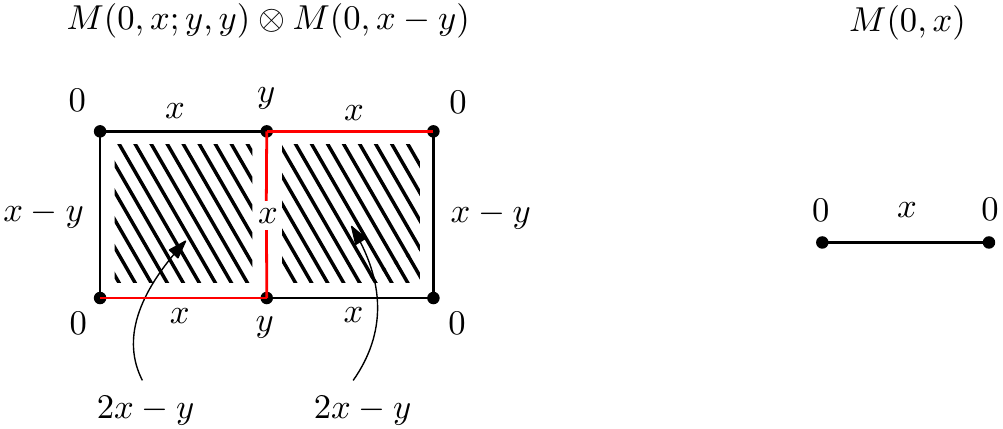}
\caption{Geometric realizations of $M(0, x; y, y) \otimes M(0, x - y)$ (left) and $M(0, x)$ (right).} \label{fig:8}
\end{figure}

We now describe a local map from $M(0, x; y, y) + M(0, x - y)$ to $M(0, x)$ by defining a cellular map between their skeleta. We send the bottom-left, top-left, and top-middle 0-cells of $A$ to the left-hand 0-cell in $B$. We similarly map the bottom-right, top-right, and bottom-middle 0-cells of $A$ to the right-hand 0-cell in $B$. Note that since $y < 0$, this correspondence satisfies the lifting condition of Lemma~\ref{lem:3.3}. We now send the three red 1-cells in $A$ to the single 1-cell in $B$, and send all the other 1- and 2-cells to zero. It is easily checked that this is a $J_0$-equivariant cellular map (over $\mathbb{Z}/2\Z$). Applying Lemma~\ref{lem:3.3}, we obtain the local map in question.

Defining a local map from $M(0, x)$ to $M(0, x; y, y) + M(0, x - y)$ is even simpler: we send the left-hand 0-cell in $B$ to the bottom-left 0-cell in $A$, and the right-hand 0-cell in $B$ to the top-right 0-cell in $A$. We then send the single 1-cell in $B$ to the sum of the three red 1-cells in $A$. Applying Lemma~\ref{lem:3.3}, we obtain the desired local equivalence. \qed
\end{example}

We now proceed to the main theorem of the section. In order to formulate this, we will need to set up a bit of extra notation. Let $X$ and $Y$ be two weakly monotone roots. We say that two roots $X'$ and $Y'$ are obtained from $X$ and $Y$ via a \textit{swap operation} if they are related as in Figure~\ref{fig:9}. More precisely, suppose that $X$ is a type-$m$ root, and fix any $1 \leq a \leq m$. Consider the set of leaves in $X$ whose indices are strictly greater than $a$ in absolute value. Let $S_X$ be the subgraph of $X$ formed by taking the union over the paths connecting these leaves to the vertex supporting the angle $\alpha_a$ in $X$. This subgraph is marked in blue in Figure~\ref{fig:9}. Letting $Y$ be a type-$n$ root and $1 \leq b \leq n$, we define $S_Y$ similarly. We then form $X'$ and $Y'$ by cutting out $S_X$ and $S_Y$ from $X$ and $Y$, respectively, and interchanging them. We say that $X'$ and $Y'$ are a \textit{valid swapped pair} if $X'$ and $Y'$ are still weakly monotone.

\begin{figure}[h!]
\center
\includegraphics[scale=0.8]{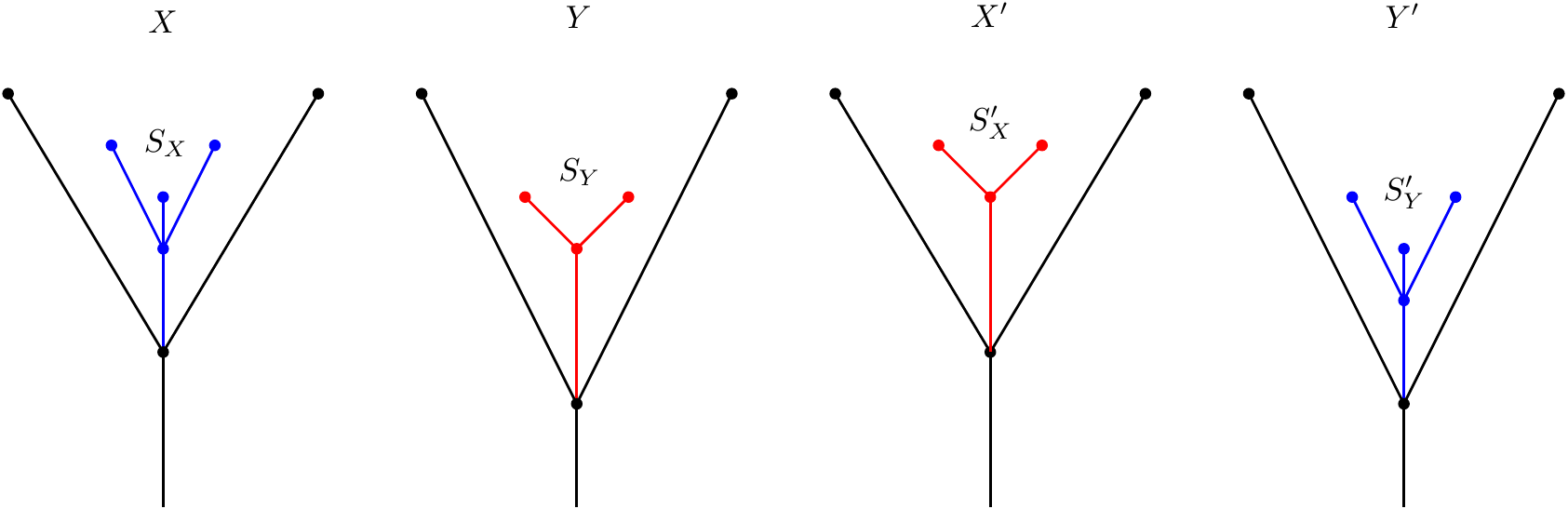}
\caption{Swap operation on two monotone roots $X$ and $Y$. In this example, $m = 3$, $n =2$, and $a = b = 1$.} \label{fig:9}
\end{figure}

To see what this does to the parameters of $X$ and $Y$, let
\begin{align*}
X &= M(p_1, s_1; \ldots; p_m, s_m) \text{ and} \\
Y &= M(q_1, t_1; \ldots; q_n, t_n).
\end{align*}
Fix any pair of integers $1 \leq a \leq m$ and $1 \leq b \leq n$, and let $\Delta = s_a - t_b$. Then the swapped roots $X'$ and $Y'$ are defined by the parameters
\begin{align*}
X' &= M(p_1, s_1; \ldots; p_a, s_a; q_{b+1} + \Delta, t_{b+1} + \Delta; \ldots; q_n + \Delta, t_n + \Delta) \text{ and} \\
Y' &= M(q_1, t_1; \ldots; q_b, t_b; p_{a+1} - \Delta, s_{a+1} - \Delta; \ldots; p_m - \Delta, s_m - \Delta).
\end{align*}
The condition that $X'$ and $Y'$ are a valid swapped pair is equivalent to the pair of inequalities
\begin{align}\label{eqn:4}
\begin{split}
p_a &\geq q_{b+1} + \Delta \text{ and }\\
q_b &\geq p_{a+1} - \Delta,
\end{split}
\end{align}
since the other monotone inequalities are automatically satisfied. (For example, $s_a \leq t_{b+1} + \Delta = s_a + (t_{b+1} - t_b)$, and so on.) Note that we allow the swapping operation even if one of the subgraphs is empty, which occurs if $a = m$ or $b = n$. In this case, the monotonicity condition for one of the swapped roots is then trivial.

We now have the main result of this section:

\begin{theorem}
\label{thm:3.5}
Let $X$ and $Y$ be two weakly monotone roots, and suppose $X'$ and $Y'$ are a valid swapped pair obtained from $X$ and $Y$. Then $X \otimes Y$ is locally equivalent to $X' \otimes Y'$. 
\end{theorem}
\begin{proof}
Let $X$ and $Y$ be two weakly monotone roots as above, and assume that $X'$ and $Y'$ are a valid swapped pair. In order to communicate the idea of the proof, we will first establish the claim under the assumption that $m = n$ and $a = b = 1$. The general case will follow easily from this simpler one with some minor modifications.

For ease of notation, let the leaf and angle generators of $X$ be denoted by $v_i$ and $\alpha_i$, respectively, and let the leaf and angle generators of $Y$ be denoted by $w_i$ and $\beta_i$. Likewise, let the generators of $X'$ be $v_i'$ and $\alpha_i'$, and let the generators of $Y'$ be $w_i'$ and $\beta_i'$. In case we need to refer to a generator of $X$ or $X'$ without specifying whether it is a leaf or an angle generator, we will use the notation $x_i$ or $x_i'$ (and similarly $y_i$ or $y_i'$ for $Y$ or $Y'$). Note that there is an obvious $\gr$-preserving identification between the unprimed and primed generators of index $|i| = 1$. For generators of index $|i| > 1$, we think of $x_i'$ as corresponding to $y_i$, except with shifted grading 
\[
\gr(x_i') = \gr(y_i) + \Delta. 
\]
Likewise, $y_i'$ may be identified with $x_i$, except that\footnote{Here, we use the term ``identification" rather loosely, as we have not (yet) made any claim about constructing maps between the unprimed and primed complexes. As the leaf and angle generators of various complexes are just 0- and 1-cells, we obviously have a tautological ``identification" between any two generators of the same dimension, although this may not preserve $\gr$. However, given the geometric nature of the swapping operation, it is clear that certain generators in the unprimed and primed complexes should be naturally thought of as related to each other.}  
\[
\gr(y_i') = \gr(x_i) - \Delta.
\]

As in Example~\ref{ex:3.4}, we proceed by analyzing the skeleta of $X \otimes Y$ and $X' \otimes Y'$. Because $m = n$, these two skeleta are identical to each other, and are in fact just $2m \times 2m$ squares. Hence the subtlety will come from understanding the grading function $\gr$ defined on each. We have displayed the skeleta of $X \otimes Y$ and $X' \otimes Y'$ in Figure~\ref{fig:10}, together with some notational colorings, which we now explain. 

\begin{figure}[h!]
\center
\includegraphics[scale=0.8]{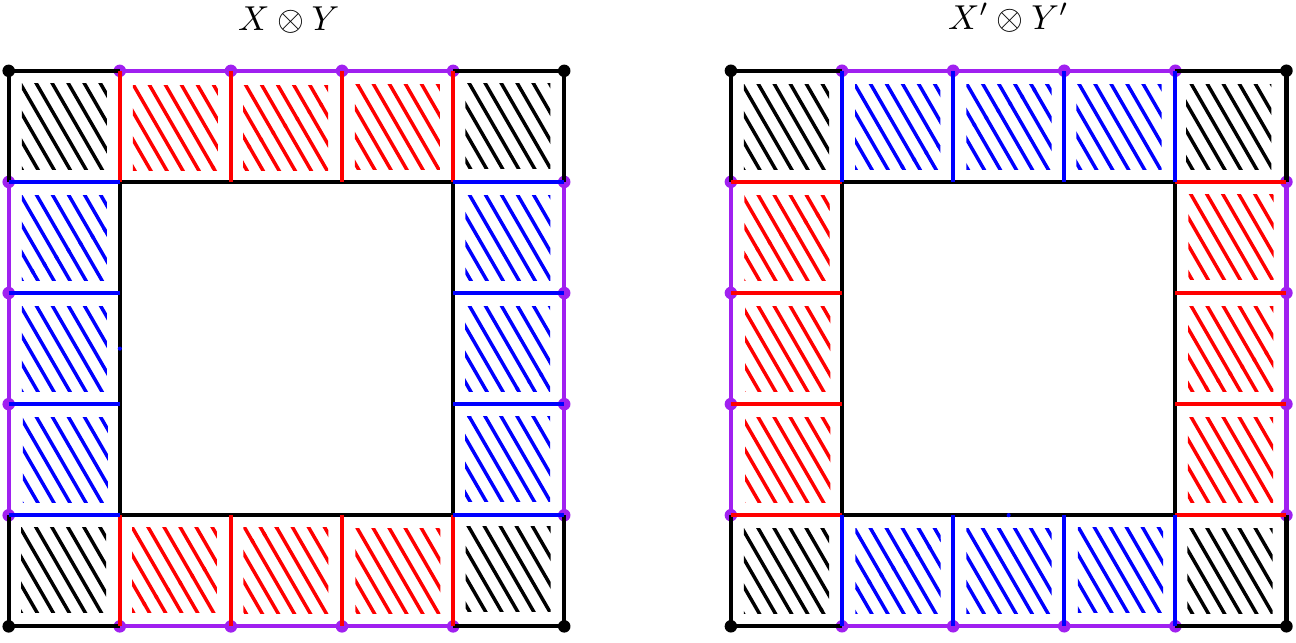}
\caption{Skeleta of $X \otimes Y$ and $X' \otimes Y'$. On the left, the inner box represents cells of the form $x_i \times y_j \text{ with } |i|, |j| > 1$; similarly for cells $x_i' 
\times y_j'$ on the right.} \label{fig:10}
\end{figure}

Consider the skeleton of $X \otimes Y$, displayed on the left in Figure~\ref{fig:10}. We divide this into several parts, as follows. First, instead of explicitly drawing the 0-, 1-, and 2-cells appearing in the center of the complex, we have drawn a $2(m-1) \times 2(m-1)$ square box enveloping them. We refer to the collection of cells appearing inside this square as the \textit{inner box}. These are the cells coming from products of leaf generators and angle generators internal to $S_X$ and $S_Y$; that is, cells of the form 
\[
x_i \times y_j, \text{with } |i|, |j| > 1. 
\]

Also appearing in the picture are the four corners of the overall square, together with the eight 1-cells and four 2-cells incident to them. These are cells of the form 
\[
x_i \times y_j, \text{with } |i| = |j| = 1. 
\]
We refer to the collection of these cells as the \textit{corner cells} and have colored/shaded them black in Figure~\ref{fig:10}. Next, there are the 0- and 1-cells which occur along the faces of the overall square, but are not corner cells. These are of the form
\begin{align*}
&v_i \times y_j, \text{with } |i| = 1 \text{ and } |j| > 1, \text{or}\\
&x_i \times w_j, \text{with } |i| > 1 \text{ and } |j| = 1.
\end{align*}
We call these cells the $\textit{side panels}$ and color them purple. Finally, there are also the 1- and 2-cells running between the inner box and the side panels. These are either of the form 
\begin{align*}
&\alpha_i \times y_j, \text{with } |i| = 1 \text{ and } |j| > 1, \text{or}\\
&x_i \times \beta_j, \text{with } |i| > 1 \text{ and } |j| = 1. 
\end{align*}
We refer to these as the \textit{bridge cells}. In Figure~\ref{fig:10}, we have colored/shaded the bridge cells of the first kind blue and the second kind red. We define a similar decomposition of $X' \otimes Y'$, displayed on the right in Figure~\ref{fig:10}, except that we switch the coloring of the red and blue cells. (The reason for this change will become clear presently.)

Now let us compare the tensor products $X \otimes Y$ and $X' \otimes Y'$. First, observe that because of the natural identification of unprimed and primed generators when $|i| = 1$, the obvious bijection between the corner cells of $X \otimes Y$ and the corner cells of $X' \otimes Y'$ preserves the grading function $\gr$. Explicitly, one can check that in both cases, the following gradings hold: all four 0-cells have $\gr = p_1 + q_1$, the four horizontal 1-cells have $\gr = s_1 + q_1$, the four vertical 1-cells have $\gr = p_1 + t_1$, and the four 2-cells have $\gr = s_1 + t_1$.

Next, consider the inner boxes of the two complexes. For $|i|, |j| > 1$, we have
\[
\gr(x_i' \times y_j') = (\gr(y_i) + \Delta) + (\gr(x_j) - \Delta) = \gr(y_i \times x_j).
\]
Hence the correspondence between the two inner boxes defined by sending $x_i' \times y_j'$ to $y_i \times x_j$ for $|i|, |j| > 1$ preserves $\gr$. We may think of this geometrically as identifying the inner boxes via reflection across the diagonal. Note that since each complex is symmetric about the vertical and horizontal axes, we can also effect a $\gr$-preserving correspondence by using a ninety-degree rotation in either the clockwise or counter-clockwise direction.

Finally, consider the bridge cells. For $X \otimes Y$, we have already described the possible forms that these can take. In $X' \otimes Y'$, the bridge cells are of the form
\begin{align*}
&\alpha_i' \times y_j', \text{with } |i| = 1 \text{ and } |j| > 1, \text{or}\\
&x_i' \times \beta_j', \text{with } |i| > 1 \text{ and } |j| = 1. 
\end{align*}
Gradings of cells of the first kind are given by 
\[
s_1 + \gr(y_j') = s_1 + \gr(x_j) - \Delta = \gr(x_j) + t_1
\] 
for $|j| > 1$. These are precisely the same as the gradings of bridge cells of the second kind in $X \otimes Y$. Similarly, gradings of cells of the second kind in $X' \otimes Y'$ are given by
\[
\gr(x_i') + t_1 = \gr(y_i) + \Delta + t_1 = s_1 + \gr(y_i),
\]
for $|i| > 1$, which are the same as the gradings of bridge cells of the first kind in $X \otimes Y$. Hence we see that, as in the inner-box case, we again have a $\gr$-preserving correspondence between the two sets of bridge cells, given by a ninety-degree rotation in either direction. To represent this identification, we have interchanged the red and blue colorings in $X' \otimes Y'$, so that the obvious bijection between the red (respectively, blue) cells in $X \otimes Y$ and the red (respectively, blue) cells in $X' \otimes Y'$ given by either rotation preserves $\gr$.

These observations give some intuition for how to define a map between the skeleton of $X \otimes Y$ and the skeleton of $X' \otimes Y'$. Roughly speaking, we map the corner cells to each other via the obvious identity bijection, while using a ninety-degree rotation on the inner box and bridge cells. There are several evident difficulties with this suggestion. First, it is unclear how to extend this to an actual cellular map from one skeleton to the other while still preserving the cellular differential. Second, we have not specified how to define our map on the side panels. Indeed, one can check that the gradings of the cells in the side panels do \textit{not} work out to produce any such similar correspondence, either by an identity bijection or a ninety-degree rotation. We thus proceed instead by first constructing two auxiliary complexes which are locally equivalent to $X \otimes Y$ and $X' \otimes Y'$. As we will see, these auxiliary complexes will turn out to be more amenable to the intuition we have built up over the last several paragraphs.

We begin by defining a geometric complex $A$ which is locally equivalent to $X \otimes Y$. The skeleton of $A$ is displayed on the right in Figure~\ref{fig:11}. Like $X \otimes Y$, the skeleton of $A$ has a $2(m-1) \times 2(m-1)$ inner box, which we define to be exactly same as that of $X \otimes Y$. More precisely, we let the cells in the inner box of $A$ be duplicates of those in $X \otimes Y$, and we set $\gr$ on the inner box of $A$ to be the same as it is on $X \otimes Y$. Next, the skeleton of $A$ contains a collection of red and blue 1- and 2-cells. These are in obvious bijection with the red and blue bridge cells in $X \otimes Y$, as indicated by the correspondence in Figure~\ref{fig:11}. We define $\gr$ on these cells as being equal to $\gr$ of their counterparts in $X \otimes Y$. Finally, $A$ contains a total of twelve black 0-, 1-, and 2-cells. We define $\gr$ on these cells as follows: the four 0-cells have $\gr = p_1 + q_1$, the two horizontal 1-cells have $\gr = s_1 + q_1$, the two vertical 1-cells have $\gr = p_1 + t_1$, and the four 2-cells have $\gr = s_1 + t_1$. These gradings should be compared with the gradings of the corner cells in $X \otimes Y$. Note that $A$ has an obvious involution given by reflection in each coordinate. 
\begin{figure}[h!]
\center
\includegraphics[scale=0.8]{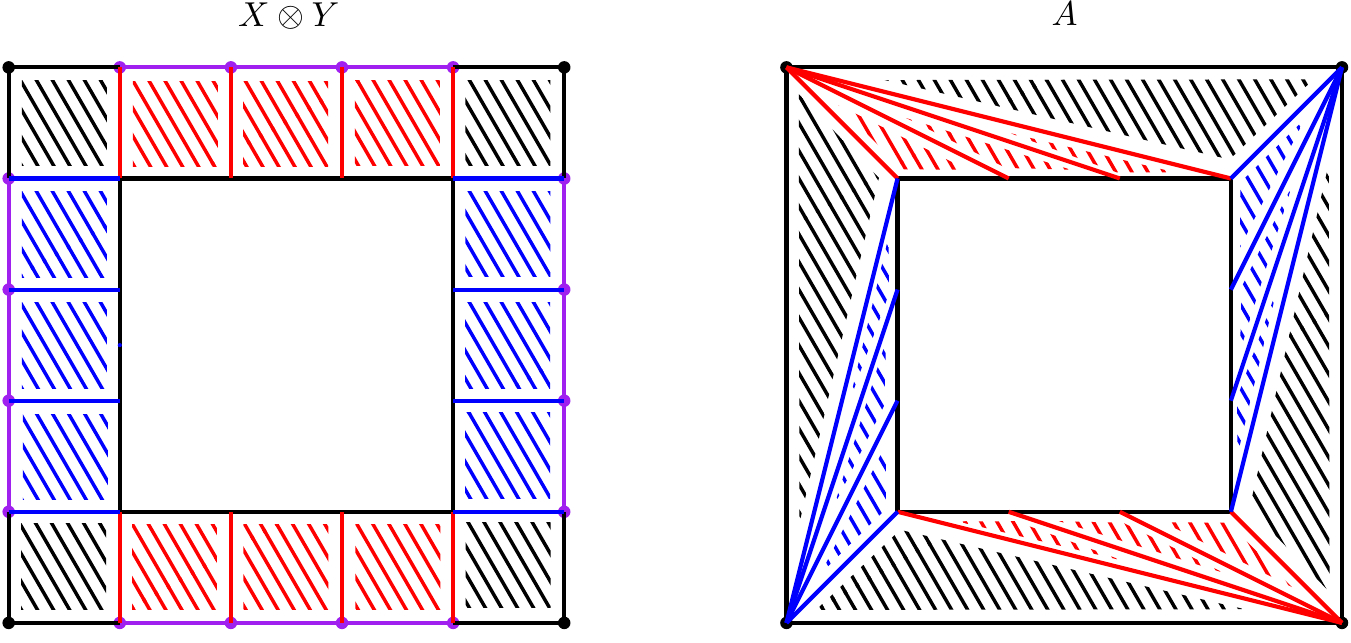}
\caption{Auxiliary complex $A$ (right), locally equivalent to $X \otimes Y$ (left).} \label{fig:11}
\end{figure}

This defines a geometric complex $A$ which we claim is locally equivalent to $X \otimes Y$. We begin by constructing a local map from $X \otimes Y$ to $A$ by defining a map $f$ between their skeleta. We let $f$ be given by the identity on the inner box, and define $f$ on the red and blue cells via the bijection indicated in Figure~\ref{fig:12}. Next, we send the four corners of $X \otimes Y$ to the four corners of $A$, and the four black 2-cells of $X \otimes Y$ to the four black 2-cells of $A$ via the correspondence shown in Figure~\ref{fig:12}. It remains to define $f$ on the black 1-cells and the side panels. 
\begin{figure}[h!]
\center
\includegraphics[scale=0.8]{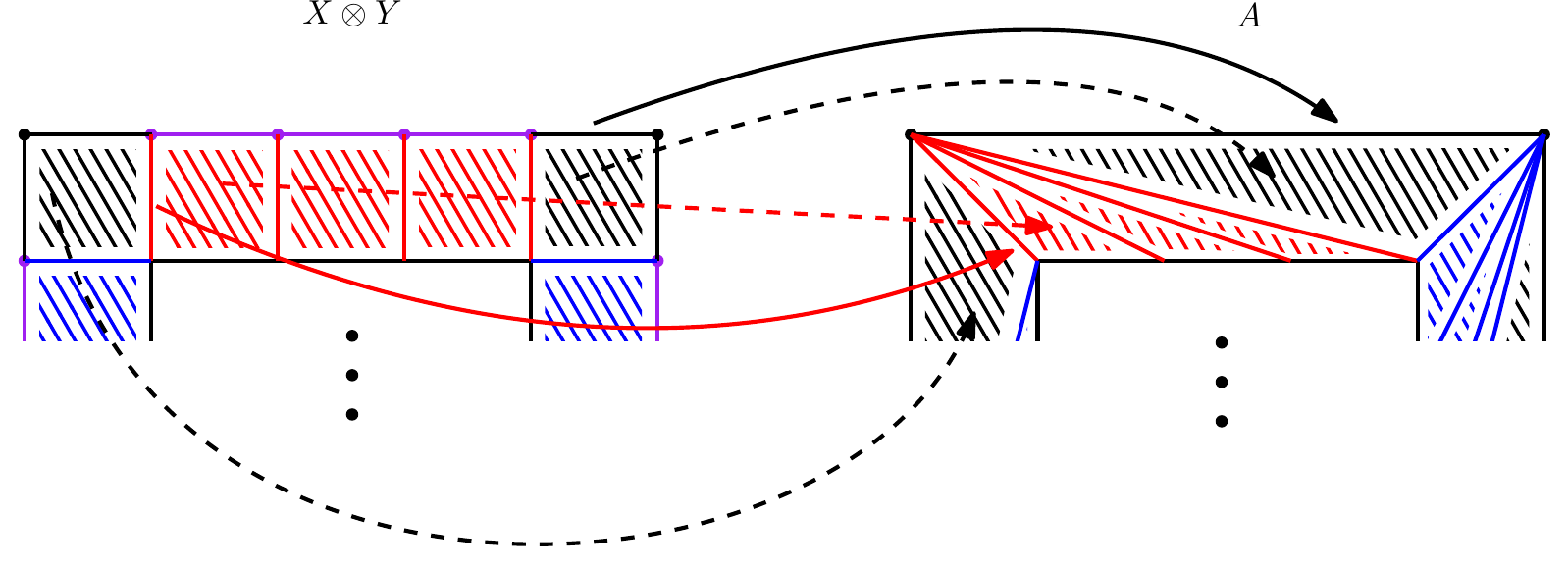}
\caption{Schematic depiction of a map from $X \otimes Y$ to $A$. Solid arrows denote maps of 1-cells and dashed arrows denote maps of 2-cells. We have only indicated the map on a single red 1-cell and 2-cell; the map on the other red cells is defined similarly. Extension to the rest of $X \otimes Y$ (not displayed in the figure) is achieved by repeatedly rotating the picture by ninety degrees.} \label{fig:12}
\end{figure}

Consider (for example) the uppermost face of $X \otimes Y$. We send all of the purple 0-cells contained in this face to the nearest counter-clockwise corner 0-cell, which in this case is the top-left corner of $A$. We define $f$ to be zero on all of the 1-cells contained in this face, except for the right-hand black 1-cell, which we map to the uppermost face in $A$. One should think of $f$ restricted to the top face of $X \otimes Y$ as a map from the interval to itself, fixing the endpoints, where everything to the left of the right-hand black 1-cell is crushed to the left endpoint and the remaining 1-cell is ``stretched out" over the full interval. We define $f$ similarly in a ``counter-clockwise" manner on the other three sides of $X \otimes Y$. It is straightforward to check that this map preserves the cellular differential.

Note that $f$ actually maps each cell of $X \otimes Y$ either to a single cell of $A$ or to zero. Because of the way $\gr$ is defined on $A$, in the former case $f$ actually preserves $\gr$ everywhere except for possibly on the purple 0-cells. Now, the purple 0-cells contained in the top face of $X \otimes Y$ have gradings 
\[
\gr = p_i + q_1, \text{with } 2 \leq i \leq m. 
\]
Due to the monotonicity of $X$, these are all less than or equal to $p_1 + q_1$, which is the grading of the top-left corner of $A$. A similar argument for the other purple 0-cells using the monotonicity of $Y$ shows that $f$ satisfies the lifting condition of Lemma~\ref{lem:3.3}. Since $f$ is evidently $J_0$-equivariant and maps each 0-cell of $X \otimes Y$ to exactly one 0-cell of $A$, applying Lemma~\ref{lem:3.3} yields the desired local map.

We now construct a local map from $A$ to $X \otimes Y$ by defining a map $g$ from the skeleton of $A$ to the skeleton of $X \otimes Y$. We set $g$ to be the inverse of $f$ on the inner box, as well as all black, red, and blue 2-cells. We define $g$ on the corner 0-cells via the obvious identity correspondence, and define $g$ on the black 1-cells by sending each face of $A$ to the sum of 1-cells constituting the corresponding face in $X \otimes Y$. Finally, we define $g$ on the red and blue 1-cells of $A$ according to the prescription of Figure~\ref{fig:13}. More precisely, we send a red or blue 1-cell in $A$ to the corresponding red or blue 1-cell in $X \otimes Y$, summed together with the chain of purple and black 1-cells running from that 1-cell to the nearest counter-clockwise corner 0-cell. With some care, one can check that $g$ preserves the cellular differential (over $\mathbb{Z}/2\mathbb{Z}$).
\begin{figure}[h!]
\center
\includegraphics[scale=1]{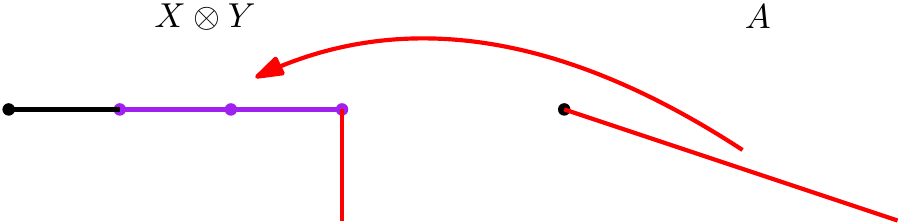}
\caption{Slanted red 1-cell in $A$ mapped to the sum of 1-cells in $X \otimes Y$ displayed on the left.} \label{fig:13}
\end{figure}

We now check that $g$ satisfies the lifting condition of Lemma~\ref{lem:3.3}. Note that, except in the case of the black 1-cells and bridge 1-cells of $A$, $g$ maps each cell of $A$ to a single cell in $X \otimes Y$ of the same grading. We thus consider the black 1-cells first. As described, $g$ takes the uppermost face of $A$ to the sum of 1-cells constituting the top face of $X \otimes Y$. The gradings of these latter 1-cells are given by
\[
\gr = s_i + q_1, \text{with }1 \leq i \leq m. 
\]
Due to the monotonicity of $X$, these are all greater than or equal to $s_1 + q_1$, which is the grading of the top face of $A$. A similar computation using the monotonicity of $Y$ holds for the other three sides, verifying the lifting condition for the black 1-cells.

Now consider a red 1-cell in $A$, as in Figure~\ref{fig:13}. This has grading $\gr = p_i + t_1$ for some fixed $|i| > 1$. Note that this is less than or equal to $p_2 + t_1$, by the monotonicity of $X$. The corresponding red 1-cell in the image of $g$ has exactly the same grading as the original, while the remaining 1-cells in the image of $g$ all lie in the uppermost face of $X \otimes Y$, and have gradings bounded below by $s_1 + q_1$ (by the previous paragraph). Thus, in this instance, to verify the monotonicity condition it suffices to establish the inequality
\[
p_2 + t_1 \leq s_1 + q_1.
\]
However, rearranging this yields the inequality $q_1 \geq p_2 - \Delta$, which is precisely the condition that $Y'$ is weakly monotone, as in (\ref{eqn:4}). A similar argument for the blue cells using the fact that $X'$ is weakly monotone proves that $g$ satisfies the lifting condition for all cells. Applying Lemma~\ref{lem:3.3} then yields the desired local equivalence.

We have thus constructed a somewhat simpler geometric complex $A$ which is locally equivalent to $X \otimes Y$. Note that (roughly speaking) $A$ consists entirely of the corner cells, inner box, and bridge cells of $X \otimes Y$, with the side panels having ``disappeared". We now construct a similar geometric complex $B$ which is locally equivalent to $X' \otimes Y'$, displayed on the right in Figure~\ref{fig:14}. This is defined in exactly the same way as $A$, except with one important modification. Whereas before we mapped the purple 0-cells in $X \otimes Y$ to the nearest \textit{counter-clockwise} corner 0-cell, we now map the purple 0-cells in $X' \otimes Y'$ to the nearest \textit{clockwise} corner 0-cell. (We likewise reverse the chirality of the various other maps needed to define the local equivalence.) This results in the complex $B$ shown in Figure~\ref{fig:14}. Again, the fact that $B$ is locally equivalent to $X' \otimes Y'$ follows from the fact that $X$, $Y$, $X'$, and $Y'$ are all weakly monotone.
\begin{figure}[h!]
\center
\includegraphics[scale=0.8]{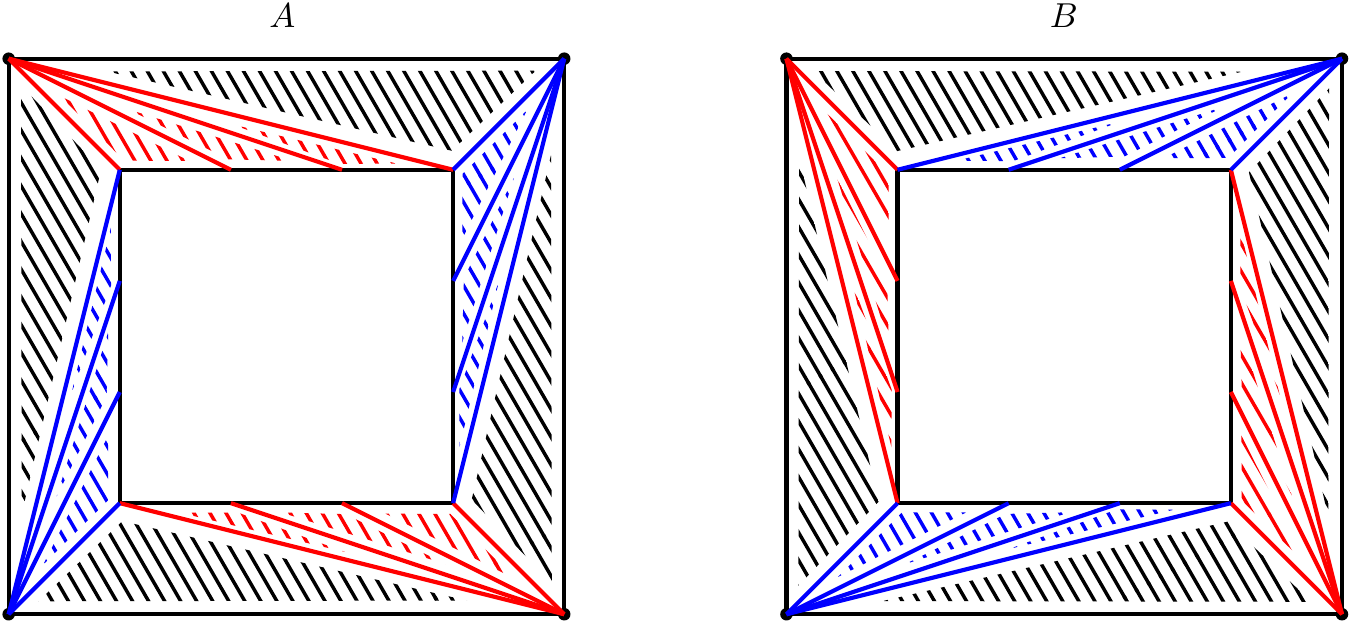}
\caption{Geometric complexes $A$ and $B$, locally equivalent to $X \otimes Y$ and $X' \otimes Y'$. Compare with Figure~\ref{fig:10}.} \label{fig:14}
\end{figure}

We now observe, however, that there is actually a $\gr$-preserving isomorphism between $A$ and $B$. This is most easily described geometrically, as follows. We view the black 0- and 1-cells as constituting a fixed outer frame, and we think of the black 2-cells together with the red and blue cells as being malleable strings/rubber sheets connecting the frame to the inner box. To map $A$ onto $B$, we rotate the inner box by ninety degrees counter-clockwise, while keeping the outer frame fixed. This ``drags" the remaining cells of $A$ into bijection with the cells of $B$. Moreover, because of our discussion at the beginning of the proof, this correspondence preserves $\gr$. Similarly, to map $B$ onto $A$, we rotate the inner box clockwise by ninety degrees. This shows that $A$ and $B$ are actually the same complex, and thus proves the theorem in the case that $m = n$ and $a = b = 1$.

We now relax the conditions on $m, n, a,$ and $b$ by allowing these to vary. For the moment, we will still assume that $m - a = n - b$.  As before, we again divide $X \otimes Y$ into four types of cells. These have been color-coded and displayed on the left in Figure~\ref{fig:15}. First, we consider cells of the form
\[
x_i \times y_j, \text{with } |i| > a \text{ and } |j| > b. 
\]
These constitute a $2(m-a) \times 2(m-a)$ box of cells which we again think of as being the inner box. Next, we have cells of the form
\[
x_i \times y_j, \text{with } |i| \leq a \text{ and } |j| \leq b,
\]
which generalize the corner cells from before. In Figure~\ref{fig:15}, we have avoided drawing most of the corner cells, and have instead covered the majority of them with four black corner wedges. Roughly speaking, these correspond to the corner 0-cells of the previous case. We also have appropriate generalizations of the side panels and the bridge cells. Cells in the side panels are now of the form
\begin{align*}
&x_i \times y_j, \text{with } |i| < a \text{ and } |j| > b, \text{ or}\\
&v_i \times y_j, \text{with } |i| = a \text{ and } |j| > b, \text{ or}\\
&x_i \times y_j, \text{with } |i| > a \text{ and } |j| < b, \text{ or}\\
&x_i \times w_j, \text{with } |i| > a \text{ and } |j| = b.
\end{align*}
Meanwhile, bridge cells are of the form
\begin{align*}
&\alpha_i \times y_j, \text{with } |i| = a \text{ and } |j| > b, \text{ or}\\
&x_i \times \beta_j, \text{with } |i| > a \text{ and } |j| = b. 
\end{align*}
As before, it is easily checked that the corner cells of $X \otimes Y$ and $X' \otimes Y'$ are identical, while the inner boxes and bridge cells are related by a ninety-degree rotation.
\begin{figure}[h!]
\center
\includegraphics[scale=0.8]{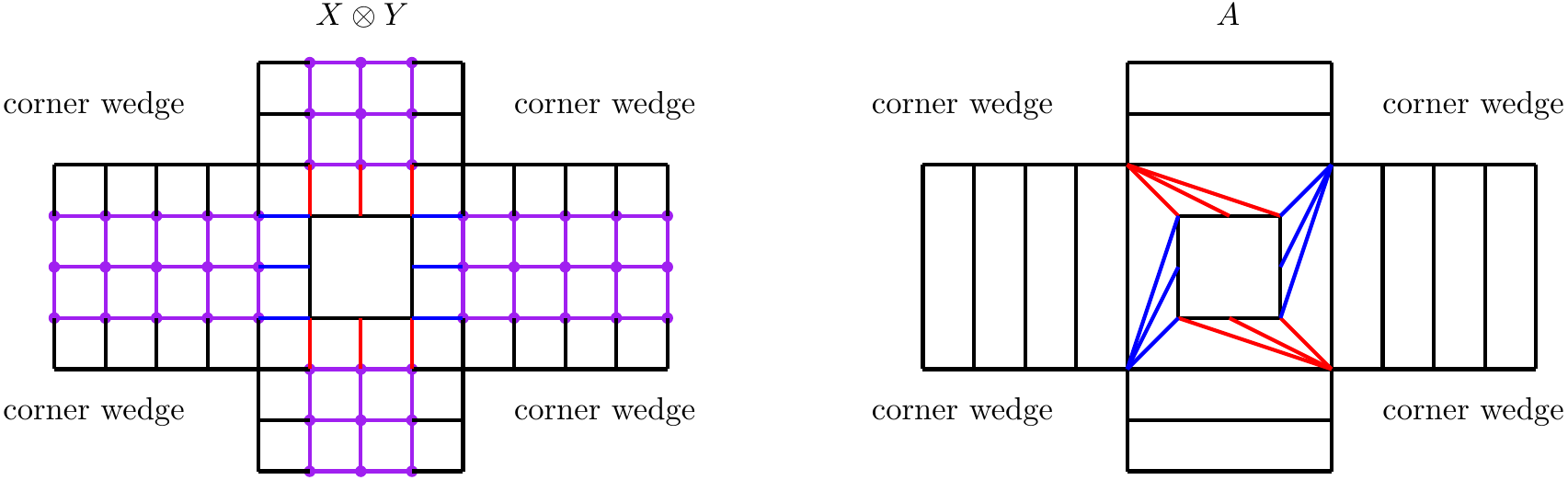}
\caption{The case where $m - a = n - b$. Auxiliary complex $A$ (right), locally equivalent to $X \otimes Y$ (left). We have avoided shading the 2-cells in order to prevent clutter in the diagram.} \label{fig:15}
\end{figure}

We will not re-formulate the entire argument here in the more general case. The idea is, of course, to construct an auxiliary complex $A$ which (roughly speaking) retains the corner cells, inner box, and bridge cells of $X \otimes Y$, while throwing out the side panels. Such a complex is displayed on the right in Figure~\ref{fig:15}. The map from $X \otimes Y$ to $A$ should again be thought of as taking (for example) the uppermost group of side panels and crushing them directly left towards the top-left corner wedge. In contrast to the previous case, there are now some purple 1-cells on which our map is non-zero, as well as some new black and purple 2-cells that must be considered. However, the desired maps on these cells are defined using the same intuition as before. We leave it to the reader to complete the generalization of the proof.

Finally, we consider the case when $m - a \neq n - b$. This initially seems more difficult, since if $m - a \neq n - b$, then the obvious inner box is no longer square, and hence cannot exhibit the desired rotational symmetry. However, if (without loss of generality) $m - a < n - b$, then as discussed in the beginning of the section we may pad the complex of $X$ with redundant parameters, in order to create an equivalent complex with a larger number of generators. More precisely, we append any number of copies of $(s_m, s_m)$ to the end of our parameter list to obtain a graded root
\[
\widetilde{X} = M(p_1, s_1; \ldots; p_m, s_m; s_m, s_m; \ldots; s_m, s_m).
\]
which is isomorphic to $X$. Replacing our original complex with this ``fattened" one does not alter the homotopy class or change the parameters $a$ and $b$ of the swapping operation. Hence we can always reduce to the case when $m - a = n - b$. This completes the proof. 
\end{proof}


\section{The Structure of $\Theta_{AR}$}
\label{sec:4}
In this section, we use the local equivalence of Theorem~\ref{thm:3.5} to prove Theorems~\ref{thm:1.1} and \ref{thm:1.2}. These will follow from a structural result expressing any monotone root as a linear combination of type-one roots and their duals. In order to establish this latter theorem, we will need the following special case of Theorem~\ref{thm:3.1}. For completeness, we give a sketch of the proof here in the language of geometric complexes:

\begin{lemma}\label{lem:4.1}
Let $M = M(h_1, r_1; \ldots; h_n, r_n)$ be a weakly monotone root. If $h_i = h_{i+1}$ for some $1 \leq i < n$, then $M$ is locally equivalent to the weakly monotone root
\[
\widetilde{M} = M(h_1, r_1; \ldots; h_{i-1}, r_{i-1}; h_{i+1}, r_{i+1}; \ldots; h_n, r_n)
\]
obtained by deleting the parameters $(h_i, r_i)$.
\end{lemma}
\begin{proof}
In Figure~\ref{fig:16}, we have displayed the two monotone roots in question, together with their associated skeleta. As before, we apply Lemma~\ref{lem:3.3}. To define the map from the left-hand to the right-hand side, we send the two 0-cells marked $\gr = h_i$ on the left to the two 0-cells marked $\gr = h_{i+1}$ on the right, and the two 1-cells with $\gr = r_i$ on the left to zero. We map all the other 0- and 1-cells via the obvious $\gr$-preserving bijection. The map in the other direction is given by sending each of the two 1-cells with $\gr = r_{i-1}$ on the right (if any) to the appropriate sum of 1-cells on the left with $\gr = r_{i-1}$ and $\gr = r_i$, and again using the obvious bijection on the remaining cells. See Figure~\ref{fig:16}.
\end{proof}
\begin{figure}[h!]
\center
\includegraphics[scale=0.8]{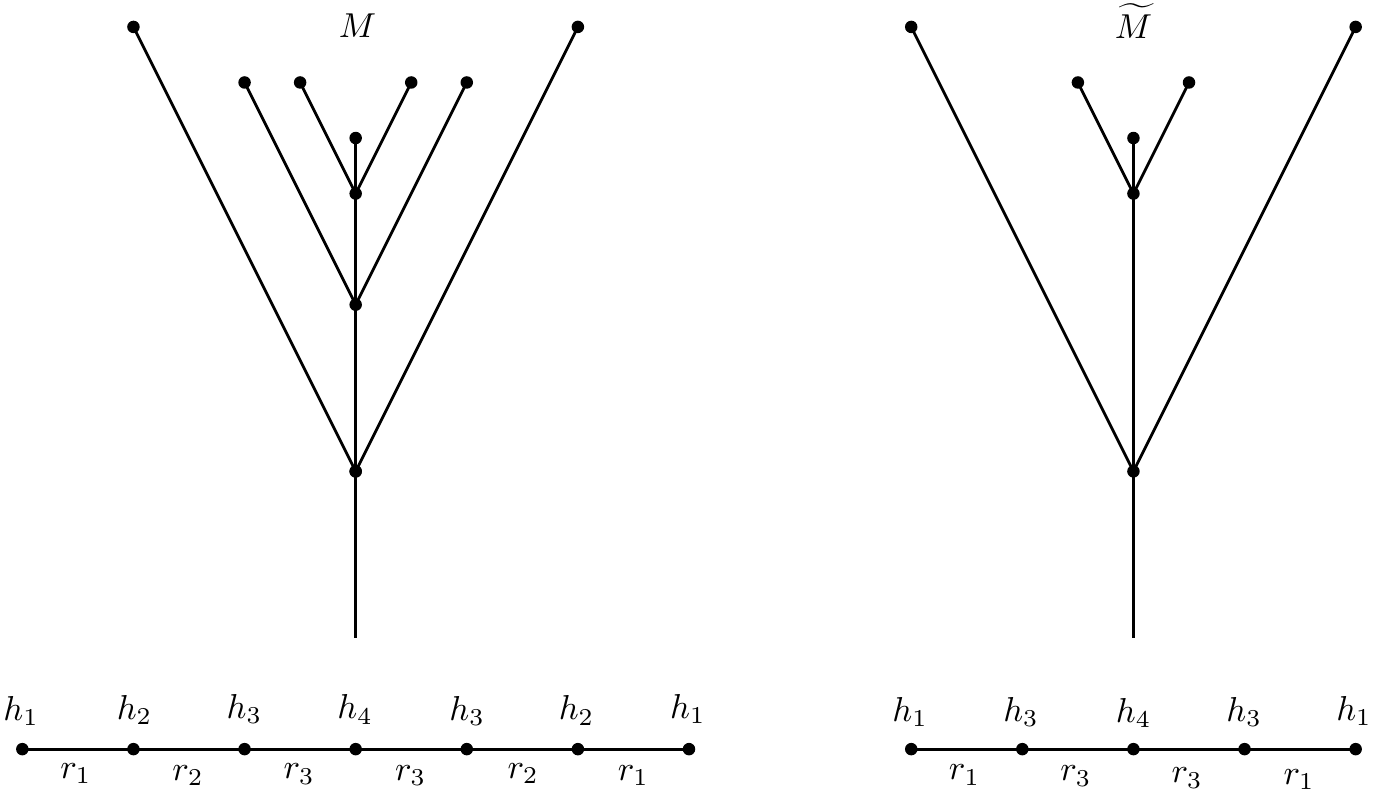}
\caption{Weakly monotone root (left) and its locally equivalent subroot (right). Gradings of cells are marked in the skeleta. Here $i = 2$.} \label{fig:16}
\end{figure}
With Lemma~\ref{lem:4.1} in hand, we now prove the central result of this section:

\begin{theorem}
\label{thm:4.2}
Let $M = M(h_1, r_1; \ldots; h_n, r_n)$ be a monotone root. Then we have the local equivalence
\[
M = \left(\sum_{i=1}^n M(h_i, r_i)\right)- \left(\sum_{i=1}^{n-1} M(h_{i+1}, r_i)\right).
\]
\end{theorem}
\begin{proof}
We proceed by induction on $n$. If $n = 1$, then $M$ is already a type-one root, and there is nothing to prove. Otherwise, we have the local equivalence
\[
M = M + M(h_2, r_1) - M(h_2, r_1)
\]
trivially obtained by taking the tensor product of $M$ with $M(h_2, r_1)$ and its dual. We now apply Theorem~\ref{thm:3.5} to $X = M$ and $Y = M(h_2, r_1)$, with the swapping parameters $a = b = 1$. This yields the valid swapped pair 
\[
X' = M(h_1, r_1)
\]
and
\[
Y' = M(h_2, r_1; h_2, r_2; \ldots; h_n, r_n),
\]
as displayed in Figure~\ref{fig:17}. Note that in this case the subgraph $S_Y$ is empty. We thus have the chain of local equivalences
\begin{align*}
M &= M + M(h_2, r_1)- M(h_2, r_1) \\
&= M(h_1, r_1) + M(h_2, r_1; h_2, r_2; \ldots; h_n, r_n) - M(h_2, r_1) \\
&= M(h_1, r_1) + M(h_2, r_2; \ldots; h_n, r_n) - M(h_2, r_1),
\end{align*}
where in the last line we have applied Lemma~\ref{lem:4.1} to obtain a monotone root of type $n-1$. This establishes the inductive step and completes the proof. See Figure~\ref{fig:17}.
\end{proof}
\begin{figure}[h!]
\center
\includegraphics[scale=0.8]{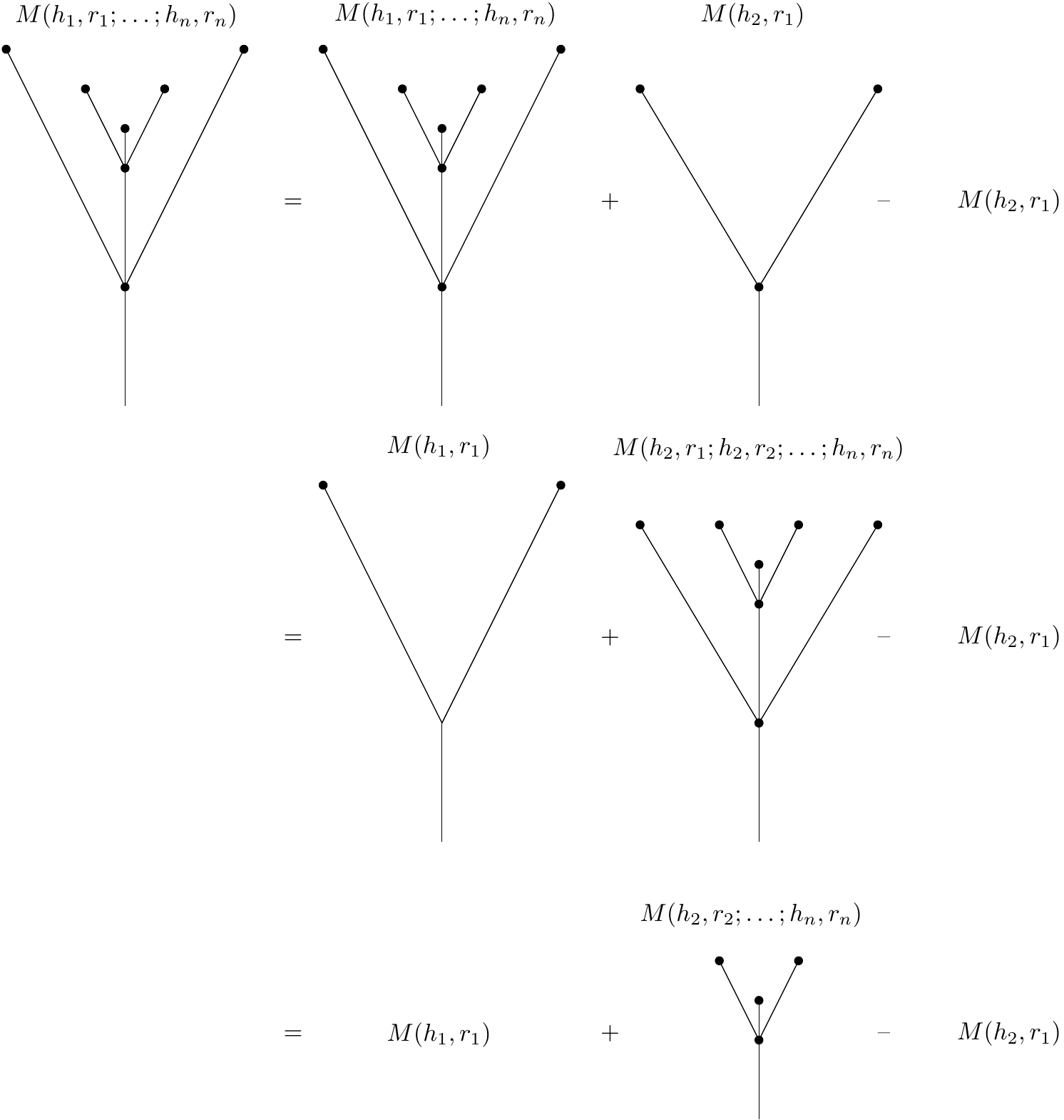}
\caption{Proof of Theorem~\ref{thm:4.2}.} \label{fig:17}
\end{figure}
\noindent
Note that in view of Theorems~\ref{thm:2.3} and \ref{thm:2.8}, this immediately implies that the $\inv$-complex of any connected sum of AR manifolds is locally equivalent to a linear combination of type-one roots.

Now let $\mathcal{M} = \{M_i\}_{i \geq 1}$ be a family of type-one roots such that
\[
\tilde{\delta}(M_i) = 2i
\]
for all $i \geq 1$. If this holds, then we call $\mathcal{M}$ a \textit{good family of roots}. If $\mathcal{M}$ is any good family of roots, then it is clear that every (nontrivial) type-one root is equal to an element of $\mathcal{M}$ up to a grading shift by some rational number. It follows that if $Y$ is any connected sum of AR manifolds, then $h(Y)$ is equal to a linear combination of type-one roots drawn from $\mathcal{M}$, together with an overall grading shift:
\begin{equation}\label{eqn:5}
h(Y) = \left(\sum_i c_i M_i\right)[\Delta].
\end{equation}
It is not hard to alter the proof of \cite[Theorem 1.7]{DM} to show that this decomposition is unique, as follows:

\begin{lemma}\cite[Theorem 1.7]{DM}\label{lem:4.3}
Let $\mathcal{M} = \{M_i\}$ be a good family of roots. If we have a local equivalence
\[
\left(\sum_i c_i M_i\right)[\Delta] = \left(\sum_i c_i' M_i\right)[\Delta'],
\]
then the linear combinations of the $M_i$ on either side are identical and $\Delta = \Delta'$.
\end{lemma}
\begin{proof}
After appropriately moving factors around, we obtain a local equivalence
\[
\left(\sum_{i\in I} a_i M_i\right)[\Delta] = \left(\sum_{j \in J} a_j M_j\right)[\Delta'],
\]
with all the $a_i, a_j > 0$ and $I \cap J = \emptyset$. According to \cite[Corollary 1.4]{DM}, if $M = \sum_i M_i$ is a sum of type-one roots (with the same orientation), then
\[
\du(M) - \dl(M) = \max_i \tilde{\delta}(M_i).
\]
(Note that here $\du(M) = d(M)$ by \cite[Theorem 1.3]{DM}.) Let $I_{\Delta}$ (respectively, $I_{\Delta'}$) be the graded root consisting of a single $\ff[U]$-tower starting in grading $-\Delta$ (respectively, $-\Delta'$). Then it is easily seen that applying a grading shift by $\Delta$ is equivalent to tensoring with $I_{\Delta}$. Viewing $I_\Delta$ and $I_{\Delta'}$ as degenerate type-one roots, applying \cite[Corollary 1.4]{DM} to both sides of the above local equivalence shows that in fact we must have $I = J = \emptyset$. This is only possible if the two original linear combinations are the same. We then additionally see that $I_\Delta = I_{\Delta'}$, showing that $\Delta = \Delta'$, as desired.
\end{proof}

There are two obvious choices for good families of roots. Let $\mathcal{X} = \{X_i\}$ and $\mathcal{Y} = \{Y_i\}$ be given by  
\[
X_i = M(0, -2i)
\]
and
\[
Y_i = M(2i, 0).
\]
The family $\mathcal{Y}$ has the benefit that it is explicitly realized by the Brieskorn homology spheres $\Sigma(p, 2p-1, 2p+1)$ for $p \geq 3 $ odd. More precisely, it is a consequence of \cite[Theorem 1.8]{Stoffregen2} that for such $p$,
\[
h(\Sigma(p, 2p-1, 2p+1)) = M(p-1, 0),
\]
as displayed in Figure~\ref{fig:1}. Hence as $p$ varies, we obtain all of the elements of $\mathcal{Y}$. This immediately yields a proof of Theorem~\ref{thm:1.1}:
\begin{proof}[Proof of Theorem~\ref{thm:1.1}]
Since $h(\Sigma(2, 3, 5)) = I_2$, we can effect any even grading shift with an appropriate multiple of $h(\Sigma(2, 3, 5))$. Noting that $\Sigma(2, 3, 5)$ and the $\Sigma(p, 2p-1, 2p+1)$ are themselves Seifert fibered, the claim then follows from Theorem~\ref{thm:4.2} (together with Theorems~\ref{thm:2.3} and \ref{thm:2.8}) and Lemma~\ref{lem:4.3}.
\end{proof}

We now make a more careful study of the grading shift in the decomposition (\ref{eqn:5}). Fix any good family of roots $\mathcal{M} = \{M_i\}$, and let $Y$ be a linear combination of AR plumbed manifolds equipped with a self-conjugate $\spinc$-structure $\s$. We define $\Delta_{\mathcal{M}}(Y, \s) \in \mathbb{Q}$ to be the grading shift associated to $h(Y, \s)$ when using the family $\mathcal{M}$ in (\ref{eqn:5}), so that
\[
h(Y, \s) = \left( \sum_i c_iM_i \right)[\Delta_{\mathcal{M}}(Y, \s)].
\]
In view of Lemma~\ref{lem:4.3}, $\Delta_{\mathcal{M}}(Y, \s)$ is well-defined and descends to a homomorphism from the subgroup of $\Inv_{\Q}$ spanned by (standard complexes of) graded roots to $\Q$. More precisely, the grading shift $\Delta$ of (\ref{eqn:5}) clearly changes sign under orientation reversal and is additive under connected sum. Moreover, by Lemma~\ref{lem:4.3}, $\Delta_{\mathcal{M}}(Y, \s)$ depends only on the local equivalence class $h(Y, \s)$. 

Now let us consider the specific families $\mathcal{X}$ and $\mathcal{Y}$. If $M$ is a single monotone root, then we may rearrange the decomposition of Theorem~\ref{thm:4.2} to read
\begin{align*}
M &= \left(\sum_{i=1}^n M(0, r_i-h_i)[-h_i] \right) - \left(\sum_{i=1}^{n-1} M(0, r_i- h_{i+1})[-h_{i+1}]\right) \\
&= \left(\sum_{i=1}^n M(0, r_i-h_i) - \sum_{i=1}^{n-1} M(0, r_i- h_{i+1}) \right)[-h_1].
\end{align*}
Hence in the case of a single monotone root, the overall grading shift for the family $\mathcal{X}$ is given by $-h_1$. Meanwhile,
\begin{align*}
M &= \left(\sum_{i=1}^n M(h_i-r_i, 0) [-r_i]\right) - \left(\sum_{i=1}^{n-1} M(h_{i+1}-r_i, 0)[-r_i]\right) \\
&= \left(\sum_{i=1}^n M(h_i - r_i, 0) - \sum_{i=1}^{n-1} M(h_{i+1} - r_i, 0)\right)[-r_n].
\end{align*}
Thus, the overall grading shift if we use $\mathcal{Y}$ is given by $-r_n$. We now quote the following result from \cite{DM} (see also \cite{Stoffregen}, \cite{Dai}):

\begin{theorem}\cite[Section 8]{DM}
\label{thm:4.4}
Let $Y$ be an AR plumbed three-manifold and let $\s$ be a self-conjugate $\spinc$-structure on $Y$. If $M(h_1, r_1; \ldots; h_n, r_n)$ is the monotone root parameterizing the local equivalence class of $h(Y, \s)$, then $h_1 = d(Y, \s)$ and $r_n = - 2\bar{\mu}(Y, \s)$.
\end{theorem}

\noindent
It follows that if $Y$ is a single AR plumbed three-manifold, then
\[
\Delta_{\mathcal{X}}(Y, \s) = - d(Y, \s)
\]
and
\[
\Delta_{\mathcal{Y}}(Y, \s) = 2\bar{\mu}(Y, \s).
\]
Consider the former. Since we already know that the $d$-invariant is a homomorphism, the fact that $\Delta_{\mathcal{X}}$ is also a homomorphism implies that $\Delta_{\mathcal{X}}(Y, \s) = - d(Y, \s)$ for any linear combination of AR plumbed manifolds. A similar argument for $\Delta_\mathcal{Y}$ yields a proof of Theorem~\ref{thm:1.2}:

\begin{proof}[Proof of Theorem~\ref{thm:1.2}]
We have already proven the existence and the uniqueness parts of the decomposition. It thus remains to establish the equality $\Delta = \Delta_{\mathcal{Y}} = 2\bar{\mu}(Y, \s)$. It is easily seen from the definitions that the Neumann-Siebenmann invariant changes sign under orientation reversal and is additive under connected sum; that is, 
\[
\bar{\mu}(-Y, \s) = -\bar{\mu}(Y, \s)
\]
and
\[
\bar{\mu}(Y_1 \# Y_2, \s_1 \# \s_2) = \bar{\mu}(Y_1, \s_1) + \bar{\mu}(Y_2, \s_2).
\]
Even though we do not \textit{a priori} know that $\bar{\mu}$ is a homology cobordism invariant, we do know that $\Delta_\mathcal{Y}$ is also additive under connected sum and changes sign under orientation reversal. Hence the fact that $\Delta_{\mathcal{Y}}(Y, \s)$ coincides with $2\bar{\mu}(Y, \s)$ for each individual AR plumbed manifold establishes the equality in general for connected sums.\footnote{Note that if $Y = Y_1 \# \cdots \# Y_k$ is a linear combination of plumbed manifolds, then any self-conjugate $\spinc$-structure $\s$ on $Y$ is of the form $\s = \s_1 \# \cdots \# \s_k$, where each $\s_i$ is a self-conjugate $\spinc$-structure on $Y_i$.} The fact that $\Delta_\mathcal{Y}$ depends only on the local equivalence class $h(Y, \s)$ then implies the same for $\bar{\mu}(Y, \s)$.

Finally, note that writing the decomposition of Theorem~\ref{thm:4.2} in order of decreasing $\tilde{\delta}$ yields an alternating sum in which the leading basis element appears with a coefficient of $+1$. Since this expression is invariant under homology cobordism, this (together with the fact that orientation reversal corresponds to multiplication by $-1$) establishes the last part of Theorem~\ref{thm:1.2}.
\end{proof}


\section{Involutive Floer Correction Terms}
\label{sec:5}
We now turn to a computation of $\du$ and $\dl$ for linear combinations of AR manifolds. In \cite{DM} this was done for connected sums of AR manifolds with the same orientation, but it turns out that allowing mixed orientations makes the calculation significantly more difficult. Our main advantage here will be to observe that in light of Theorem~\ref{thm:1.2}, it suffices to consider the case when all of the summands are of projective type. To this end, we begin by defining a class of geometric complexes that will be useful for describing products of type-one roots.

Let $B^n$ be the standard $n$-dimensional unit ball, and let $J_0$ be the usual involution on $B^n$ given by reflection through the origin. There is an obvious $J_0$-equivariant cellular decomposition of $B^n$ which consists of exactly one symmetric pair of cells in each dimension $0$ through $n-1$, and a single $J_0$-invariant cell in dimension $n$. We denote the cells lying in dimension $i$ by $e_i$ and $J_0e_i$, with the understanding that $e_n = J_0e_n$. The cellular boundary operator is given by 
\[
\partial(e_i) = \partial(J_0e_i) = e_{i-1} + J_0e_{i-1},
\]
with the convention that $e_i = J_0e_i = 0$ for $i < 0$. See Figure~\ref{fig:18}.

Now let $d \in \mathbb{Q}$ and let $(\Delta_1, \Delta_2, \ldots, \Delta_n)$ be a sequence of non-negative even integers. We associate a geometric complex to these parameters as follows. The skeleton of our complex is defined to be the above cellular decomposition of $B^n$. The function $\gr$ is defined by setting
\[
\gr(e_0) = \gr(J_0e_0) = d
\]
and
\[
\gr(e_i) = \gr(J_0e_i) = d - \left(\sum_{j = 1}^i \Delta_j\right)
\]
for $1 \leq i \leq n$. Note that this means the differential on our complex is given by
\[
\partial(e_i) = \partial(J_0e_i) = U^{\Delta_i/2}(e_{i-1} + J_0e_{i-1}).
\]
We call such a complex a \textit{spherical complex} and denote it by 
\[
S = S(d, n; \Delta_1, \ldots, \Delta_n).
\] 
See Figure~\ref{fig:18}. 
\begin{figure}[h!]
\center
\includegraphics[scale=1]{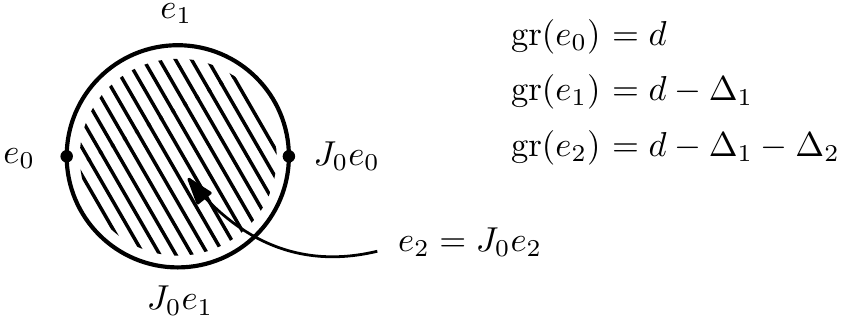}
\caption{Spherical complex $S(d, n; \Delta_1, \ldots, \Delta_n)$ with $n = 2$. } \label{fig:18}
\end{figure}
\noindent

\begin{remark}
Note that spherical complexes also arise naturally in the setting of $\mathrm{Pin}(2)$-equivariant Floer homology.  In particular, a connected sum of (negatively oriented) Seifert integer homology spheres $\sum_{i=1}^n \Sigma_i$, with $\tilde{\delta}_i:=\tilde{\delta}(\Sigma_i)$, is chain locally equivalent to the unreduced suspension of a certain free $\mathrm{Pin}(2)$-space $X$, with a natural projection map $\pi\co X \to \partial S(0,n;\tilde{\delta}_1,\dots,\tilde{\delta}_n)$ whose fibers are determined by the $\tilde{\delta}_i$ (cf.\ \cite[Section 4]{Stoffregen2}). 
\end{remark}

\begin{lemma}\label{lem:5.1}
Let $M_1, \ldots, M_n$ be a family of type-one roots parameterized by $M_i = (h^i, r^i)$ for $1 \leq i \leq n$. Without loss of generality, suppose that
\[
\tilde{\delta}(M_1) \geq \tilde{\delta}(M_2) \geq \cdots \geq \tilde{\delta}(M_n).
\]
Then we have the local equivalence
\[
M_1 + \cdots + M_n = S(d, n; \tilde{\delta}(M_1), \ldots, \tilde{\delta}(M_n)),
\]
where $d = h^1 + \cdots + h^n$.
\end{lemma}
\begin{proof}
We proceed by induction on $n$. If $n = 1$, then the spherical complex in question is the same as the standard complex defined in Section~\ref{sec:2.3}. To establish the inductive step, it suffices to prove the local equivalence 
\[
M(h, r) + S(d, n; \Delta_1, \ldots, \Delta_n) = S(d+h, n+1; h-r, \Delta_1, \ldots, \Delta_n),
\]
where without loss of generality we have assumed that $h - r \geq \Delta_i$ for all $\Delta_i$. Denote the cells of $M(h, r)$ by $v$, $J_0v$, and $\alpha = J_0\alpha$. For convenience, we denote the cells of the left-hand spherical complex by $x_i$ for $0 \leq i \leq n$, and the cells of the right-hand spherical complex by $y_i$ for $0 \leq i \leq n+1$. Note that $x_n = J_0x_n$ and $y_{n+1} = J_0y_{n+1}$. 

As usual, we apply Lemma~\ref{lem:3.3}. The skeleta of the product complex on the left and the spherical complex on the right are displayed in Figure~\ref{fig:19}. To define a cellular map $f$ from the left- to the right-hand side, we send
\begin{align*}
&f(v \times x_0) = f(v \times J_0x_0) = y_0, \\
&f(v \times x_i) = f(v \times J_0x_i) = 0 \text{ for } 1 \leq i \leq n, \text{ and} \\
&f(\alpha \times x_i) = y_{i+1} \text{ for } 0 \leq i \leq n,
\end{align*}
extending $J_0$-equivariantly. It is easily checked that this preserves the cellular differential by computing (for example)
\[
\partial f(\alpha \times x_i) = \partial y_{i+1} = y_i + J_0y_i
\]
and
\[
f\partial (\alpha \times x_i) = f( \alpha \times (x_{i-1} + J_0x_{i-1}) + (v + J_0v) \times x_i ) = y_i + J_0y_i.
\]
To see that $f$ satisfies the lifting condition, we check that $f$ preserves $\gr$ wherever it is nonzero:
\[
\gr(v \times x_0) = \gr(v \times J_0x_0) = h + d = \gr(y_0)
\]
and
\[
\gr(\alpha \times x_i) = r + \left(d - \sum_{j=1}^i \Delta_j\right) = (d + h) - \left( (h-r) + \sum_{j=1}^i \Delta_j\right) = \gr(y_{i+1}).
\]
This gives the local map in one direction.

To define the map in the other direction, set
\begin{align*}
&g(y_i) = v \times x_i + \alpha \times J_0x_{i-1} \text{ for } 0 \leq i \leq n, \text{ and}\\
&g(y_{n+1}) = \alpha \times x_n,
\end{align*}
again extending $J_0$-equivariantly. Note that $g(y_0) = v \times x_0$. We check that this preserves the cellular differential (over $\mathbb{Z}/2\mathbb{Z}$) by computing
\begin{align*}
\partial g(y_i) &= \partial(v \times x_i + \alpha \times J_0x_{i-1}) \\
&= v \times (x_{i-1} + J_0x_{i-1}) + (v + J_0v) \times J_0x_{i-1} + \alpha \times (x_{i-2} + J_0 x_{i-2}) \\
&= v \times x_{i-1} + J_0v \times J_0x_{i-1} + \alpha \times x_{i-2} + \alpha \times J_0 x_{i-2}
\end{align*}
and
\begin{align*}
g\partial (y_i) &= g(y_{i-1} + J_0y_{i-1}) \\
&= v\times x_{i-1} + \alpha \times J_0x_{i-2} + J_0v\times J_0x_{i-1} + \alpha \times x_{i-2}
\end{align*}
for $1 \leq i \leq n$. A similar computation holds for $y_{n+1}$, utilizing the fact that $J_0x_n = x_n$. To see that $g$ satisfies the lifting condition, consider the two summands $v \times x_i$ and $\alpha \times J_0x_{i-1}$ of $g(y_i)$. The grading of the second term follows from the computation of the previous paragraph and is equal to $\gr(y_i)$. The grading of the first term is given by
\[
\gr(v \times x_i) = h + d - \left( \sum_{j = 1}^i \Delta_j \right).
\]
This is greater than or equal to $\gr(y_i)$, since $h-r \geq \Delta_i$ for all $\Delta_i$. Applying Lemma~\ref{lem:3.3} establishes the local equivalence and completes the proof. We leave it to the reader to give geometric interpretations to $f$ and $g$; see Figure~\ref{fig:19}.
\end{proof}
\begin{figure}[h!]
\center
\includegraphics[scale=1]{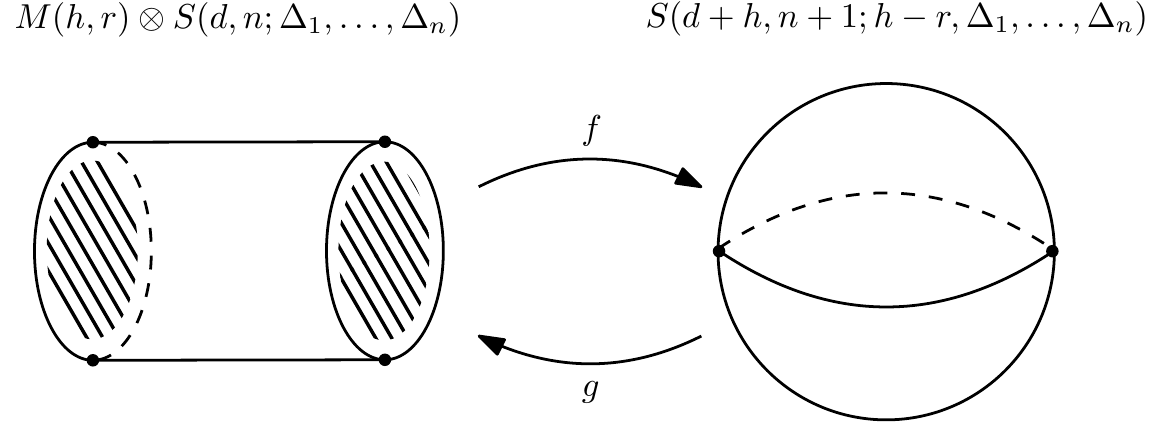}
\caption{Skeleta of the complexes considered in the proof of Lemma~\ref{lem:5.1}.} \label{fig:19}
\end{figure}

It follows from Lemma~\ref{lem:5.1} that to understand linear combinations of type-one roots, it suffices to consider the case of a single spherical complex, tensored with the inverse of another spherical complex. Denote the inverse complex of $S$ by $S^\vee = S(d, n; \Delta_1, \ldots, \Delta_n)^\vee$. This may be explicitly described as follows. Dualizing the skeleton of $S$ yields a complex with $2n+1$ generators, which we denote by $e_i^\vee$ and $J_0e_i^\vee$ for $1 \leq i \leq n$ (with the understanding that $e_n^\vee = J_0e_n^\vee$). The differential on this complex is given by  
\begin{align*}
&\partial (e_n^\vee) = 0, \\
&\partial (e_{n-1}^\vee) = \partial (J_0e_{n-1}^\vee) = e_n^\vee, \text{ and} \\
&\partial (e_i^\vee) = \partial(J_0e_i^\vee) = e_{i+1}^\vee + J_0e_{i+1}^\vee \text{ for } 0 \leq i < n-1.
\end{align*}
We refer to this data as the skeleton of $S^\vee$, even though technically we have not given it an interpretation as a cellular complex in the usual sense. (See however Remark~\ref{rem:5.2}.) An explicit computation (simply by taking the kernel and image of $\partial$) shows that the homology of the skeleton is isomorphic to $\ff$ and is generated by $e_0^\vee + J_0e_0^\vee$.

We define $\gr$ on the skeleton of $S^\vee$ to be the negative of $\gr$ on $S$ after dualizing, so that
\[
\gr(e_0^\vee) = \gr(J_0e_0^\vee) = - d
\]
and
\[
\gr(e_i^\vee) = \gr(J_0e_i^\vee) = -d + \sum_{j = 1}^i \Delta_j.
\]
We give the dual skeleton tensored with $\ff[U]$ a grading by subtracting the cellular dimension, so that $e_i^\vee$ has grading $\gr(e_i^\vee) - i$. For clarity, we refer to this quantity (and similarly its counterpart $\gr(e_i) + i$ in the non-dualized case) as the \textit{chain complex grading}, in order to distinguish it from $\gr$. As before, we declare $U$ to be of degree $-2$. Then the obvious prescription
\[
\partial (\square_i^\vee) = \sum _{\square_{i+1}^\vee \in \text{ bdry} (\square_i^\vee)} U^{(\gr(\square_{i+1}^\vee)-\gr(\square_i^\vee))/2}\square_{i+1}^\vee
\]
has degree $-1$ and turns this into a graded chain complex, which is evidently the inverse of $S$ in the sense of \cite[Section 8]{HMZ}. Explicitly, we have
\begin{align*}
&\partial (e_n^\vee) = 0, \\
&\partial (e_{n-1}^\vee) = \partial (J_0e_{n-1}^\vee) = U^{\Delta_n/2}e_n^\vee, \text{ and} \\
&\partial (e_i^\vee) = \partial(J_0e_i^\vee) = U^{\Delta_{i+1}/2}(e_{i+1}^\vee + J_0e_{i+1}^\vee) \text{ for } 0 \leq i < n-1.
\end{align*}

\begin{remark}\label{rem:5.2}
It is possible to give the dual complex $S^\vee$ a geometric interpretation by using a \textit{relative} cell complex, as in Figure~\ref{fig:20}. More precisely, we can give $B^n$ a $J_0$-equivariant decomposition in such a way so that the skeleton of $S^\vee$ is identified with the relative cellular complex $(B^n, S^{n-1})$. (This is simply a manifestation of Poincar\'e duality applied to the manifold-with-boundary $B^n$.) This picture can be used to gain some geometric intuition for Lemmas~\ref{lem:5.3} and \ref{lem:5.4} below. Note that the homology of the skeleton in this case is still isomorphic to $\ff$, even though its generator $[x]$ no longer has dimensional grading zero.
\end{remark}

\begin{figure}[h!]
\center
\includegraphics[scale=1]{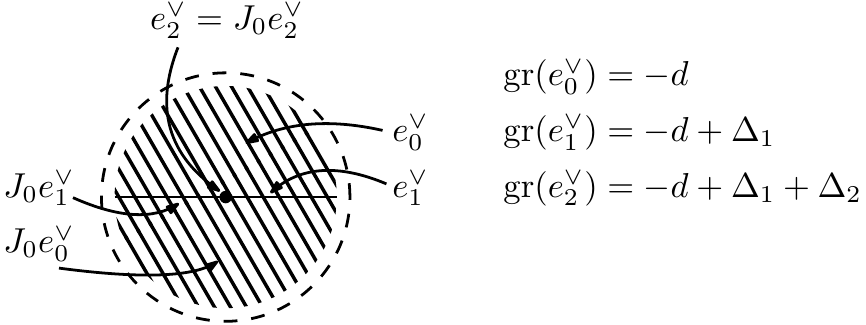}
\caption{Dual complex $S^\vee(d, n; \Delta_1, \ldots, \Delta_n)$ with $n = 2$.} \label{fig:20}
\end{figure}

We are now in a position to carry out our computation of $\du$ and $\dl$. We will henceforth suppress writing down powers of $U$, so that when we write a sum of generators $\Sigma$ (of the same dimension), we mean that each term of $\Sigma$ is implicitly multiplied by the appropriate power of $U$ so that its (chain complex) grading is equal to that of the lowest-grading generator in $\Sigma$. By $\gr(\Sigma)$ we thus mean the minimum of $\gr$ over the terms appearing in $\Sigma$. Similarly, when we write that two sums of generators are equal to each other, we mean equality after multiplying one or both sides by a sufficiently high power of $U$. 

We proceed by explicitly understanding the mapping cone of $S_1 \otimes S_2^\vee$. Recall from Section~\ref{sec:2.1} that if $(C, \inv)$ is an $\inv$-complex, then its mapping cone is given by
\[
C_* \xrightarrow{\phantom{o} Q (1+\inv) \phantom{o}} Q \cdot C_* [-1].
\]
Explicitly, this consists of the complex $C[-1] \otimes \ff[Q]/(Q^2)$, together with the total differential 
\[
\partial_{tot} = \partial + Q \cdot (1 + \inv),
\] 
where $\partial$ is the differential on $C$. Thus elements of $C$ in the mapping cone which are not decorated by a $Q$ have grading one higher than their usual chain complex grading in $C$, while multiplying by $Q$ shifts grading by $-1$. Where confusion with other gradings is possible, we will refer to this as the \textit{total grading} or the \textit{mapping cone grading}. We also remind the reader of our convention from Section~\ref{sec:2.1} regarding the definition of $\du$ and $\dl$ for $\inv$-complexes; these are shifted by two relative to their usual definition. See (\ref{eqn:dudl}).

Recall that in sufficiently low gradings, the involutive homology $\HFIm$ consists entirely of two $U$-nontorsion towers, one of which is taken to the other via multiplication by $Q$. We refer to these as the $\du$- and $\dl$-towers depending on their mod $2$ gradings relative to $\tau$. (The $\du$-tower is the tower in the image of $Q$.) It is easily checked that an element $x \in \HFIm$ satisfies the conditions in the definition of $\du$ or $\dl$ if and only if sufficiently high $U$-powers of $x$ lie in the $\du$- or $\dl$-tower, respectively. The following lemma gives preferred representatives of these elements in $\CFIm(S_1 \otimes S_2^\vee)$.

For brevity, we now stop writing the product symbol $\times$ and the subscript on $J_0$. 

\begin{lemma}\label{lem:5.3}
Let $S_1 = S(d_1, m; A_1, \ldots, A_m)$ and $S_2 = S(d_2, n; B_1, \ldots, B_n)$. Denote the generators of $S_1$ by $x_i$ and the generators of $S_2$ by $y_i$. In sufficiently low gradings, elements in the $\du$-tower of $\HFIm(S_1 \otimes S_2^\vee)$ may be represented by $U$-powers of
\[
Q \cdot x_0(y_0^\vee + Jy_0^\vee).
\]
Similarly, in sufficiently low gradings, elements in the $\dl$-tower may be represented by $U$-powers of
\[
x_0 (y_0^\vee + Jy_0^\vee) + Q \cdot x_1(y_0^\vee + Jy_0^\vee).
\]
Recall, in the latter equation, the convention that we implicitly homogenize by multiplying $x_0(y_0^\vee+Jy_0^\vee)$ by an appropriate power of $U$. 
\end{lemma}
\begin{proof}
It is straightforward to check that $\partial_{tot}$ applied to both of the above expressions is zero. To prove the lemma, consider the term $x_0 (y_0^\vee + Jy_0^\vee)$, viewed as a cycle in the usual homology of $S_1 \otimes S_2^\vee$. Note that this element generates the homology of the skeleton of $S_1 \otimes S_2^\vee$, which is $\ff$ by the K\"unneth formula. By the same argument as in Lemma~\ref{lem:3.3}, we then see that $x_0 (y_0^\vee + Jy_0^\vee)$ is $U$-nontorsion in $H_*(S_1 \otimes S_2^\vee)$. An easy argument using the mapping cone exact sequence (see \cite[Proposition 4.6]{HMinvolutive}) shows that the $U$-powers of $Q \cdot x_0(y_0^\vee + Jy_0^\vee)$ must then generate the $\du$-tower. The claim for $x_0 (y_0^\vee + Jy_0^\vee) + Q \cdot x_1(y_0^\vee + Jy_0^\vee)$ follows from the fact that in sufficiently low gradings, multiplication by $Q$ is an isomorphism (on the level of homology) from the $\dl$-tower onto the $\du$-tower.
\end{proof}

We now establish lower bounds for $\du(S_1 \otimes S_2^\vee)$ and $\dl(S_1 \otimes S_2^\vee)$. Our strategy for doing this will be to exhibit explicit elements in the mapping cone of $S_1 \otimes S_2^\vee$ which are homologous (after multiplying by sufficient powers of $U$) to the generators of Lemma~\ref{lem:5.3}. Since this means that these elements satisfy the relevant conditions in the definitions of $\du$ and $\dl$, taking their (mapping cone) gradings lead to lower bounds for $\du$ and $\dl$. As with Remark~\ref{rem:5.2}, it is possible to frame this argument in a more geometric fashion, but since this will not be needed elsewhere in the paper, we take a more direct route and simply present the algebraic details.

\begin{lemma}\label{lem:5.4}
Let $S_1 = S(d_1, m; A_1, \ldots, A_m)$ and $S_2 = S(d_2, n; B_1, \ldots, B_n)$. Define
\begin{align*}
P_i &= \sum_{j=1}^i B_j - \sum_{j=1}^i A_j \text{ for } 0 \leq i \leq \min(m, n),\\
Q_i &= \sum_{j=1}^i B_j - \sum_{j=1}^{i + 1} A_j \text{ for } 0 \leq i \leq \min(m-1, n),\text{ and}\\
R_i &= \sum_{j=1}^i B_j - \sum_{j=1}^{i -1 } A_j \text{ for } 1 \leq i \leq \min(m + 1, n).
\end{align*}
Note that $P_0 = 0$ and $Q_0 = -A_1$. Then
\begin{align*}
\du(S_1 \otimes S_2^\vee) \geq d_1 - d_2 + \max \{ P_0, &\min(R_1, P_1), \\
&\min(R_1, R_2, P_2), \\
&\ \ \ \ \ \ \ \ \ \ \ \ \ \vdots \\
&\min(R_1, R_2, \ldots, R_{\min(m+1, n)}, P_{\min(m+1, n)})\},
\end{align*}
with the understanding that if $\min(m+1, n) = m + 1$, the $P_{\min(m+1, n)}$ in the last line should be deleted. Similarly,
\begin{align*}
\dl(S_1 \otimes S_2^\vee) \geq d_1 - d_2 + \max \{ &\min(P_0, Q_0), \\
&\min(P_0, P_1, Q_1), \\
&\ \ \ \ \ \ \ \ \ \ \ \ \ \vdots \\
&\min(P_0, P_1, P_2, \ldots, P_{\min(m,n)}, Q_{\min(m,n)})\},
\end{align*}
with the understanding that if $\min(m, n) = m$, the $Q_{\min(m, n)}$ in the last line should be deleted.
\end{lemma}
\begin{proof}
By shifting gradings, it is clear that we may assume $d_1 = d_2 = 0$. For convenience, define
\[
\Sigma^u_k = \sum_{i = 0}^{k-1} \left(x_{i-1} Jy_i^\vee + Jx_{i-1} y_i^\vee\right)
\]
for $0 \leq k \leq \min(m+1, n)$, and
\[
\Sigma^l_k = \sum_{i = 0}^{k-1} \left(x_i Jy_i^\vee + Jx_i y_i^\vee\right)
\]
for $0 \leq k \leq \min(m, n)$. (Here we set $x_{-1}=0$.) Note that the total grading of an element $x_{i_1}y_{i_2}$ in the mapping cone of $S_1 \otimes S_2^\vee$ is given by
\[
i_1 - i_2 + \gr(x_{i_1}y_{i_2}) + 1 = i_1 - i_2 - \sum_{j=1}^{i_1} A_j + \sum_{j=1}^{i_2} B_j + 1,
\]
since the chain complex grading of $x_{i_1}y_{i_2}$ is $i_1 - i_2 + \gr(x_{i_1}y_{i_2})$. Similarly, the total grading of $Q \cdot x_{i_1}y_{i_2}$ is given by the above expression, minus one.

We begin with the desired lower bound for $\du$. To do this, we define a sequence of generators $U_k$ for all $0 \leq k \leq \min(m+1, n)$. There is some extra casework when $k = \min(m+1, n)$, so for the moment, let $0 \leq k < \min(m+1, n)$. Define
\[
U_k = \Sigma^u_k + x_{k-1}(y_k^\vee + Jy_k^\vee) + Q\cdot x_k(y_k^\vee + Jy_k^\vee).
\]
Note that $U_0$ is the generator of the $\du$-tower from Lemma~\ref{lem:5.3}. We claim that all of the $U_k$ are homologous to each other, and thus to $U_0$. To see this, let $0 \leq k < \min(m+1, n) - 1$. Then one can check that
\begin{align*}
U_k + U_{k+1} &= \left(x_{k-1} Jy_k^\vee + Jx_{k-1} y_k^\vee\right) + x_{k-1}(y_k^\vee + Jy_k^\vee) + x_{k}(y_{k+1}^\vee + Jy_{k+1}^\vee) \ + \\
&\ \ \ \ Q\cdot x_k(y_k^\vee + Jy_k^\vee) + Q\cdot x_{k+1}(y_{k+1}^\vee + Jy_{k+1}^\vee) \\
&= (x_{k-1} + Jx_{k-1})y_k^\vee + x_{k}(y_{k+1}^\vee + Jy_{k+1}^\vee) \ + \\
&\ \ \ \ Q\cdot x_k(y_k^\vee + Jy_k^\vee) + Q\cdot x_{k+1}(y_{k+1}^\vee + Jy_{k+1}^\vee) \\
&= \partial_{tot}(x_ky_k^\vee + Q \cdot x_{k+1}Jy_k^\vee),
\end{align*}
where in the first line we have canceled most of the terms in $\Sigma^u_{k+1}$ against $\Sigma^u_k$. We also define $U_k$ in the special case that $k = \min(m + 1, n)$, as follows. If $n < m + 1$, then we define $U_n$ by replacing all instances of $(y_k^\vee + Jy_k^\vee)$ in the usual definition of $U_k$ with $y_n^\vee$, so that
\[
U_n = \Sigma^u_n + x_{n-1}y_n^\vee + Q\cdot x_ny_n^\vee.
\]
This is simply to reflect the fact that $\partial y_{n-1}^\vee = y_n^\vee$, rather than $y_n^\vee + Jy_n^\vee$. It is easily checked that $U_n$ is still homologous to $U_{n-1}$ via essentially the same chain of equalities as above. If $m + 1 < n$, we instead define
\[
U_{m+1} = U_m + \partial_{tot}(x_my_m^\vee) = \Sigma^u_{m+1}+ x_m(y_{m+1}^\vee + Jy_{m+1}^\vee),
\]
where here we have used the fact that $Jx_m = x_m$. If $m + 1 = n$, then the $(y_{m+1}^\vee + Jy_{m+1}^\vee)$ in the expression for $U_{m+1}$ above should be replaced by $y_{m+1}^\vee$. This defines $U_k$ for all $0 \leq k \leq \min(m+1, n)$. The total gradings of the elements $U_k$ in the mapping cone are given by
\begin{align*}
&\text{total grading of }U_k = \min\{R_1, \ldots, R_k, P_k \} \text{ if } k \leq m, \text{ and} \\
&\text{total grading of }U_{m+1} = \min\{R_1, \ldots, R_{m+1} \} \text{ else}.
\end{align*}
This establishes the desired lower bound for $\du$.

We now turn to $\dl$. Let $0 \leq k < \min(m, n)$. Define
\[
L_k = \Sigma^l_k + x_k(y_k^\vee + Jy_k^\vee) + Q\cdot x_{k+1}(y_k^\vee + Jy_k^\vee).
\]
Note that $L_0$ is the generator of the $\dl$-tower from Lemma~\ref{lem:5.3}. We claim that all of the $L_k$ are homologous to each other, and thus to $L_0$. To see this, let $0 \leq k < \min(m, n) - 1$. Then
\begin{align*}
L_k + L_{k+1} &= (x_kJy_k^\vee + Jx_ky_k^\vee) + x_k(y_k^\vee + Jy_k^\vee) + x_{k+1}(y_{k+1}^\vee + Jy_{k+1}^\vee)\ + \\
&\ \ \ \ Q\cdot x_{k+1}(y_k^\vee + Jy_k^\vee) + Q\cdot x_{k+2}(y_{k+1}^\vee + Jy_{k+1}^\vee) \\
&= (x_k + Jx_k)y_k^\vee + x_{k+1}(y_{k+1}^\vee + Jy_{k+1}^\vee)\ + \\
&\ \ \ \ Q\cdot x_{k+1}(y_k^\vee + Jy_k^\vee) + Q\cdot x_{k+2}(y_{k+1}^\vee + Jy_{k+1}^\vee) \\
&= \partial_{tot}(x_{k+1}y_k^\vee + Q \cdot x_{k+2}Jy_k^\vee).
\end{align*}
We also define $L_k$ in the special case that $k = \min(m, n)$. As before, if $n < m$, we define $L_n$ by replacing all instances of $(y_k^\vee + Jy_k^\vee)$ with $y_n^\vee$:
\[
L_n = \Sigma^l_n + x_ny_n^\vee + Q\cdot x_{n+1}y_n^\vee.
\]
If $m < n$, we instead define
\[
L_m = L_{m-1} + \partial_{tot}(x_my_{m-1}^\vee) = \Sigma^l_m + x_m(y_m^\vee + Jy_m^\vee),
\]
where if $m = n$, we replace $y_m^\vee + Jy_m^\vee$ with $y_m^\vee$. The mapping cone gradings of the $L_k$ are given by
\begin{align*}
&\text{total grading of }L_k = \min\{P_0, \ldots, P_k, Q_k \} + 1 \text{ if } k < m, \text{ and} \\
&\text{total grading of }L_m = \min\{P_0, \ldots, P_m \} + 1 \text{ else}.
\end{align*}
This establishes the desired lower bound for $\dl$ and completes the proof.
\end{proof}

We now use the lower bound for $\du$ derived in Lemma~\ref{lem:5.4} to provide an upper bound for $\dl$ by interchanging the roles of $S_1$ and $S_2$. More precisely, since
\[
\dl(S_1 \otimes S_2^\vee) = - \du(S_1^\vee \otimes S_2),
\]
we obtain an upper bound for $\dl(S_1 \otimes S_2^\vee) $ by exchanging the roles of $S_1$ and $S_2$ in the lower bound for $\du(S_1 \otimes S_2^\vee)$ and multiplying through by $-1$. Swapping $A_i$ and $B_i$ sends $P_i$ to $-P_i$, $R_i$ to $-Q_{i-1}$, and interchanges $m$ and $n$. Applying this to the lower bound for $\du$ yields the expression 
\[
\min\{ P_0, \max(Q_0, P_1), \max(Q_0, Q_1, P_2), \ldots, \max(Q_0, \ldots, Q_{\min(m-1, n)}, P_{\min(m, n+1)}) \},
\]
with the understanding that if $\min(m, n+1) = n+1$, the final $P_{\min(m, n+1)}$ should be deleted. (Note that in the above expression, the $P_i$ and $Q_i$ are still are as defined in Lemma~\ref{lem:5.4} for $S_1 \otimes S_2^\vee$, and not the analogous definitions for $S_1^\vee \otimes S_2$.) We claim that this upper bound for $\dl(S_1 \otimes S_2^\vee)$ coincides with the lower bound derived in Lemma~\ref{lem:5.3}. This is a purely combinatorial lemma, which we prove below:
\begin{lemma}\label{lem:5.5}
Let $P_i$ and $Q_i$ be as in Theorem~\ref{thm:1.5}. Define
\[
S = \max \{ \min(P_0, Q_0), \min(P_0, P_1, Q_1), \ldots, \min(P_0, P_1, P_2, \ldots, P_{\min(m,n)}, Q_{\min(m,n)})\},
\]
where if $\min(m, n) = m$, we delete the final $Q_{\min(m,n)}$. Define
\[
T = \min\{ P_0, \max(Q_0, P_1), \max(Q_0, Q_1, P_2), \ldots, \max(Q_0, \ldots, Q_{\min(m-1, n)}, P_{\min(m, n+1)}) \},
\]
where if $\min(m, n+1) = n+1$, we delete the final $P_{\min(m, n+1)}$. Then $T \leq S$.
\end{lemma}
\begin{proof}
We proceed by casework on the value of $T$. Let $k$ be the least index for which $T$ is equal to the term
\[
T = \max(Q_0, \ldots, Q_{k-1}, P_k),
\]
where as usual the $P_k$ in the above expression may be deleted if $k = \min(m, n+1)$. (That is, we consider the first term in the expression for $T$ which achieves the desired minimum.) Suppose moreover that
\[
T = \max(Q_0, \ldots, Q_{k-1}, P_k) = Q_i
\]
for some $0 \leq i < k$. We claim that then $Q_i < P_j$ for all $j \leq i$. Indeed, by minimality of $k$, we have the strict inequality 
\[
Q_i < \max(Q_0, \ldots, Q_{j-1}, P_j)
\]
for all such $j$. Since we already know that $Q_0, \ldots, Q_{k-1}$ are bounded above by $Q_i$, this implies the claim. In particular, we thus see that 
\[
\min(P_0, \ldots, P_i, Q_i) = Q_i.
\]
Since this term appears in the above expression for $S$, we conclude that $T = Q_i \leq S$, as desired. Note that $i < k \leq m$, so there is no problem with the possible deletion of $Q_{\min(m,n)}$.

Now suppose instead that
\[
T = \max(Q_0, \ldots, Q_{k-1}, P_k) = P_k.
\]
Note that implicitly this means $k \leq n$. We claim that $P_k < P_j$ for all $j < k$. Indeed, as before, we have
\[
P_k < \max(Q_0, \ldots, Q_{j-1}, P_j) 
\]
for all $j < k$, so the fact that $Q_0, \ldots, Q_{k-1}$ are bounded above by $P_k$ implies the claim. In addition, for all $r > k$, we have 
\[
P_k \leq \max(Q_0, \ldots, Q_{r-1}, P_r)
\]
and hence
\[
P_k \leq \max(Q_k, \ldots, Q_{r-1}, P_r),
\]
again by the fact that $Q_0, \ldots, Q_{k-1}$ are bounded above by $P_k$. (As usual, we delete $P_r$ from the above expression if $r = \min(m, n+1)= n+1$.) There are now two subcases. First, suppose that for all $r > k$, the above maximum is given by $P_r$. Then we have that $P_k \leq P_j$ for all possible values of $j$. Note that implicitly in this case $P_{\min(m, n+1)}$ is not deleted, so we know that $m < n + 1$. Hence $\min(m, n) = m$, and the final term in the expression for $S$ is given by
\[
\min(P_0, P_1, \ldots, P_m) = P_k.
\]
This shows that $T \leq S$, as desired.

In the second subcase, let $r > k$ be the least index such that 
\[
\max(Q_k, \ldots, Q_{r-1}, P_r) = Q_s
\]
for some $k \leq s < r$. (That is, we consider the first value of $r > k$ for which the expression $\max(Q_k, \ldots, Q_{r-1}, P_r)$ is equal to some $Q_s$, rather than $P_r$.) Then $P_k \leq P_j$ for all $j < r$, and also $P_k \leq Q_s$. Thus
\[
\min(P_0, \ldots, P_s, Q_s) = P_k.
\]
Since this term appears in the expression for $S$, we conclude that $T \leq S$. This completes the proof.
\end{proof}

Putting everything together, we thus obtain:
\begin{proof}[Proof of Theorem~\ref{thm:1.5}]
According to Lemma \ref{lem:5.1}, we need only calculate the $\dl$-invariant of $S(0,m; 2s_1,\dots,2s_m)\otimes S(0,n; 2t_1,\dots,2t_n)^\vee$.  We obtain lower bounds for $\dl$ of the tensor product from Lemma~\ref{lem:5.4}, as well as an upper bound using duality applied to the lower bound for $\du$ in the same lemma.  Lemma~\ref{lem:5.5} shows that the upper and lower bounds are equal, giving the calculation of $\dl$ as in Theorem \ref{thm:1.5}. 
\end{proof}
\noindent
For the purposes of distinguishing linear combinations of AR manifolds, Theorem~\ref{thm:1.5} is of course overshadowed by the explicit decomposition afforded by Theorem~\ref{thm:1.2}. However, since $\du$ and $\dl$ are defined for all rational homology spheres, Theorem~\ref{thm:1.5} can be used to compare linear combinations of AR manifolds with other rational homology spheres whose involutive Floer correction terms have been computed by different methods. 


\section{Examples and Applications}
\label{sec:6}
In this section, we give some examples and applications of Theorems~\ref{thm:1.2} and \ref{thm:1.5}. We begin by illustrating the basis decomposition into $Y_i$ afforded by Theorem~\ref{thm:1.2}:

\begin{example}\label{ex:ex1}
Let $Y = \Sigma(5,8,13)$. The graded root associated to $h(Y)$ can be computed using the numerical semigroups algorithm of Can and Karakurt \cite{Can}, and the relevant monotone subroot can then be extracted by using the procedure described in \cite[Section 6]{DM}. We have drawn part of the graded root associated to $h(Y)$ on the left in Figure~\ref{fig:21}; the monotone subroot is displayed on the right. (The full graded root corresponding to $h(Y)$ has several other length-one branches occurring in lower gradings which we have suppressed for brevity.) Applying Theorem~\ref{thm:4.2}, we see that
\begin{align*}
h(\Sigma(5, 8, 13)) &= M(4, 0; 2, 2) \\
&= M(4, 0) + M(2, 2) - M(2, 0) \\
&= (Y_2 - Y_1)[-2].
\end{align*}
\begin{figure}[h!]
\center
\includegraphics[scale=1]{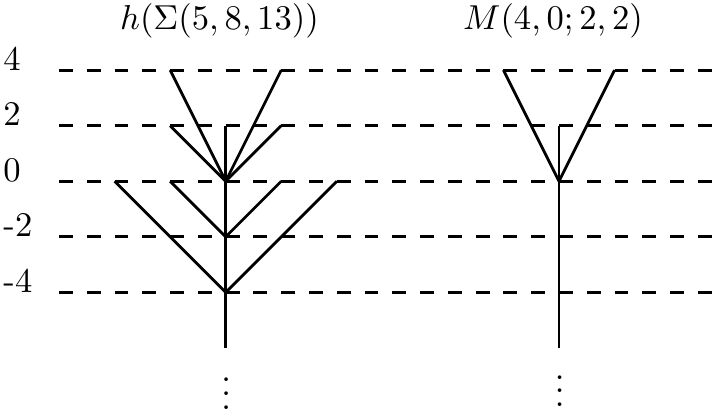}
\caption{Graded root and monotone subroot for $h(\Sigma(5,8,13))$. See also \cite{Stoffregen}.} \label{fig:21}
\end{figure}
\end{example}

\begin{example}\label{ex:ex2}
Let $Y = \Sigma(13,21,34)$. We have drawn part of the graded root associated to $h(Y)$ on the left in Figure~\ref{fig:22}; the monotone subroot is displayed on the right. (The full graded root is significantly more complicated, but we have displayed only the portion which is relevant for computing the local equivalence class.) Applying Theorem~\ref{thm:4.2}, we see that
\begin{align*}
h(\Sigma(13, 21, 34)) &= M(12, 0; 10, 2) \\
&= M(12, 0) + M(10, 2) - M(10, 0) \\
&= (Y_6 + Y_4 - Y_5)[-2].
\end{align*}
\begin{figure}[h!]
\center
\includegraphics[scale=1]{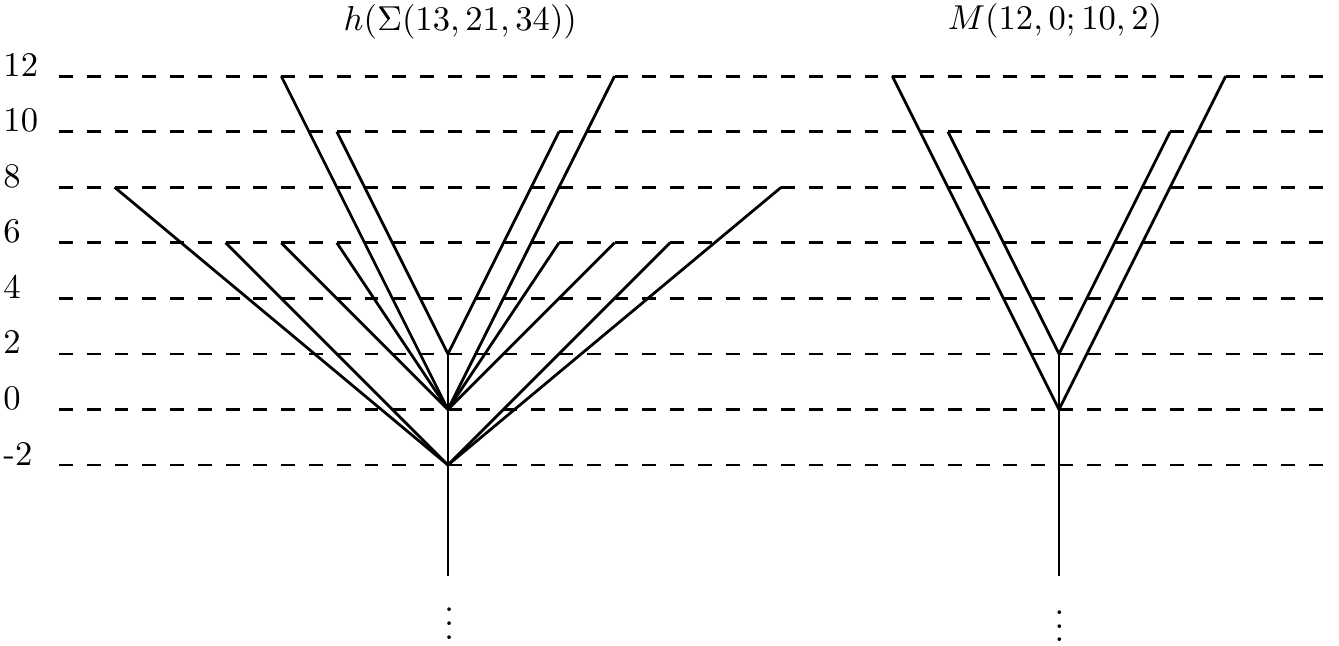}
\caption{Graded root and monotone subroot for $h(\Sigma(13,21,34))$.} \label{fig:22}
\end{figure}
\end{example}

As evidenced by the above examples, most ``small" Brieskorn spheres have a fairly simple decomposition into the $Y_i$. Indeed, several familiar classes of Brieskorn spheres are actually of projective type, as discussed in \cite{Stoffregen}. However, it appears that arbitrarily high complexity can be achieved by considering Brieskorn spheres with sufficiently large parameters; for example, using parameters on the order of $100\sim500$, one can construct examples that decompose into a linear combination of upwards of twenty basis elements. (These, as well as Example~\ref{ex:ex2} above, are due to Duncan McCoy.)

We now turn to the proof of Corollaries~\ref{cor:1.6} and \ref{cor:1.7}:

\begin{proof}[Proof of Corollary~\ref{cor:1.6}]
Let $h(Y, \s)$ be as in the statement of the corollary. Putting $\#_k(Y, \s)$ into the statement of Theorem~\ref{thm:1.5}, observe that for $0 \leq i \leq k$, we then have
\[
P_i = 2i(t_1 - s_1),
\]
and that for $0 \leq i < k$, we have
\[
Q_i = 2i(t_1 - s_1) - 2s_1.
\]
Now assume $t_1 > s_1$. Suppose $k$ is large enough so that
\[
Q_{k-1} = 2(k-1)(t_1 - s_1) - 2s_1 \geq 0.
\]
Since $P_1, \ldots, P_{k-1}$ are all also greater than zero, this means that we have
\[
\min(P_0, P_1, \ldots, P_{k-1}, Q_{k-1}) = P_0 = 0.
\]
Every other term appearing inside the $\max$ in the expression for $\dl$ is certainly bounded above by zero, since $P_0 = 0$. Hence $\dl(\#_k(Y, \s)) = k \cdot d(Y, \s)$.

Now assume $t_1 < s_1$. Suppose $k$ is large enough so that
\[
P_k = 2k(t_1 - s_1) \leq -2s_1.
\]
We claim that every term appearing inside the $\max$ in the expression for $\dl$ is bounded above by $Q_0 = - 2s_1$. To see this, observe that $Q_i \leq Q_0$ for all $0 \leq i < k$. Since every such term in the expression for $\dl$ either contains such a $Q_i$ or contains $P_k$, this implies the bound. Hence $\dl(\#_k(Y, \s)) = k \cdot d(Y, \s) - 2s_1$. Considering $-Y$ instead of $Y$ establishes the theorem for $\du$. 
\end{proof}

\begin{proof}[Proof of Corollary~\ref{cor:1.7}] For simplicity, we first apply an overall grading shift so that $\bar{\mu} = 0$. It then suffices to exhibit distinct linear combinations of basis elements $Y_i$ with the invariants $d, \du,$ and $\dl$. Let 
\begin{align*}
\du - d &= 2M \text{ and}\\
d - \dl &= 2N,
\end{align*}
with the understanding that $M$ and $N$ are not both zero. In order to give some intuition for the proof, we begin by considering the ansatz
\[
\Sigma = Y_a - Y_b - Y_c - n Y_1,
\]
for non-negative integers $a \geq b \geq c \geq 1$ and non-negative parameter $n$. Applying Theorem~\ref{thm:1.5}, we have
\begin{align*}
\dl(\Sigma) &= d(\Sigma) + \max\{ \min(0, -2a), \min(0, 2(b - a)) \} \\
&= d(\Sigma) - 2(a-b).
\end{align*}
Computing $\du$, we have
\begin{align*}
\dl(-\Sigma) &= d(-\Sigma) + \max\{ \min(0, -2b), \min(0, 2(a-b), 2(a-b-c)) \} \\
&= d(-\Sigma) + \min(0, 2(a-b-c)),
\end{align*} 
which shows that $\du(\Sigma) = d(\Sigma) + \max(0, 2(-a + b + c))$. Note that the differences $\du(\Sigma) - d(\Sigma)$ and $d(\Sigma) - \dl(\Sigma)$ are independent of $n$. The system
\begin{align*}
2(-a + b + c) &= 2M \\
2(a-b) &= 2N
\end{align*}
is solved by
\begin{align*}
c &= M + N \\
b &= M + N + k \\
a &= M + 2N + k
\end{align*}
for any $k \geq 0$. (Note that $b + c \geq a$.) Linear combinations $\Sigma$ with these parameters thus yield the desired differences $\du - d$ and $d - \dl$. Moreover, we have 
\[
d(\Sigma) = 2(a - b - c - n) = -2M - 2n. 
\]
Hence if $d \leq -2M$, then there is some fixed $n \geq 0$ for which the one-parameter family
\[
\Sigma_{n,k} = Y_{M+2N+k} - Y_{M+N+k} - Y_{M+N} - n Y_1
\]
has the desired invariants for all $k \geq 0$. In the case that $d > -2M$, we construct the slightly more complicated ansatz
\[
\Sigma_{n,k}' = Y_{M + 2N + 1 + k} - Y_{M + N + 1 + k} - Y_{M + N + 1} + 2Y_1 + nY_1
\]
with $n \geq 0$ and $k \geq 0$. Then
\[
d(\Sigma_{n,k}') = -2M + 2 + 2n.
\]
Applying Theorem~\ref{thm:1.5}, we obtain
\begin{align*}
\dl(\Sigma_{n,k}') = d(\Sigma_{n,k}') + \max\{ &\min(0, -2(M + 2N + 1+ k)), \\
&\min(0, -2N, -2N - 2), \\
&\min(0, -2N, 2M, 2M - 2) \} \\
&\hspace{-3.08cm}= d(\Sigma_{n,k}') - 2N.
\end{align*}
Similarly,
\begin{align*}
\dl(-\Sigma_{n,k}') = d(-\Sigma_{n,k}') + \max\{ &\min(0, -2(M + N + 1+ k)), \\
&\min(0, 2N, -2M - 2), \\
&\min(0, 2N, -2M) \} \\
&\hspace{-3.38cm}= d(-\Sigma_{n,k}') - 2M,
\end{align*}
which shows that $\du(\Sigma_{n, k}') = d(\Sigma_{n, k}') + 2M$. Hence if $d > -2M$, then the appropriate choice of $n \geq 0$ again yields a one-parameter family of examples with the desired $d, \du,$ and $\dl$.
\end{proof}

\end{document}